\documentclass{article}
\usepackage[utf8]{inputenc}
\usepackage{amsmath}
\usepackage{mathtools}
\usepackage{amsfonts}
\usepackage{amsthm}
\usepackage{float}
\usepackage{amssymb,thmtools}
\usepackage{bm}
\usepackage{tikz}
\usepackage{url}
\usetikzlibrary{positioning}
\usetikzlibrary{arrows.meta}
\usepackage{indentfirst}
\usepackage{hyperref}
\usepackage[capitalise]{cleveref}
\usepackage{blkarray}
\usepackage{caption}
\usepackage{subcaption}
\usepackage{enumitem}
\usepackage{graphicx}

\usepackage[doi=false,isbn=false,url=false]{biblatex}
\Crefformat{footnote}{#2\footnotemark[#1]#3}
\newlist{primenumerate}{enumerate}{1}
\setlist[primenumerate,1]{label={\arabic*$'$.}}

\newif\ifptitle
\newif\ifpnumber
\newcounter{para}
\newcommand\ptitle[1]{\par\refstepcounter{para}
{\ifpnumber{\noindent\textcolor{lightgray}{\textbf{\thepara}}\indent}\fi}
{\ifptitle{\textbf{[{#1}]}}\fi}}

\makeatletter
\g@addto@macro\@floatboxreset\centering
\makeatother

\usepackage{microtype}
\usepackage{geometry}
\hypersetup{
	pdfstartview={XYZ null null 1.00}, 
	pdfpagemode=UseNone, 
	colorlinks,
	breaklinks, 
	linkcolor=blue,
	urlcolor=blue, 
	anchorcolor=blue,
	citecolor=blue
}

\definecolor{dark-gray}{gray}{0.35}
\newcommand{\Pl}{Pl{\"u}cker }
\newcommand{\Gr}{\textnormal{Gr}}
\newcommand{\GL}{\textnormal{GL}}

\newcommand{\Fl}{\textnormal{Fl}}

\newcommand{\TP}{\textnormal{TrFl}^{\geq 0}}
\newcommand{\trop}{\overline{\textnormal{trop}}}
\newcommand{\Trop}{\textnormal{Trop}\hspace{0.7mm}}
\newcommand{\Perm}{\textnormal{Perm}}
\newcommand{\supp}{\textnormal{Supp}}

\newcounter{typeop}
\Crefformat{typeop}{#2\textit{#1}#3}

\newcommand{\typeop}[2]{%
  \def\thetypeop{#1}%
  \refstepcounter{typeop}%
  \label{#2}%
  \textit{#1}
}

\newtheorem{definition}{Definition}
\newtheorem{theorem}[definition]{Theorem}
\newtheorem{proposition}[definition]{Proposition}
\newtheorem{cor}[definition]{Corollary}
\newtheorem{example}[definition]{Example}
\newtheorem{lemma}[definition]{Lemma}
\newtheorem{remark}[definition]{Remark}
\newtheorem*{intro}{Theorem}

\newenvironment{thmbis}[1]
  {\renewcommand{\thetheorem}{\ref{#1}$^{\textnormal{trop}}$}%
   \addtocounter{theorem}{-1}%
   \begin{theorem}}
  {\end{theorem}}
  
\newenvironment{propbis}[1]
  {\renewcommand{\theproposition}{\ref{#1}$^{\textnormal{trop}}$}%
   \addtocounter{proposition}{-1}%
   \begin{proposition}}
  {\end{proposition}}

\newenvironment{lemmabis}[1]
  {\renewcommand{\thelemma}{\ref{#1}$^{\textnormal{trop}}$}%
   \addtocounter{lemma}{-1}%
   \begin{lemma}}
  {\end{lemma}}
  
  \newenvironment{defbis}[1]
  {\renewcommand{\thedefinition}{\ref{#1}$^{\textnormal{trop}}$}%
   \addtocounter{definition}{-1}%
   \begin{definition}}
  {\end{definition}}

\numberwithin{definition}{section}
\numberwithin{theorem}{section}
\numberwithin{proposition}{section}
\numberwithin{cor}{section}
\numberwithin{example}{section}
\numberwithin{lemma}{section}
\numberwithin{remark}{section}

\usepackage{enumitem}
\setlist{listparindent=\parindent}

\title{Totally Nonnegative Tropical Flags and the Totally Nonnegative Flag Dressian}
\author{Jonathan Boretsky}
\date{\today}

\begin{document}

\maketitle

\begin{abstract}

 We study the totally nonnegative part of the complete flag variety and of its tropicalization. We show that Lusztig's notion of nonnegative complete flag variety coincides with the set of complete flags which have nonnegative Pl{\"u}cker coordinates. This mirrors the characterization of the totally nonnegative Grassmannian as the set of points in the Grassmannian with nonnegative Pl{\"u}cker. We then study the tropical complete flag variety and complete flag Dressian, which are two tropical versions of the complete flag variety, capturing realizable and abstract flags of tropical linear spaces, respectively. In general, the complete flag Dressian properly contains the tropical complete flag variety. However, we show that the totally nonnegative parts of these spaces coincide.
\end{abstract}
\tableofcontents
\section{Introduction}

We build upon and unite two perspectives on flag varieties: On the one hand, there has been progress in understanding the totally nonnegative parts of flag varieties and how we can characterize them, including \cite{Lus2}\cite{pos}\cite{Wil}\cite{TW}\cite{Lam}\cite{Lus3} and \cite{BK}. On the other hand, there has been interest in the tropicalizations of flag varieties and various ways to understand the resulting tropical flag varieties, including \cite{SS}\cite{Spe}\cite{BLM} and \cite{BEZ}. The Grassmannian is a particularly nice flag variety where these two mathematical notions have been brought together. It has been proven that the \textit{totally nonnegative tropical Grassmannian}, which is the space of realizable positive tropical linear spaces, equals the totally nonnegative part of the \textit{Dressian}, which parameterizes all tropical linear spaces, not just those that are realizable \cite{SW}\cite{ALS}\cite{SW2}. Here, we show that the \textit{totally nonnegative tropical complete flag variety} equals the \textit{totally nonnegative complete flag Dressian}. 

The \textit{real Grassmannian} of $k$-planes in $n$-space, $\Gr_{k,n}$, is the variety where each point corresponds to a $k$-dimensional linear subspace of $\mathbb{R}^n$. A natural generalization of the Grassmannian is the \textit{flag variety} $\textnormal{Fl}_{\bm{r},n}$ of rank $\bm{r}=(r_1,r_2,\ldots, r_k)$ in $n$-space. The points of this space correspond to collections of linear subspaces  $L_1\subsetneq L_2\subsetneq \cdots \subsetneq L_k\subset \mathbb{R}^n$ such that $\dim(L_i)=r_i$. Two notable examples of flag varieties are $\Gr_{k,n}$, of rank $(k)$, and the complete flag variety $\textnormal{Fl}_n$, of rank $(1,2,\ldots,n)$.

In \cite{Lus2}, the totally nonnegative part of a flag variety is defined. A number of authors, among them \cite{Rie}, \cite{TW}, \cite{Lam} and \cite{Lus3}, have proven that the totally nonnegative Grassmannian consists precisely of those points in the Grassmannian for which each \Pl coordinate (see \Cref{TNNflagvariety} for the definition) is nonnegative. We extend this result to the setting of the complete flag variety. A construction based on the parameterization of the totally nonnegative complete flag variety, $\textnormal{Fl}_n^{\geq 0}$, by Marsh and Rietsch \cite{MR} will allow us to understand explicitly the \Pl coordinates of an arbitrary flag $F$ in $\textnormal{Fl}_n^{\geq0}$. These coordinates will be indexed by proper nonempty subsets $I\subset [n]$ and denoted $P_I(F)$. We can now state our first main result, \Cref{nonnegPluckCoords}:

\begin{intro}
The totally nonnegative complete flag variety $\textnormal{Fl}_n^{\geq 0}$ equals the set $\{F\in \textnormal{Fl}_n|\;P_I(F)\geq 0\; \forall \; \emptyset\neq I\subset[n]\}$.
\end{intro}

Around the time that this preprint appeared on the Arxiv, this result also appeared in independent work of Bloch and Karp \cite{BK}, where they show that the totally nonnegative part of the flag variety $\textnormal{Fl}_{\bm{r},n}$ can be characterized as the set of flags with nonnegative \Pl coordinates if and only if $\bm{r}=(a, a+1,\ldots, b)$ consists of consecutive integers. In contrast to the proof of Bloch and Karp, our proof offers insight into the dependencies of the \Pl coordinates on the parameters in the Marsh-Rietsch parameterization of the totally nonnegative complete flag variety. This will prove to be an important feature for proving our second main result.

Tropical geometry is the geometry of the tropical semiring $\mathbb{T}=\mathbb{R}\cup\{\infty\}$ where multiplication is replaced by addition, and addition is replaced by minimization. Thus, if we tropicalize a polynomial, we get a minimization over a collection of sums of variables. We say a point is a solution of a tropical polynomial if that minimum is achieved at least twice. In this setting, varieties become polyhedral cell complexes, making them amenable to combinatorial study. For precise definitions, see \Cref{tropicaldefs}. Tropical geometry has indeed proven a useful tool in algebraic combinatorics, as in \cite{Mik}\cite{Kap}, and most notably for our purposes, \cite{SW} and \cite{SW2}. 

For $k\leq n$, $\Gr_{k,n}$ is an algebraic variety cut out by the \emph{\Pl relations}, which generate an ideal called the \textit{\Pl ideal}. The set of points satisfying the tropicalizations of all the \Pl relations is called the \textit{Dressian} and is the parameter space of abstract tropical linear spaces \cite{Spe}. The set of points satisfying the tropicalizations of all polynomials in the \Pl ideal is called the \textit{tropical Grassmannian} and is the parameter space of realizable tropical linear spaces \cite{SS}\cite{HKT}. In general, the Dressian properly contains the tropical Grassmannian (see, for instance, \cite{HJM}). However, independently in \cite{SW2} and in \cite{ALS}, it is shown that if we restrict to positive solutions, for an appropriate notion of positivity, the situation is simpler: the \textit{positive (resp. totally nonnegative) Dressian} equals the \textit{positive (resp. totally nonnegative) tropical Grassmannian}. More explicitly, this means that a common positive solution to the tropicalizations of all the \Pl relations is also a positive solution to the tropicalization of any polynomial in the ideal generated by the \Pl relations. We generalize this fact to the setting of the complete flag variety:

\begin{intro}
The totally nonnegative tropical complete flag variety, $\textnormal{TrFl}_n^{\geq{0}}$, equals the totally nonnegative complete flag Dressian, $\textnormal{FlDr}_n^{\geq 0}$.
\end{intro}

Here, the tropical complete flag variety is the set of points satisfying the tropicalizations of all polynomials in the ideal cutting out the complete flag variety, which is called the \textit{incidence \Pl ideal}. The complete flag Dressian is the set of points satisfying the tropicalizations  of a particular set of generators for the incidence \Pl ideal, called the \textit{incidence \Pl relations}, which generalize the \Pl relations. These spaces can be thought of as parameterizing abstract flags of tropical linear spaces and realizable flags of tropical linear spaces, respectively \cite{BEZ}.  This result implies the equivalence of the totally positive complete flag variety and the totally positive complete flag Dressian, proven independently by \cite{JLLO}.

While the results highlighted in this introduction may feel fairly different from one another, their proofs are virtually identical. In fact, in this paper, we will present two parallel stories, one in the real world and the second in the tropical world. Many of our results and definitions about the complete flag variety will immediately be followed by an analogous statement about the complete flag Dressian, which will be indicated by a superscript ``trop" in the numbering of the statement.

Let $\mathfrak{S}_n$ be the symmetric group on $[n]$. Define the \textit{permutahedron} $\Perm_n$ to be the convex hull of the points $\{(w(1),w(2),\ldots, w(n))\mid w\in \mathfrak{S}_n\}$ in $\mathbb{R}^n$. The symmetric group carries a partial order called the (strong) Bruhat order, which will be defined in 
\Cref{sectparam}. For $u,v\in \mathfrak{S}_n$ with $u\leq v$ in Bruhat order, define the \textit{Bruhat interval polytope} $P_{u,v}$ to be the convex hull of the points $\{(w(1),w(2),\ldots, w(n))\mid w\in \mathfrak{S}_n\textnormal{ and }u\leq w\leq v \textnormal{ in Bruhat order}\}$ in $\mathbb{R}^n$ \cite{KW}. Note that Bruhat interval polytopes generalize $\Perm_n=P_{id,w_0}$, where $id$ is the identity permutation, which is minimal in Bruhat order, and $w_0$ is the unique maximal element in Bruhat order. In \cite{JLLO}, the authors showed that a point of the positive complete flag Dressian can be thought of as a height function on the permutahedron, which induces a coherent subdivision of the permutahedron into Bruhat interval polytopes. 

In this paper, we study the totally nonnegative complete flag Dressian, which extends the positive complete flag Dressian by allowing points to have coordinates whose values are $\infty$. Each such point can be seen as a height function on some Bruhat interval polytope $P_{u,v}$. In our follow up paper \cite{BEW}, we will study coherent subdivisions of Bruhat interval polytopes arising from the points of the totally nonnegative flag Dressian, and also generalize some of the results of this paper to the totally nonnegative flag varieties $\textnormal{Fl}_{\bm{r},n}^{\geq 0}$ where $\bm{r}=(a,a+1,\cdots, b)$ consists of consecutive integers, that is, the varieties of flags of subspaces of consecutive ranks.

The structure of the paper is as follows: In \Cref{sect:flagvarieties}, we introduce the totally nonnegative complete flag variety and its tropicalization. In \Cref{sect:comb}, we present the two key combinatorial ideas that underlie many of our proofs. In \Cref{sectparam}, we delve deeper into the combinatorics of the totally nonnegative complete flag variety and explore the details of a parameterization of it. In \Cref{secttopcell}, we present our main results in the context of the totally positive flag variety, where all \Pl coordinates are strictly positive. Our results are significantly easier to prove in this setting. As such, the totally positive flag variety offers an illustrative application of our proof method which avoids many of the technical details. In \Cref{sectgeneral}, we extend the methods from the previous section to show our main results in full generality. We present the proof of \Cref{extremunique} in its own section, \Cref{sec:BigProof}, due to its technicality and length. 

\ptitle{Acknowledgements}\textbf{Acknowledgements:} The author thanks their supervisor Lauren Williams for introducing them to the tropical flag variety and for many helpful conversations as this paper developed. The author also thanks Chris Eur, Mario Sanchez and Melissa Sherman-Bennett for helpful conversations, examples and references. The author greatly appreciates the thorough and thoughtful recommendations provided by an anonymous reviewer. The author was supported by the Natural Sciences and Engineering Research Council of Canada (NSERC). L'auteur a été financé par le Conseil de recherches en sciences naturelles et en génie du Canada (CRSNG), [Ref. no. 557353-2021].

\section{The Flag Variety}\label{sect:flagvarieties}

In this section, we will introduce the complete flag variety and its nonnegative part. We will then discuss tropicalizations of the flag variety and their nonnegative parts. These spaces will be the setting for our main results. 

 \subsection{The Totally nonnegative Complete Flag Variety}\label{TNNflagvariety}

 \begin{definition}
 The \textbf{complete flag variety} $\textnormal{Fl}_n$ is the set of all \textbf{complete flags} in $\mathbb{R}^n$, which are collections $(V_i)_{i=0}^n$ of linear subspaces satisfying $\{0\}=V_0\subsetneq V_1 \subsetneq \cdots \subsetneq V_n=\mathbb{R}^n$.  
 \end{definition}

 $\textnormal{Fl}_n$ is a multi-projective variety. We can represent a flag $(V_i)_{i=0}^n$ by a full rank $n\times n$ matrix $M$ such that $V_i$ equals the span of the topmost $i$ rows of $M$. Let $\GL(n,\mathbb{R})$ be the group of invertible $n\times n$ real matrices. Let $B_-$ be the \textit{Borel subgroup} of $\GL(n,\mathbb{R})$ consisting of lower triangular matrices. One can check that two matrices $M,M'\in \GL(n,\mathbb{R})$ represent the same flag if and only if they are related by left multiplication by some $B\in B_{-}$.  Thus, we can think of the complete flag variety as $\textnormal{Fl}_n=\{B_-u|u\in GL(n,\mathbb{R})\}$, where a flag in $\textnormal{Fl}_n$ represented by a matrix $u$ is identified with the set $B_-u$.
 
 For $\emptyset\neq I\subset[n]=\{1,\ldots,n\}$ and $M$ an $n\times n$ matrix, the \textit{\Pl coordinate} (or, alternatively, \textit{flag minor}) $P_I(M)$ is the determinant of the submatrix of $M$ in rows $\{1,2,\ldots,|I|\}$ and columns $I$. To any flag $F$, associate the collection of \Pl coordinates $P(F)=(P_I(F))_{\emptyset \neq I\subset[n]}$, defined to be the \Pl coordinates of any matrix representative of that flag.  By {\cite[Proposition 14.2]{MS}}, this is an embedding of $\textnormal{Fl}_n$ in $\mathbb{R}\mathbb{P}^{{\binom{n}{1}}-1}\times\cdots \times \mathbb{R}\mathbb{P}^{\binom{n}{n-1}-1}$.  The \Pl coordinates of flags in $\textnormal{Fl}_n$ are cut out by multi-homogeneous polynomials which generalize the usual (Grassmann) \Pl relations, as shown in the following statements. Note that we will often use shorthand notation such as $(S\setminus ab)\cup cd$ in place of $\left(S\setminus \{a,b\}\right)\cup \{c,d\}$.

 \ptitle{incidence \Pl relations}\begin{definition}[\cite{Ful}]
  Consider $\mathbb{R}\mathbb{P}^{\binom{n}{1}-1}\times\cdots \times \mathbb{R}\mathbb{P}^{\binom{n}{n-1}-1}$, with coordinates indexed by proper nonempty subsets of $[n]$. For $1\leq r\leq s\leq n$, the \textbf{incidence \Pl relations} for indices of size $r$ and $s$ are 
  \begin{equation}\label{pluckerrelation}
      \mathfrak{P}_{r,s;n}=\left\{\sum_{j\in J\setminus I}sign(j,I,J)P_{I\cup j}P_{J\setminus j }\left| I \in {\binom{[n]}{r-1}},\: J\in {\binom{[n]}{s+1}} \right.\right\},
  \end{equation}
 where $sign(j,I,J)=(-1)^{|\{k\in J|k<j\}|+|\{i\in I|j<i\}|}$.
 
The full set of incidence \Pl relations is $\mathfrak{P}_{IP;n}=\bigcup_{1\leq r\leq s\leq n} \mathfrak{P}_{r,s;n}$. The ideal generated by $\mathfrak{P}_{IP;n}$, denoted $I_{{IP;n}}$, is called the \textbf{incidence \Pl ideal}.

 \end{definition}
 
\begin{remark} \label{rmk:grassmannrelations}
 When $r=s$, we recover the (Grassmann) \textit{\Pl relations}. When we want to emphasize that we are interested in the incidence \Pl relations for which $r\neq s$, we will call them \textit{incidence relations}. 

\end{remark} 

\begin{remark} \label{rem:threetermrelations}
The shortest relations appearing in $\mathfrak{P}_{IP;n}$ consist of three terms. These three-term relations will be of fundamental importance in this paper, so we take the time here to write them out explicitly. When $s=r+1$, we have three-term incidence relations: Let $S\subset[n]$ be of size $r-1$ and fix $i,j,k\notin S$ with $i<j<k$. Then, letting $I=S$ and $J=S\cup ijk$, we obtain $P_{S\cup j}P_{S\cup ik}=P_{S\cup i}P_{S\cup jk}+P_{S\cup k}P_{S\cup ij}$. When $s=r$, we have three-term \Pl relations: Let $S\subset[n]$ be of size $r-2$ and fix $i,j,k,l\notin S$ with $i<j<k<l$. Then, letting $I=S\cup i$ and $J=S\cup jkl$, we obtain $P_{S\cup ik}P_{S\cup jl}=P_{S\cup il}P_{S\cup jk}+P_{S\cup ij}P_{S\cup kl}$. 
    
\end{remark}

\begin{proposition}
[{\cite[Section 9, Proposition 1 and the following discussion]{Ful}}]\label{pluckerideal}
 Let $P\in\mathbb{R}\mathbb{P}^{\binom{n}{1}-1}\times\cdots \times \mathbb{R}\mathbb{P}^{\binom{n}{n-1}-1}$. Then $P=P(F)$ for some $F\in \textnormal{Fl}_n$ if and only if $P$ satisfies the incidence \Pl relations $\mathfrak{P}_{IP;n}$. 
 \end{proposition}

In particular, this means the incidence \Pl relations are precisely the relations between the topmost minors of a generic full rank matrix.

\begin{definition}
    A real matrix is \textbf{totally positive} if all of its minors are positive. We denote the set of totally positive $n\times n$ matrices by $\GL_n^{>0}$.
\end{definition}

Lusztig \cite{Lut} defined the non-negative flag variety to be the closure of the set of flags with a totally positive representative: 

\begin{definition}
    The (Lusztig) positive part of the flag variety, denoted $\Fl_n^{>0}$, is the subset of $\Fl_n$ with a representative in $\GL_n^{>0}$. The (Lusztig) nonnegative part of the flag variety, denoted $\Fl_n^{\geq 0}$, is the closure of $\Fl_n^{>0}$ in the Euclidean topology on $\mathbb{R}\mathbb{P}^{\binom{n}{1}-1}\times\cdots \times \mathbb{R}\mathbb{P}^{\binom{n}{n-1}-1}$. 
\end{definition}

In \Cref{sectparam}, we give a cell decomposition and a parameterization of the totally nonnegative complete flag variety and investigate the \Pl coordinates of totally nonnegative complete flags.

\subsection{The Complete Flag Dressian and the Tropical Complete Flag Variety}\label{tropicaldefs}

\ptitle{Section outline} We now introduce some notation and terminology needed for discussing \textit{tropical varieties}, as well as the precise definitions of the \textit{totally positive (resp. nonnegative) tropical complete flag variety} and the \textit{totally positive (resp. nonnegative) complete flag Dressian}.

Let $\bm x=(x_1,\ldots,x_n)\in \mathbb{R}^n$. For $\bm{b}=(b_1,\ldots,b_n)\in\mathbb{Z}^n$, we write $\bm {x}^{\bm {b}}=x_1^{{b}_1}\cdots x_n^{{b}_n}$. We now define the \textit{tropicalization of a polynomial $p$} which is, roughly speaking, the expression obtained by replacing addition by minimization and multiplication by addition. \Cref{tropsol} and \Cref{tropideal} are modified from \cite{BEZ}. \Cref{otherdef} will offer further motivation for some of the terminology employed. 

\begin{definition}\label{tropsol}
Let $p=\sum_i a_i\bm{x}^{\bm b_i}$ be a Laurent polynomial, where each $a_i\in\mathbb{R}$ and each $\bm b_i\in \mathbb{Z}^n$. We define the \textbf{tropicalization} of $p$, $\Trop(p):\mathbb{R}^n\rightarrow\mathbb{R}$, by

\begin{equation} 
\Trop(p)=\min_i \left\{\bm{x}\cdot \bm {b_i}\right\}= \min_i \left\{x_1(b_i)_1+\cdots+ x_n(b_i)_n\right\},
\end{equation}
if $p$ is non-zero. We define $\Trop(0)=\infty$. For an $m-$tuple of polynomials $p=(p_1,\cdots, p_m):\mathbb{R}^n\rightarrow\mathbb{R}^m$, we define $\Trop(p):\mathbb{R}^n\rightarrow \mathbb{R}^m$ componentwise.
\end{definition}

\begin{definition}
Let $\mathbb{T}=\mathbb{R}\cup \infty$. For a polynomial $p=\sum_i a_i\bm{x}^{\bm{b}_i}:\mathbb{R}^n\rightarrow\mathbb{R}$, we say that a point $\bm{x}\in\mathbb{T}^n$ is a \textbf{solution of the tropicalization of $p$} if $\min_i \left\{\bm{x}\cdot \bm {b_i}\right\}$ is achieved at least twice. We say that a point in $\mathbb{T}^n$ is a \textbf{nonnegative solution of the tropicalization of $p$} if, additionally, at least one of the minima comes from a term of $p$ with a positive real coefficient, and at least one of the minima comes from a term of $p$ with a negative real coefficient. Equivalently, if we rewrite $p=0$ in the form $\sum_j c_j\bm{x}^{\bm b_j}=\sum_i d_i\bm{x}^{\bm b_i}$ with all $c_j$ and $d_i$ positive, then we want at least one minimum to occur in a term coming from each side of the equality. We say that a nonnegative solution of the tropicalization of $p$ is \textbf{positive} if $x_i\neq \infty$ for all $i\in[n]$.
\end{definition}

\begin{remark}
    Since tropical addition is minimization, $\infty$ can be described as the tropical zero. This justifies the relationship between nonnegative and positive tropical solutions to a polynomial in the previous definition: a nonnegative solution is positive if none of its coordinates equal the tropical zero. 
\end{remark}

\begin{remark} In this paper, we will only consider the tropicalization of polynomials with real coefficients. We briefly outline the construction for more general coefficients, which can be found in full in \cite{EKL}. If $K$ is a field with a valuation $\nu:K\rightarrow\mathbb{T}$, we can define the tropicalization of a polynomial with coefficients $a_i\in K$. We let $\Trop(\sum_i a_i\bm{x}^{\bm{b}_i})=\min_i \left\{\nu(a_i)+\bm{x}\cdot \bm {b_i}\right\}$. A tropical solution is then still defined to be a point in $\mathbb{T}^n$ where the minimum is achieved at least twice. For instance, we recover real-representables matroid by considering $K=\mathbb{R}$ with the trivial valuation. 

\end{remark}

\ptitle{Projective Tropical spaces} The tropical objects we are interested in will live in \textit{projective tropical spaces}, which are spaces that interact nicely with homogeneous polynomials.

\begin{definition}
 \textbf{Projective tropical space}, denoted $\mathbb{T}\mathbb{P}^n$, is given by $\left(\mathbb{T}^{n+1}\setminus (\infty,\cdots \infty)\right)/\sim$, where $\bm x \sim \bm y$ if there exists $c\in \mathbb{R}$ such that $x_i=y_i+c$ for all $i\in [n]$. 
\end{definition}

\begin{proposition}
If $p$ is a homogeneous Laurent polynomial, then $\bm x$ is a (nonnegative) solution of $\Trop(p)$ if and only if $\bm y$ is a (nonnegative) solution of $\Trop (p)$ for all $\bm y \sim \bm x$. 
\end{proposition}

\begin{proof}
    Let \[p=\sum_{\substack{{\bm b}=(b_1,\ldots, b_n)
    \in \mathbb{Z}^n\\b_1+\cdots+ b_n=d}} a_{\bm b}{\bm x}^{\bm b}\]

    \noindent be a homogeneous polynomial of degree $d$. 
    Then,
    \[\Trop(p)=\min_{\substack{{\bm b}=(b_1,\ldots, b_n)
    \in \mathbb{Z}^n\\b_1+\cdots+ b_n=d}}\left\{b_1x_1+\cdots +b_nx_n\right\}.\]
    
    If $y_i=x_i+c$ for all $i\in[n]$, then $b_1y_1+\cdots +b_ny_n=(b_1x_1+\cdots+ b_nx_n)+cd$. Thus, all terms of $\Trop(p)$ get shifted by $cd$. While this shifts the value of the minimum by $cd$, the $\bm b$ which actually achieve that minimum are unchanged, and so $\bm x$ is a (nonnegative) solution of $\Trop(p)$ if and only if $\bm y$ is a (nonnegative) solution of $\Trop(p)$.
\end{proof}

\ptitle{Example of tropical solutions} \begin{example}
 Consider the homogeneous polynomial $p(x,y)=x^2-xz+y^2$. Then $\Trop(p)$ is given by $\min\{2x, x+z, 2y\}$. The point $(1,1,2)$ is a solution to $\Trop(p)$ since $2x=2y=2\leq x+z=3$. However, this is not a positive solution to $\Trop(p)$ since the minima, $2x$ and $2y$, originate from the terms $x^2$ and $y^2$ of $p$, which have coefficients of the same sign. Since $p$ is homogeneous and $(3,3,4)\sim (1,1,2)$, we know that $(3,3,4)$ is also a solution of $\Trop(p)$. 
\end{example}

 \ptitle{Definition of tropical prevariety}

\begin{definition}\label{tropideal}
 Given a set of multi-homogeneous polynomials $\mathcal{P}$, each of which is homogeneous with respect to sets of variables of sizes $\{n_i\}_{i=1}^t$, and the ideal $I=\langle\mathcal{P}\rangle$ which they generate, we define the following sets in $\mathbb{T}\mathbb{P}^{n_1-1}\times\cdots\times \mathbb{T}\mathbb{P}^{n_t-1}$:
 
 \begin{itemize}
     \item The \textbf{tropical prevariety}, $\trop(\mathcal{P})$, and the \textbf{tropical variety},  $\trop(I)$, are the sets of simultaneous solutions of the tropicalizations of all the polynomials in $\mathcal{P}$ and in $I$, respectively.
     
     \item The \textbf{nonnegative tropical prevariety}, $\trop^{\geq 0}(\mathcal P)$, and the \textbf{nonnegative tropical variety}, $\trop^{\geq 0}(I)$, are the sets of simultaneous nonnegative solutions of the tropicalizations of all the polynomials in $\mathcal{P}$ and in $I$, respectively.
     \item The \textbf{positive tropical prevariety}, $\trop^{>0}(\mathcal P)$, and the \textbf{positive tropical variety}, $\trop^{>0}(I)$, are the sets of simultaneous positive solutions of the tropicalizations of all the polynomials in $\mathcal{P}$ and in $I$, respectively. Equivalently, $\trop^{>0}(\mathcal P)=\trop^{\geq 0}(\mathcal P)\cap\mathbb{R}^n$ and $\trop^{>0}(I)=\trop^{\geq 0}(I)\cap\mathbb{R}^n$.
 \end{itemize}
 
\end{definition}

Solutions of tropicalizations of polynomials can be described in a different way, which clarifies the term ``positive solution" and the choice of notation. Let $\mathcal{C}=\bigcup_{n=1}^\infty\mathbb{C}((t^{1/n}))$ be the field of \textit{Puiseux series} over $\mathbb{C}$. A Puiseux series $p(t)\in \mathcal{C}$ has a term with a lowest exponent, say $at^{u}$ with $a\in \mathbb{C}^*$ and $u\in\mathbb{Q}$. We define $\textnormal{val}(p(t))=u$. Further, we will say $p(t)\in\mathcal{R}^+$ if $a\in \mathbb{R}_{>0}$. 

 Given an ideal $I\trianglelefteq {\mathcal{C}}[x_1,\cdots x_n]$, let $V(I)\subseteq \mathcal{C}^n$ be the variety where all polynomials in $I$ vanish. Thinking of the semifield $\mathcal{R}^+$ in the field $\mathcal{C}$ as being analogous to the semifield $\mathbb{R}_{>0}$ in the field $\mathbb{C}$, we define the positive part of this variety to be $V^+(I)=V(I)\cap \mathcal{({R}^+)}^n$. 
 
\begin{proposition}[{\cite[Proposition 3.7]{Pay}} and {\cite[Proposition 2.2]{SW}}]\label{otherdef}

Let $I$ be an ideal of $\mathcal{C}[x_1,\ldots, x_n]$. Then $\trop(I)\cap\mathbb{R}^{n}=\overline{\textnormal{val}(V(I))}$ and $\trop^{>0}(I)=\overline{\textnormal{val}(V^+(I))}$, where $\overline{\textnormal{val}(V(I))}$ and $\overline{\textnormal{val}(V^+(I))}$ are the closures of $\textnormal{val}(V(I))$ and $\textnormal{val}(V^+(I))$, respectively.

\end{proposition}

\ptitle{The tropical flag variety and the Flag Dressian} We now focus in on the tropical spaces we will study in the rest of this paper. We define two tropical analogues of $\textnormal{Fl}_n$ along with their totally nonnegative and totally positive parts. 

\begin{definition} We define spaces relating to the \textbf{complete flag Dressian} as follows:
\begin{itemize}
    \item The \textbf{complete flag Dressian} is $\textnormal{FlDr}_n=\trop(\mathfrak{P}_{IP;n})$, the tropical prevariety defined by the incidence \Pl relations.
    \item The \textbf{totally nonnegative complete flag Dressian} is $\textnormal{FlDr}_n^{\geq 0}=\trop^{\geq 0}(\mathfrak{P}_{IP;n})$.
    \item The \textbf{totally positive complete flag Dressian}  is $\textnormal{FlDr}_n^{> 0}=\trop^{>0}(\mathfrak{P}_{IP;n})$.
\end{itemize}

\end{definition}

\begin{definition} We define spaces relating to the \textbf{tropical complete flag variety} as follows:
\begin{itemize}
    \item The \textbf{tropical complete flag variety} is  $\textnormal{TrFl}_n=\trop(I_{IP;n})$, the tropical variety defined by the entire incidence \Pl ideal.
    \item The \textbf{totally nonnegative tropical complete flag variety} is  $\textnormal{TrFl}_n^{\geq 0}=\trop^{\geq 0}(I_{IP;n})$. 

    \item The \textbf{totally positive tropical complete flag variety} is
 $\textnormal{TrFl}_n^{>0}=\trop^{> 0}(I_{IP;n})$.
\end{itemize}

\end{definition}

One can think of $\textnormal{FlDr}_n$ and $\textnormal{TrFl}_n$ as parameterizing abstract flags of tropical linear spaces and realizable flags of tropical linear spaces, respectively \cite{BEZ}. Expanding on this, one can think of $\textnormal{FlDr}_n^{\geq 0}$ and $\textnormal{TrFl}_n^{\geq0}$ as parameterizing abstract flags of positive tropical linear spaces and positively realizable flags of tropical linear spaces, respectively. 

 Our main result, \Cref{tropnonnegPluckcoords}, will show that $\TP_n$ and $\textnormal{FlDr}_n^{\geq 0}$ coincide. Note that, a priori, a point in $\textnormal{TrFl}_n^{\geq 0}$ satisfies more relations than a point in $\textnormal{FlDr}_n^{\geq 0}$. In fact, it is shown in {\cite[Example 5.2.4]{BEZ}} that for $n\geq 6$, $\textnormal{FlDr}_n$ properly contains $\textnormal{TrFl}_n$.   

\section{Combinatorial Background}\label{sect:comb}

In this section, we discuss two important combinatorial ideas. First, we give a natural extension of positroids to the setting of the complete flag variety, called \textit{flag positroids}. Then, we discuss the useful Lindstr\"om-Gessel-Viennot lemma \cite{Lindstrom,GV}, which relates matrix minors to graph theory. In \Cref{sectparam}, we will define a parameterization of cells in the totally nonnegative complete flag variety, due to Marsh and Rietsch \cite{MR}. These cells are naturally in bijection with flag positroids. We will then use the Lindstr\"om-Gessel-Viennot lemma to give a combinatorial interpretation of this parameterization for each cell.

 \subsection{Flag Positroids}
 
Realizable matroids and positroids relate closely to the Grassmannian and totally nonnegative Grassmannian, respectively. We now introduce flag matroids and flag positroids. We assume basic familiarity with the definitions of matroids and positroids and direct the reader to \cite{Oxl} for background on matroids and to \cite{pos} for background on positroids. 
 \begin{definition} [{\cite[Lemma 8.1.7]{Kun}}]\label{def:flagmatroid}
  A (complete) \textbf{flag matroid} on a ground set $E$ is a sequence of matroids $\bm{\mathcal{M}}=(\mathcal{M}_1,\mathcal{M}_2,\ldots, \mathcal{M}_n)$ on the ground set $E$, with the rank of $\mathcal{M}_i$ equal to $i$, such that for any $j<k$,

  \begin{itemize}
      \item each basis of $\mathcal{M}_j$ is contained in some basis of $\mathcal{M}_k$, and
      \item each basis of $\mathcal{M}_k$ contains some basis of $\mathcal{M}_j$. 
  \end{itemize}
  
  The matroids $\mathcal{M}_i$ are called the \textbf{constituent matroids} of the flag matroid.
     \end{definition}
  As with many matroid theoretic concepts, there are numerous cryptomorphic descriptions of flag matroids, some of which can be found in \cite{Coxeter}. Note that the indices of the non-zero \Pl coordinates of an invertible square matrix are the bases of a flag matroid. 

  The totally nonnegative part of the Grassmannian coincides with the subset of the Grassmannian with all nonnegative \Pl coordinates. This suggests two natural generalizations of positroids to the flag setting, presented below. We will show in \Cref{synthetic} that they actually coincide.
  
  \begin{definition}
  \label{flagpositroid} A flag matroid on $[n]$ is \textbf{realizable} if its bases are the non-zero \Pl coordinates of $P(F)$ for some $F\in \textnormal{Fl}_n$. Similarly, a \textbf{realizable flag positroid} on $[n]$ is a flag matroid whose bases are the non-zero \Pl coordinates of a flag $F\in \textnormal{Fl}_n^{\geq 0}$.
  \end{definition}

  \begin{definition}
   A \textbf{synthetic flag positroid} on $[n]$ is a flag matroid on $[n]$ whose bases are the non-zero \Pl coordinates of a flag $F\in \textnormal{Fl}_n$ such that $P_I(F)\geq 0$ for all $\emptyset\neq I\subset [n]$. 
  \end{definition}
  
  \begin{remark}\label{rmk:Pluckersprojective} The above definition is slightly imprecise. Since the \Pl coordinates are multi-projective, we should really insist that for each $k\in [n]$, $P_I$ has a fixed sign for all $I$ with $|I|=k$. However, throughout this paper, we will use the projective degrees of freedom to make our \Pl coordinates nonnegative.
\end{remark}

   Both of the above definitions define stronger conditions than just being a flag of matroids in which each constituent matroid is a positroid. For instance, the definition of a synthetic flag positroid requires each of the constituent positroids to be simultaneously realizable, with nonnegative \Pl coordinates, by a single matrix. By contrast, there exist flags of positroids which are not realizable in that way.
  
  \begin{example}
  Consider the flag matroid $\bm{\mathcal{M}}=(\mathcal{M}_1,\mathcal{M}_2,\mathcal{M}_3)$ on the ground set $[3]$ whose constituent matroids have bases $\{1,3\},$ $\{12,13,23\}$ and $\{123\}$, respectively. Each of these is a positroid. Any three by three matrix with nonnegative \Pl coordinates which represents $\bm{\mathcal{M}}$ must have the first row 
 
 \begin{equation} 
  \begin{pmatrix}
  a & 0 & b
  \end{pmatrix}
  \end{equation}
  with $P_1=a>0$ and $P_3=b>0$. 
  
 We can thus write the full matrix as 
 
  \begin{equation} 
  \begin{pmatrix}
  a & 0 & b\\
  * & c & * \\
  * & * & *
  \end{pmatrix},
  \end{equation}
 where the $*$ entries are not of interest to us. Note that there is no value of $c$ which allows both $P_{12}>0$ and $P_{23}>0$. Thus, $\bm{\mathcal{M}}$ is a flag of positroids which cannot be realized by a single matrix with all nonnegative \Pl coordinates. As a result, $\bm{\mathcal{M}}$ is not a synthetic flag positroid.  
  \end{example}

 \subsection{Lindstr{\"o}m-Gessel-Viennot Construction}\label{LGVconst}

\ptitle{Motivation for LGV construction} The \textit{Lindstr{\"o}m-Gessel-Viennot (LGV) construction} relates matrix minors with certain path collections in a weighted directed graph (digraph). We will use it to study the coordinates of flags in $\textnormal{Fl}_n^{\geq 0}$.

\ptitle{Basic terminology} The construction works as follows: Let $G$ be a finite cycle-free digraph with an edge weight $\omega_e$ associated to each directed edge $e$. Within $G$, fix any collection of vertices $A=\{a_1,\ldots, a_n\}$ to be called \textit{sources} and any collection of vertices $B=\{b_1,\ldots, b_n\}$, disjoint from $A$, to be called \textit{sinks}. The \textit{weight} of a directed path $p$ from $a_i$ to $b_j$ is defined to be the product of the weights of the edges in that path. It is denoted $\omega(p)$. We define  $\omega(i,j)\coloneqq \sum_{p:a_i\rightarrow b_j}\omega(p)$ to be the sum of the weights of all paths from $a_i$ to $b_j$. 

\begin{definition}
 Let $|A|=|B|=n$. Then a \textbf{non-intersecting path collection} from $A$ to $B$ is a path collection from $a_i$ to $b_{\sigma(P)(i)}$, where $\sigma(P)$ is some permutation of $[n]$, such that no two paths have a common vertex. Such a collection is denoted $P=(p_1,\ldots, p_n)$ where $p_i$ is the path originating from $a_i$. We define the \textbf{weight of a path collection} $\omega(P) \coloneqq \prod_{i=1}^n \omega(p_i)$ to be the product of the weights of its constituent paths.

\end{definition}

\begin{definition}
    For $I,J\subset [n]$ and $\sigma:I\rightarrow J$ a bijection between $I$ and $J$, we define $sgn(\sigma')\coloneqq(-1)^t$, where $t$ is the number of inversions of $\sigma$, that is, the number of pairs $i<j\in I$ such that $\sigma(j)<\sigma(i)$. When $\sigma$ is a permutation of $[n]$, this recovers the usual notion of the sign of a permutation. 
\end{definition}

We are now ready to state the LGV construction.

\begin{theorem}[\cite{Lindstrom,GV}]\label{LGV}
Let $G$ be a finite acyclic weighted directed graph. Choose a source set $A$ and a sink set $B$ disjoint from $A$, each of size n. Consider the matrix 

\begin{equation*}
    N=\left(\omega (i,j)\right)_{i,j}.
\end{equation*}
Then 
\begin{equation*}\label{LGVe}
    \det(N)=\sum_{P:A\rightarrow B}{\rm sgn}(\sigma(P))\omega(P),
\end{equation*}
where the sum is over all non-intersecting path collections $P$ from $A$ to $B$. 
\end{theorem}

\begin{cor}\label{cor:LGVcor}
In the setting of \Cref{LGV}, let $I\subseteq A$ and $J\subseteq B$ with $|I|=|J|$. Let $\hat{I}$ and $\hat{J}$ be the sets of indices of the corresponding sources in $A$ and sinks in $B$, respectively. Let $N_{I,J}$ denote the submatrix of $N$ consisting of rows $\hat{I}$ and columns $\hat{J}$. For a non-intersecting path collection $P=(P_1,\ldots, P_n)$ from $I$ to $J$, let $\hat{\sigma}(P):\hat{I}\rightarrow \hat{J}$ be the bijection such that $P_i$ is a path from $a_i$ to $b_{\hat{\sigma}(P)(i)}$. Then,

\begin{equation}\label{LGVcor}
    \det(N_{I,J})=\sum_{P:I\rightarrow J}{\rm sgn}(\hat{\sigma}(P))\omega(P),
\end{equation}
where the sum is over all non-intersecting path collections $P$ from $I$ to $J$.
\end{cor}

In the rest of this paper, we will always label source vertices by primed vertex labels and sink vertices by unprimed vertex labels. Thus, for instance, the entry $N_{1,2}$ will be given by the sum of the weights of paths from $1'$ to $2$.

\begin{example}
Consider the weighted digraph in \Cref{LGVthreegraph} with source set $\{1',2',3'\}$ and sink set $\{1,2,3\}$.
\begin{figure}[H]
 \begin{tikzpicture}[node distance={10.5 mm}, thick, main/.style = {draw, circle,minimum size=2 mm}, 
blank/.style={circle, draw=green!0, fill=green!0, very thin, minimum size=3.5mm},]

\node[main] (1) {$1'$};
\node[main] (2) [above of=1] {$2'$};
\node[main] (3) [above of = 2] {$3'$}; 
\node (blank)[above of = 3]{};
\node[main] (33) [right of = blank] {$3$};
\node[main] (22) [right of = 33] {$2$};
\node[main] (11) [right of = 22] {$1$};
\node (first) [right of = 2]{};
\node (second) [right of = 3]{};
\node (third)[right of = second]{};
\draw[-{Latex[length=3mm]}] (1) -- (11);
\draw[-{Latex[length=3mm]}] (2) -- (22);
\draw[-{Latex[length=3mm]}] (3) -- (33);

\draw[-{Latex[length=3mm]}] (first) -- (second) node [near start, left] {$a$};

\draw[-{Latex[length=3mm]}] (second) -- (33) node [near start, left] {$b$};

\draw[-{Latex[length=3mm]}] (third) -- (22) node [near start, left] {$c$};
\end{tikzpicture} 
\caption{A weighted directed graph. The unmarked diagonal edges have weight $1$.}
\label{LGVthreegraph}
 \end{figure}

We construct the matrix $N$. There is a unique path from $1'$ to $1$, consisting of the diagonal edge, which has weight $1$. Thus, $N_{1,1}=1$. There are two paths from $1'$ to $2$, one using the weighted vertical edge $a$ and the other using the weighted vertical edge $c$. Thus, $N_{1,2}=a+c$. There is a unique path from $1'$ to $3$ using the weighted vertical edges $a$ and $b$. Thus, $N_{1,3}=ab$. Note that there are no paths from $2'$ to $1$, so $N_{2,1}=0$. Continuing in this manner, we obtain 
\begin{equation*}
N=\begin{pmatrix}
    1&a+c&ab\\    0&1&b\\
    0&0&1
\end{pmatrix}.
\end{equation*}

There are two non-intersecting path collections from $\{1',3'\}$ to $\{2,3\}$. The first, $P_1$ uses the weighted vertical edge $a$ and the second, $P_2$ uses the weighted vertical edge $c$. Both these path collections connect $1'$ to $2$ and $3'$ to $3$. Accordingly, $\sigma(P_1)=\sigma(P_2)$ is the bijection between $\{1,3\}$ and $\{2,3\}$ mapping $1$ to $2$ and $3$ to $3$. This bijection has no inversions and so $\rm sgn (\sigma(P_1))=\rm sgn (\sigma(P_2))=1$. Correspondingly, $\det(N_{\{1,3\},\{2,3\}})=a+c$. There is a unique non-intersecting path collection $P_3$ from $\{1',2',3'\}$ to $\{1,2,3\}$, consisting entirely of diagonal edges. This path collection connects $i'$ to $i$ for $i\in [3]$. Accordingly, $\sigma(P_1)$ is the identity permutation of $[3]$. This permutation has no inversions and so $\rm sgn (\sigma(P_3))=1$, which is reflected by the fact that $\det(N)=1$. 

\end{example}

\section{Parametrization of the Totally Nonnegative Complete Flag Variety}\label{sectparam}

\ptitle{Introduction to MR parameterization} As shown by Rietsch \cite{Rie}, $\textnormal{Fl}_n^{\geq 0}$ is a cell complex, whose cells $\mathcal{R}_{v,w}^{>0}$ are indexed by pairs of permutations $v\leq w\in \mathfrak{S}_n$ in a partial order called \textit{Bruhat order} on $\mathfrak{S}_n$. Each such $\mathcal{R}_{v,w}^{>0}$ is given an explicit parameterization in \cite{MR}. We will describe this parameterization here, making some choices that in principle are arbitrary but will be convenient for our purposes, and invite the reader to look at the cited references for full generalities.  

\ptitle{Useful notation} 

\subsection{Bruhat Order and Positive Distinguished Subexpressions} 

Any permutation $w$ in $\mathfrak{S}_n$ has \textit{expressions}, that is, ways to write it as a product of simple reflections $s_i=(i,\; i+1)$. The \textit{length} of $w$, $\ell(w)$, is the fewest number of simple reflections in any expression for $w$. An expression for $w$ consisting of $\ell(w)$ simple reflections is called \textit{reduced}. A \textit{subexpression} of an expression is a choice of some of the simple reflections appearing in that expression. If we view an expression as a sequence of simple reflections, a subexpression is simply a subsequence. We now define Bruhat order.  

\begin{definition}\label{def:expression}
 Let $v,w\in \mathfrak{S}_n$. If there exists a reduced expression $\bm{w}=s_{i_1}s_{i_2}\cdots s_{i_\ell}$ for $w$ which contains a subexpression for $v$, we say $v\leq w$ in \textbf{Bruhat order}. Equivalently, $v\leq w$ in Bruhat order if every expression for $w$ contains a subexpression for $v$ \cite[Corollary 2.2.3]{BB}.
\end{definition}

The second characterization of Bruhat order in \Cref{def:expression} guarantees that if we want to test whether $v\leq w$, it suffices to fix any expression for $w$. Given $v\leq w$ and an expression for $w$, we will be interested in a special choice of subexpression which is called the \textit{positive distinguished subexpression} for $v$. Intuitively, this can be thought of as the leftmost reduced subexpression. Note that our definition here differs from \cite{MR} because we use different conventions. 

\begin{definition}\label{negative} Let $v\leq w\in \mathfrak{S}_n$. Choose a reduced expression $\bm{w}=s_{i_1}s_{i_2}\cdots s_{i_\ell}$ for $w$. Then a reduced subexpression $\bm{v}=s_{i_{j_1}}\cdots s_{i_{j_m}}$ for $v$ in $\bm{w}$ is a \textbf{positive distinguished subexpression} if, whenever $\ell(s_{i_p}s_{i_{j_r}}\cdots s_{i_{j_m}})<\ell(s_{i_{j_r}}\cdots s_{i_{j_m}})$ for $j_{r-1}\leq p< j_{r}$, we have $p=j_{r-1}$. 
\end{definition}

\begin{lemma}\label{existsposdist}
For every $v\leq w\in \mathfrak{S}_n$, and every reduced expression $\bm{w}$ of $w$, there is a unique positive distinguished subexpression for $v$ in $\bm{w}$.
\end{lemma}

\begin{proof}
Given an expression $\bm{w}$ for $w$, define $\bm{w^{-1}}$ to be the expression for $w^{-1}$ obtained by reversing the order of the simple reflections in $\bm{w}$. Observe that $\bm{v}$ is a positive distinguished subexpression in $\bm{w}$ by our convention if and only if $\bm{v^{-1}}$ is a positive distinguished subexpression of $\bm{w^{-1}}$ by the convention of \cite{MR}. Thus, the result follows from {\cite[Lemma 3.5]{MR}}.
\end{proof}

\ptitle{Example of positive distinguished subexpression}
\begin{example}
Let $n=4$. Fix the expression $\bm{w}=s_{i_1}\cdots s_{i_6}=s_1s_2s_3s_1s_2s_1$ for $w$ and let $v=s_1s_2s_1$. There are a number of subexpressions for $v$ in $\bm{w}$. For instance, $v$ can be written as $s_{i_{j_1}}s_{i_{j_2}}s_{i_{j_3}}$ with ${j_1}=1$, $j_2=5$ and $j_3=6$. Note that $\ell(s_{i_2}s_{i_{j_2}}s_{i_{j_3}})=\ell(s_2^2s_1)=\ell(s_1)=1$ is less than $\ell(s_{i_{j_2}}s_{i_{j_3}})=\ell(s_2s_1)=2$, but $j_1\neq 2$. Thus, this is not the positive distinguished subexpression for $v$.

The leftmost subexpression for $v$ is $j_1=1$, $j_2=2$ and $j_3=4$. Indeed, one can verify that this choice satisfies the definition of a positive distinguished subexpression.
\end{example}

 In the remainder of this paper, we will fix a reduced subexpression \begin{equation}\label{eq:w0expression}\bm{w_0}=(s_1\cdots s_{n-1})(s_1\cdots s_{n-2})(\cdots)(s_1s_2)(s_1)\end{equation}

 \noindent for $w_0$. We give a straightforward necessary (and in fact, sufficient!) condition to identify positive distinguished subexpressions of $\bm{w_0}$. 
 
 \begin{definition} \label{def:runs}
    We define the $i^{\textnormal{th}}$ run of $\bm{w_0}$ to be the set of simple reflections in the $i^{\textnormal{th}}$ pair of parentheses of \Cref{eq:w0expression}. Explicitly, the first run is $s_1\cdots s_{n-1}$, the second run is $s_1\cdots s_{n-2}$, and so on, until reaching the $(n-1)^{\textnormal{st}}$ run, which is just $s_1$. 
  \end{definition}

\begin{lemma}\label{posdist}
 
Let $\bm{w}$ be any positive distinguished subexpression in \begin{equation*}\bm{w_0}=(s_1\cdots s_{n-1})(s_1\cdots s_{n-2})(\cdots)(s_1s_2)(s_1).\end{equation*} If, for $k\geq 2$, $\bm{w}$ uses an $s_i$ from the $k^\textnormal{th}$ run of $\bm{w_0}$ then it also uses an $s_{i+1}$ from the $(k-1)^\textnormal{st}$ run of $\bm{w_0}$. 
\end{lemma}

\begin{proof}
We will prove this lemma by induction on $i$. For $i=1$, suppose that $\bm{w}$ uses the $s_1$ in run $k$ but not the $s_2$ in run $k-1$. Then $\bm{w}$ must not use the $s_1$ in run $k-1$, since if it did, $\bm{w}$ would not be reduced. However, we then contradict \Cref{negative} by choosing the $s_{1}$ in run $k-1$ to be the $s_{i_p}$ appearing in \Cref{negative}. 

For the induction step, assume to the contrary that $\bm{w}$ uses the $s_i$ in run $k$ but not the $s_{i+1}$ in run $k-1$. Let $r$ be maximal such that $\bm{w}$ uses the simple reflections $s_{i-r},s_{i-r+1},\cdots s_{i}$, with $r$ possibly $0$. By induction, $\bm{w}$ has the simple reflections $s_{i-r+1}\cdots s_i$ in run $k-1$. Since $\bm{w}$ does not use the $s_{i+1}$ in run $k-1$ or the $s_{i-r-1}$ in run $k$, the reflections in run $k-1$ can be commuted past those in run $k$ until we obtain a reduced expression for $w$ containing the consecutive sequence of simple reflections $\left(s_{i-r+1}\cdots s_i\right)\left(s_{i-r}\cdots s_i\right)$; The simple reflections in the first pair of parentheses are from run $k-1$ and those in the second pair of parentheses are from run $k$. The product $\left(s_{i-r+1}\cdots s_i\right)\left(s_{i-r}\cdots s_i\right)$ can be rewritten as $\left(s_{i-r}s_{i-r+1}\cdots s_i\right)\left(s_{i-r}\cdots s_{i-1}\right)$, which can be verified either by computing the actions of these permutations on integers or using braid moves. Again due to the fact that $\bm{w}$ did not use the $s_{i+1}$ in run $k-1$ or the $s_{i-r-1}$ in run $k$, we can move the simple reflections in the first pair of parentheses into run $k-1$ using commutation moves. See \Cref{ex:runs} for an example of this procedure. From this new expression for $w$, we conclude that $\bm{w}$ must not have an $s_{i-r}$ in run $k-1$, since if it did, our new expression would not be reduced. However, we then contradict \Cref{negative} by choosing the $s_{i-r}$ in run $k-1$ to be the $s_{i_p}$ appearing in \Cref{negative}. 

\end{proof}

\begin{example}\label{ex:runs}
We give two examples of the argument in the induction step of the previous proof.

Consider the expression $\bm{w}=(s_1s_2s_3s_5)(s_1s_2s_3s_4)(s_1s_2)()()$ for a permutation $w\in \mathfrak{S}_6$. This uses $s_1,s_2, s_3, s_5$ from run $1$, $s_1,s_2,s_3,s_4$ from run $2$, $s_1,s_2$ from run $3$, and no simple reflections from runs $4$ or $5$. In particular, $\bm{w}$ uses $s_3$ from run $2$ but not $s_4$ from run $1$. By \Cref{posdist}, this is not a positive distinguished subexpression. To see this, note that the permutation $w$ can be expressed as 

\begin{align*}
    w&=(s_1s_2s_3s_5)(s_1s_2s_3s_4)(s_1s_2)\\
    &= s_1(s_2s_3)s_5(s_1s_2s_3)s_4s_1s_2\\
    &=s_1s_5(s_2s_3)(s_1s_2s_3)s_4s_1s_2\\
    &=s_1s_5(s_1s_2s_3)(s_1s_2)s_4s_1s_2\\
    & = s_1(s_1s_2s_3)s_5(s_1s_2)s_4s_1s_2.
\end{align*}

Note that the last expression has two copies of $s_1$ next to each other, showing that $\bm{w}$ was not reduced. 

Similarly, consider the expression $\bm{w}=(s_2s_3s_5)(s_1s_2s_3s_4)(s_1s_2)()()$ for a permutation $w\in \mathfrak{S}_6$. This uses $s_2, s_3, s_5$ from run $1$, $s_1,s_2,s_3,s_4$ from run $2$, $s_1,s_2$ from run $3$, and no simple reflections from runs $4$ or $5$. Again, $\bm{w}$ uses $s_3$ from run $2$ but not $s_4$ from run $1$ and so is not a positive distinguished subexpression. Indeed, similarly to the previous choice of permutation, $w$ can be expressed as 

\begin{align*}
    w&=(s_2s_3s_5)(s_1s_2s_3s_4)(s_1s_2)\\
    &= (s_1s_2s_3)s_5(s_1s_2)s_4s_1s_2.
\end{align*}

Thus, $s_1w<w$. By \Cref{negative}, for $\bm{w}$ to be positive distinguished, the $s_1$ appearing in the first run of $\bm{w_0}$ should appear in $\bm{w}$, but it does not.
\end{example}

We now give a version of the \textit{tableau criterion}, which is a useful test for when two permutations are comparable in Bruhat order without relying on expressions and subexpressions. To state it, we must first introduce \textit{Gale order} on subsets of $[n]$:

\begin{definition}
  Let $I=\{i_1<\cdots< i_k\}$ and $J=\{j_1<\cdots< j_k\}$ be subsets of $[n]$. Then $I\leq J$ in the \textbf{Gale order} if $i_r\leq j_r$ for every $r\in[k]$.
 \end{definition}

 \begin{proposition}\label{tableaucrit}
     Let $v,w\in \mathfrak{S}_n$. For $k\in[n]$, define $v[k]$ to be $\{v(1),v(2),\ldots,v(k)\}$. Then $v\leq w$ in Bruhat order if and only if $v[k]\leq w[k]$ in Gale order for all $k\in [n]$. 
 \end{proposition}

 \begin{proof}
    This is similar to the condition of \cite[Theorem 2.6.3]{BB}, which says that $v\leq w$ if and only if $v[k]\leq w[k]$ in Gale order for all $k$ such that $v(k+1)<v(k)$. We show explicitly that the two conditions are equivalent. We must prove that if $v[k]\leq w[k]$ in Gale order for all $k$ such that $v(k+1)<v(k)$, then $v[j]\leq w[j]$ in Gale order for all $j\in[n]$. 
    
    Assume towards a contradiction that $v[k]\leq w[k]$ in Gale order for all $k$ such that $v(k+1)<v(k)$ and that there exists $j\in [n]$ such that $v[j]\nleq w[j]$. Let $j$ be minimal with this property. Let $j'$ be the minimal number $j' \geq j$ such that $v(j'+1)<v(j')$ if such $j'$ exists and let $j'=n$ otherwise. By our assumption on $v$ and $w$, if $v(j'+1)<v(j')$, then $v[j']\leq w[j']$. If, instead, $j'=n$, then $v[j']=[n]\leq [n]= w[j']$. We also have $v(j)<v(j+1)<\cdots<v(j')$. Since $v[j]\nleq w[j]$, if $v[j]=\{i_1<\cdots<i_j\}$ and $w[j]=\{l_1<\cdots<l_j\}$, then $i_t>l_t$ for some $t\in[j]$. 
    
    Let $v[j-1]=\{i'_1<\cdots<i'_{j-1}\}$ and $w[j-1]=\{l'_1<\cdots<l'_{j-1}\}$. In going from $w[j-1]$ to $w[j]$, exactly one element is added. Thus, for any $s\in[j]$, either $l_s=l'_s$ or $l_s=l'_{s-1}$. In particular, this means that $l'_{s-1}\leq l_s$. Similarly, whenever $i_s>v(j)$, we have $i_s=i'_{s-1}$. By the minimality of $j$, $v[j-1]\leq w[j-1]$, and so $i'_s\leq l'_s$ for all $s\in [j-1]$. Thus, if $i_s>v(j)$, $i_s=i'_{s-1}\leq l'_{s-1}\leq l_s$. Since $i_t>l_t$, we must have $i_t\leq v(j)$. 
    
    To obtain $v[j']$ from $v[j]$, we add new entries, each larger than $v(j)$. Thus, if $v[j']=\{i''_1<\cdots< i''_{j'}\}$, then $i_t=i_t''$. To construct $w[j']$ from $w[j]$, we add some entries to $w[j]$. It is possible some of these are smaller than $l_t$, so if $w[j']=\{l''_1<\cdots< l''_{k'}\}$, then $l_t\geq l''_t$. Thus, $i''_t=i_t>l_t\geq l''_t$. This contradicts the fact that $v[j']\leq w[j']$. 
    
 \end{proof}

\subsection{The Marsh-Rietsch Parametrization}

For $1\leq k < n$, let $x_k(a)$ be the $n\times n$ matrix which is the identity matrix with an $a$ added in row $k$ of column $k+1$. Explicitly,
\begin{equation*}
    x_k(a)=\begin{blockarray}{ccccccc}
& &  & k & k+1 & & \\
\begin{block}{c(cccccc)}
  & 1 &  &  & & & \\
  & & \ddots &  & &  & \\
  k& &  & 1 & a &  &  \\
  k+1& &  &0  & 1 &  &  \\
  & &  &  &  & \ddots  & \\
  &&&&&&1\\
\end{block}
\end{blockarray} \hspace{5pt},
\end{equation*}
\noindent where unmarked off-diagonal matrix entries are $0$.

For $1\leq k<n$, let $\dot{s}_k$ be the $n\times n$ identity matrix with the $2\times 2$ submatrix in rows $\{k,k+1\}$ and columns $\{k,k+1\}$ replaced by the matrix $\begin{pmatrix}
0&1\\
-1&0
\end{pmatrix}$. Explicitly, 
\begin{equation*}
    \dot{s}_k=\begin{blockarray}{ccccccc}
& &  & k & k+1 & & \\
\begin{block}{c(cccccc)}
  & 1 &  &  & & & \\
  & & \ddots &  & &  & \\
  k& &  & 0 & 1 &  &  \\
  k+1& &  &-1  & 0 &  &  \\
  & &  &  &  & \ddots  & \\
  &&&&&&1\\
\end{block}
\end{blockarray}\hspace{5pt},
\end{equation*}
\noindent where unmarked off-diagonal matrix entries are $0$.

\begin{definition}\label{matrixdecompose}
Fix $v\leq w\in \mathfrak{S}_n$. Fix a vector $\bm{a}\in \mathbb{R}^{\ell(w)-\ell(v)}$. Consider the reduced expression $\bm{w_0}=(s_1s_2\cdots s_{n-1})(s_1s_2\cdots s_{n-2})(\cdots)(s_1s_2)(s_1)$ for $w_0$, the longest permutation in Bruhat order in $\mathfrak{S}_n$\footnote{This choice of expression is arbitrary in the context of the Marsh-Rietsch parameterization, but plays an important role in the rest of this paper.}. Choose the positive distinguished subexpression $\bm{w}$ for $w$ in $\bm{w_0}$, and the positive distinguished subexpression $\bm{v}$ for $v$ in $\bm{w}$, and write them as $\bm{w}= s_{i_1}\cdots s_{i_\ell}$ and $\bm{v}=s_{i_{j_1}}\cdots s_{i_{j_m}}$, respectively. Let $J=\{h\;|\;h=j_t \textnormal{ for some }t\}$, that is, the positions in $\bm{w}$ of simple reflections which are used in $\bm{v}$. Let
\begin{equation*}
    M_{v,w}(\bm{a})\coloneqq M_1\cdots M_\ell,\hspace{3mm} \textnormal{  where  }\hspace{3mm} M_j=\begin{cases}\dot{s}_{i_j}, & j\in J\\ x_{i_j}(a_{j}),&j\notin J\end{cases}\hspace{1mm} .
\end{equation*}
\end{definition}

 \begin{theorem}[Marsh-Rietsch Parametrization, \cite{MR}] \label{marshrietsch}

The totally nonnegative complete flag variety $\textnormal{Fl}_{n}^{\geq 0}$ admits a decomposition into cells $\mathcal{R}_{v,w}^{>0}$ for $v\leq w\in \mathfrak{S}_n$, each of which can be parameterized as 

\begin{equation*}
    \mathcal{R}_{v,w}^{>0}=\left\{M_{v,w}(\bm{a})\left|\bm{a}\in\mathbb{R}_{>0}^{\ell(w)-\ell(v)}\right.\right\}.
\end{equation*}

In particular, each $\mathcal{R}_{v,w}^{> 0}$ is homeomorphic to an open ball and each flag $F\in \textnormal{Fl}_n^{\geq 0}$ is uniquely represented in some unique $\mathcal{R}_{v,w}^{> 0}$. \end{theorem} 

\begin{remark}
    We adopt a slightly different convention than \cite{MR}. The matrices in their $\mathcal{R}_{v,w}^{>0}$ would be the transposes of the matrices in our $\mathcal{R}_{v^{-1},w^{-1}}^{>0}$. This difference is accounted for by the fact that we define positive distinguished subexpressions differently from \cite{MR}, as noted in the proof of \Cref{existsposdist}. This convention makes notation a bit cleaner in the rest of this paper. 
\end{remark}

 \begin{remark}
     We are slightly abusing notation, as we should be describing $\mathcal{R}_{v,w}^{>0}$ as consisting of flags, not matrices. We are implicitly identifying a flag $F\in \mathcal{R}_{v,w}^{>0}$ with the unique matrix of the form $M_{v,w}(\bm{a})$ representing it, and will continue to do so in the rest of this paper.  
 
 \end{remark}

 \begin{definition}\label{parameterization}
 Let $\Phi_{v,w}:\mathbb{R}^{\ell(w)-\ell(v)}_{>0}\rightarrow \mathbb{R}^{\binom{n}{1}+\cdots+\binom{n}{n-1}}$ be the map which sends $\bm{a}$ to the \Pl coordinates of $M_{v,w}(\bm{a})$. This map is a bijection onto $\mathcal{R}^{>0}_{v,w}$, and thus it admits an inverse on $\mathcal{R}^{>0}_{v,w}$, which we denote $\Psi_{v,w}$.
\end{definition}

\begin{example}\label{toymodel}
Let $n=4$, $\bm{w}=s_{i_1}s_{i_2}s_{i_3}s_{i_4}=s_1s_3s_2s_1$ and $v=s_2$. The positive distinguished subexpression for $v$ in $\bm {w}$ is the subexpression $\bm{v}=s_{i_{j_1}}$ where $j_1=3$, so $J=\{3\}$. Thus, $M_1=x_1(a_1)$, $M_2=x_3(a_2)$, $M_3=\dot{s}_2$, $M_4=x_1(a_4)$ and $\bm{a}=(a_1,a_2,a_4)$. The cell $\mathcal{R}_{v,w}^{>0}$ of $\Fl_4^{\geq 0}$ is given by 

\begin{equation*}
   \left\{ M_{v,w}(\bm{a})=M_1M_2M_3M_4 
    =\begin{pmatrix}
    1&a_4&a_1&0\\
    0&0&1&0\\
    0&-1&0&a_2 \\
    0&0&0&1
    \end{pmatrix} \;\middle\vert \; \bm a=(a_1,a_2,a_4)\in \mathbb{R}_{>0}^3 \right\}.
\end{equation*}
 Note that the \Pl coordinates of $M_{v,w}(\bm{a})$ are all nonnegative, as our next result will tell us is true for any cell of $\textnormal{Fl}_n^{\geq 0}$. 
 \end{example}

\begin{lemma}[{\cite[Lemma 3.10]{KW}}]\label{forwarddir}
For any $v\leq w\in \mathfrak{S}_n$ and $\bm{a}\in \mathbb{R}_{>0}^{\ell(w)-\ell(v)}$, the matrices $M_{v,w}(\bm{a})$ have all \Pl coordinates nonnegative. In particular, we can choose nonnegative \Pl coordinates for any $F\in \textnormal{Fl}_{n}^{\geq 0}$.
\end{lemma}

We establish some notation that we will use going forward.

\begin{remark}
    The map $\Phi_{v,w}$ can be extended to $(\mathbb{R}^*)^{\ell(w)-\ell(v)}$, in which case it is a bijection onto the \textit{Deodhar component} $\mathcal{R}_{\mathbf{v},\mathbf{w}}$, where $\mathbf{v}$ is the positive distinguished subexpression for $v$ in the expression $\mathbf{w}$, which is the positive distinguished subexpression for $w$ in $\bm{w_0}$. Further detail can be found in \cite{MR}, but is not needed for our purposes.
\end{remark}

We now highlight a key property of the cells $\mathcal{R}_{v,w}^{>0}$ and $\textnormal{TrFl}_n^{\geq 0}$ that is not immediately clear from our definition.

\begin{lemma}[{\cite[Lemma 3.11]{KW}}]\label{bruhatinterval}
Let $F\in \mathcal{R}_{v,w}^{>0}$ and $K=\{k_1,\cdots, k_m\}$. Then, $P_K(F)\neq 0$ if and only if $K$ consists of $\{u(1),\ldots,u(m)\}$ for some $v^{-1}\leq u\leq w^{-1}$.
\end{lemma} 

Note that the above lemma differs from the statement cited by the inclusion of inverses on the permutations $v$ and $w$. This is because of the difference in convention in defining $\mathcal{R}_{v,w}^{>0}$ between \cite{MR} and \Cref{marshrietsch}.

\begin{lemma}
The image of $\Trop \Phi_{v,w}$, the tropicalization of $\Phi_{v,w}$, lies in $\textnormal{TrFl}_n^{\geq 0}$. 
\end{lemma}

\begin{proof}
In the case of $\Phi_{id,w_0}$, this result is a direct application of {\cite[Theorem 2]{PS}} and \cite{SW}. For other cells, we make use of \Cref{bruhatinterval} to conclude that for some $S\subset 2^{[n]}$, $P_I=0$ if and only if $I\in S$. We can work with the lower dimensional space where we project away the coordinates $P_I$ for which $I\in S$. This space is cut out by the ideal obtained from the incidence \Pl ideal by setting $P_I=0$ for all $I\in S$ in each polynomial in the incidence \Pl ideal. The (non-zero) coordinates in the resulting variety are still the images of subtraction free polynomials by \Cref{sumonly}, allowing us to directly apply {\cite[Theorem 2]{PS}} again.    
\end{proof}

\begin{definition}
 We will denote the image of $\Trop \Phi_{v,w}$ by $\textnormal{TrFl}_{v,w}^{>0}$.
\end{definition}

\begin{remark}
    
As we continue through the next few sections, we will translate much of what we say about the complete flag variety into the language of the tropical complete flag variety and the complete flag Dressian. We will emphasize such ``translations" by labeling the relevant statements with the superscript ``trop".

\end{remark}

 \begin{lemmabis}{bruhatinterval}\label{tropicalbruhatinterval}

Let $P\in \textnormal{TrFl}_{v,w}^{>0}$ and $K=\{k_1,\cdots, k_m\}$. Then, $P_K\neq \infty$ if and only if $K$ consists of $\{u(1),\ldots,u(m)\}$ for some $v^{-1}\leq u\leq w^{-1}$.

\end{lemmabis}

 \subsection{A Graphical Description of the Marsh-Rietsch Parametrization}

\ptitle{Motivation/ outline of the section} 

Fix $v\leq w\in \mathfrak{S}_n$. Our main goal in this section is to construct a digraph $G_{v,w}$ which offers a combinatorial way to determine which \Pl coordinates are positive in $\mathcal{R}_{v,w}^{>0}$. 

Recall that in constructing $\mathcal{R}_{v,w}^{>0}$, we use subexpressions $\bm{w}$ and $\bm{v}$ of the expression $\bm{w_0}$ defined in \Cref{eq:w0expression}. The expression $\bm{w_0}$ consists of a series of runs, as defined in \Cref{def:runs}. By \Cref{def:expression}, a subexpression of $\bm{w_0}$ is just a subsequence of the sequence of simple reflections in \Cref{eq:w0expression}. Thus, each simple reflection in $\bm{w}$ or $\bm{v}$ can be associated with the run that it comes from in \Cref{eq:w0expression}.

To define $G_{v,w}$, we start by constructing a skeleton digraph $S$ to which we will apply a sequence of transformations. The skeleton $S$ has $2n$ labeled vertices, with vertices labeled from $1'$ to $n'$ lying vertically, with $n'$ on top and $1'$ on the bottom, and vertices labeled from $1$ to $n$ lying horizontally, with $1$ to the right and $n$ to the left. It has a directed edge from $i'$ to $i$ for each $i\in [n]$, as illustrated in \Cref{emptyfig}. For any graph obtained from $S$ by adding vertices, adding directed edges, and permuting vertex labels, we will say \textit{``strand $\sigma$"} to refer to the diagonal edge that is $\sigma$ from the bottom, so that the bottom-most of these diagonal edges is strand $1$ and the top-most is strand $n$. We will also say \textit{``column $c$"} to refer to the vertical slice of the graph between the vertices $n-c+2$ and $n-c+1$ (with column $1$ referring to the part of the graph to the left of vertex $n$). The columns are ordered left to right and give us a way to refer to a particular vertical slice of the graph. We use the notation $(\sigma,c)_G$ to refer to the intersection of strand $\sigma$ and column $c$ in graph $G$.

\begin{definition}\label{constructG}
 Let $v\leq w\in \mathfrak{S}_n$, with $\ell(w)=k$. We present an inductive construction of $G_{v,w}$ starting from $G_{v,w}^{(0)}=S$. Recall the product expansion for a matrix $M_{v,w}(\bm{a})$ as $M=M_1\cdots M_k$ described in \Cref{matrixdecompose}. Each $M_j$ corresponds to some simple reflection $s_{i_j}$ in $\bm{w}$, and can thus be associated to a specific run of $\bm{w_0}$. For $1\leq j\leq k$, we obtain $G_{v,w}^{(j)}$ from $G_{v,w}^{(j-1)}$ as follows: 

\begin{enumerate}
    \item \label{firststep} If $M_j=x_{i_{j}}(a_{j})$ comes from run $r$ of $\bm{w_0}$, then we obtain $G_{v,w}^{(j)}$ from $G_{v,w}^{(j-1)}$ by adding a vertical edge in column $r$, from $(i_j,r)_{G_{v,w}^{(j-1)}}$ to $(i_{j}+1,r)_{G_{v,w}^{(j-1)}}$.
    \item \label{secondstep} If $M_j=\dot{s}_{i_j}$, then we obtain $G_{v,w}^{(j)}$ from $G_{v,w}^{(j-1)}$ by swapping strands $i_j$ and $i_{j}+1$, including the primed source vertices on these strands but not the unprimed sink vertices. As part of this swap, we maintain the incidence of any vertical edges. In other words, if a vertical edge started (terminated) on strand ${i_{j}}$ in $G_{v,w}^{(j-1)}$, then it now starts (terminates) on strand $i_j+1$ in $G_{v,w}^{(j)}$.
\end{enumerate}
We define $G_{v,w}=G_{v,w}^{(k)}$ to be the end result of this procedure. 
\end{definition}

 \begin{figure}[H]
 \begin{tikzpicture}[node distance={10.5 mm}, thick, main/.style = {draw, circle}, 
blank/.style={circle, draw=green!0, fill=green!0, very thin, minimum size=7mm},]

\node[main] (1) {$1'$};
\node[main] (2) [above of=1] {$2'$};
\node (vdots) [above of = 2] {$\vdots$}; 
\node[main] (n) [above of=vdots] {$n'$};
\node (blank)[above of = n]{};
\node[main] (nn) [right  of=blank] {$n$};
\node (hdots) [right of = nn] {$\hdots$};
\node[main] (22) [right of = hdots] {$2$};
\node[main] (11) [right of = 22] {$1$};
\draw[-{Latex[length=3mm]}] (1) -- (11);
\draw[-{Latex[length=3mm]}] (2) -- (22);
\draw[-{Latex[length=3mm]}] (n) -- (nn);
\end{tikzpicture} 
\caption{The skeleton digraph, $S$.}

\label{emptyfig}
 \end{figure}

\begin{remark} \label{rmk:wkandvk}
 Fix $v\leq w\in \mathfrak{S}_n$ and let $\bm{w}=s_{i_1}\cdots s_{i_\ell}$ be the positive distinguished subexpression for $w$ in $\bm{w_0}$, with positive distinguished subexpression $s_{i_{j_1}}\cdots s_{i_{j_m}}$ for $v$. The inductive nature of the definition of $G_{v,w}$ guarantees that for $d\in[\ell]$, $G_{v,w}^{(d)}=G_{v_{(d)},w_{(d)}}$, where $w_{(d)}=s_{i_1}\cdots s_{i_d}$ and $v_{(d)}=s_{i_{j_1}}\cdots s_{i_{j_t}}$, with $t$ maximal such that $j_t\leq d$. We will make use of this fact implicitly in the following example.  
\end{remark}

\ptitle{Example of constructing the graph} \begin{example}\label{toymodel2}
Let's continue with \Cref{toymodel}, where $n=4$, $\bm{w}=s_1s_3s_2s_1$ is the positive distinguished subexpression for $w=4213$ in $\bm{w_0}$, and $\bm{v}=s_2$ is the positive distinguished subexpression for $v=1324$ in $\bm{w}$. To visualize this information, we can write out the expression $\bm{w_0}$ with the simple reflections in $\bm{w}$ underlined, and those in $\bm{v}$ double underlined. We obtain $(\underline{s_1}s_2\underline{s_3})(s_1\underline{\underline{s_2}})(\underline{s_1})$. We see here that the first two simple reflections of $\bm{w}$ come from the first run of $\bm{w_0}$, the third simple reflection comes from the second run of $\bm{w_0}$, and the final simple reflection comes from the third run of $\bm{w_0}$. 

Let $\bm{a}=(a_1,a_2,a_4)$. We had $M_{v,w}(\bm{a})=x_1(a_1)x_3(a_2)\dot{s}_2x_1(a_4)$. The first term, $x_1(a_1)$, comes from the first run of $\bm{w_0}$ and tells us to add a vertical edge from strand $1$ to strand $2$ in column $1$, obtaining \Cref{examp1}.  

\begin{figure}[H]
\begin{tikzpicture}[node distance={10.5 mm}, thick, main/.style = {draw, circle}, 
blank/.style={circle, draw=green!0, fill=green!0, very thin, minimum size=3.5mm},]
\node[main] (1) {$1'$};
\node[main] (2) [above of=1] {$2'$};
\node[main] (3) [above of = 2] {$3'$}; 
\node[main] (4) [above of=3] {$4'$};
\node (blank)[above of = 4]{};
\node[main] (44) [right  of=blank] {$4$};
\node[main] (33) [right of = 44] {$3$};
\node[main] (22) [right of = 33] {$2$};
\node[main] (11) [right of = 22] {$1$};
\node (bot1)[right of =2]{};
\node (top1)[above of = bot1]{};
\draw[-{Latex[length=3mm]}] (1) -- (11);
\draw[-{Latex[length=3mm]}] (2) -- (22);
\draw[-{Latex[length=3mm]}] (3) -- (33);
\draw[-{Latex[length=3mm]}] (4) -- (44);
\draw[-{Latex[length=3mm]}] ([xshift=-5mm,yshift=-5mm]bot1.center) -- ([xshift=-5mm,yshift=-5mm] top1.center) node [pos=0.6,right] {};
\end{tikzpicture} 
\caption{The graph $G_{v,w}^{(1)}=G_{id,s_1}$.}
\label{examp1}
\end{figure}

The next term, $x_3(a_2)$, also comes from the first run of $\bm{w_0}$ and tells us to add an edge from strand $3$ to strand $4$, still in column $1$. This is shown in \Cref{examp2}. 
 
  \begin{figure}[H]
 \begin{tikzpicture}[node distance={10.5 mm}, thick, main/.style = {draw, circle,minimum size=2 mm}, 
blank/.style={circle, draw=green!0, fill=green!0, very thin, minimum size=3.5mm},]

\node[main] (1) {$1'$};
\node[main] (2) [above of=1] {$2'$};
\node[main] (3) [above of = 2] {$3'$}; 
\node[main] (4) [above of=3] {$4'$};
\node (blank)[above of = 4]{};
\node[main] (44) [right  of=blank] {$4$};
\node[main] (33) [right of = 44] {$3$};
\node[main] (22) [right of = 33] {$2$};
\node[main] (11) [right of = 22] {$1$};
\node (bot1)[right of =2]{};
\node (top1)[above of = bot1]{};
\node (bot2)[right of = 4]{};
\draw[-{Latex[length=3mm]}] (1) -- (11);
\draw[-{Latex[length=3mm]}] (2) -- (22);
\draw[-{Latex[length=3mm]}] (3) -- (33);
\draw[-{Latex[length=3mm]}] (4) -- (44);
\draw[-{Latex[length=3mm]}] ([xshift=-5mm,yshift=-5mm]bot1.center) -- ([xshift=-5mm,yshift=-5mm]top1.center) node [pos=0.6, right] {};
\draw[-{Latex[length=3mm]}] ([xshift=-5mm,yshift=-5mm]bot2.center) -- ([xshift=-5mm,yshift=-5mm]44.center) node [pos=0.6, right] {};
\end{tikzpicture} 
\caption{The graph $G_{v,w}^{(2)}=G_{id,s_1s_3}$.}
\label{examp2}
 \end{figure}
 
 The next term, $\dot{s}_2$, tells us to swap the strands $2$ and $3$. Note that the vertical edge terminating on strand $2$ in $G_{v,w}^{(2)}$ should terminate on strand $3$ in $G_{v,w}^{(3)}$, as in \Cref{examp3}. Similarly, the edge originating on strand $3$ in $G_{v,w}^{(2)}$ should originate on strand $2$ in $G_{v,w}^{(3)}$. Observe that while the primed vertices are permuted, the unprimed vertices remain unchanged.
 
   \begin{figure}[H]
 \begin{tikzpicture}[node distance={10.5 mm}, thick, main/.style = {draw, circle,minimum size=2 mm}, 
blank/.style={circle, draw=green!0, fill=green!0, very thin, minimum size=3.5mm},]

\node[main] (1) {$1'$};
\node[main] (2) [above of=1] {$3'$};
\node[main] (3) [above of = 2] {$2'$}; 
\node[main] (4) [above of=3] {$4'$};
\node (blank)[above of = 4]{};
\node[main] (44) [right  of=blank] {$4$};
\node[main] (33) [right of = 44] {$3$};
\node[main] (22) [right of = 33] {$2$};
\node[main] (11) [right of = 22] {$1$};
\node (bot1)[right of =2]{};
\node (top1)[right of = 4]{};
\node (bot2)[right of = 4]{};
\node (bot5) [right of = 3]{};
\node (bot4)[right of =bot5]{};
\node (top4)[above of = bot4]{};
\node (colend) [below of = 33]{};
\draw[-{Latex[length=3mm]}] (1) -- (11);
\draw[-{Latex[length=3mm]}] (2) -- (22);
\draw[-{Latex[length=3mm]}] (3) -- (33);
\draw[-{Latex[length=3mm]}] (4) -- (44);
\draw[-{Latex[length=3mm]}] ([xshift=-5mm,yshift=-5mm]bot1.center) to [out=-260,in=-102] ([xshift=-5mm,yshift=-5mm]top1.center);
\draw[] ([xshift=-5mm,yshift=-5mm]bot5.center) -- ([xshift=-5mm,yshift=-5mm]colend.center);
\draw[-{Latex[length=3mm]}] ([xshift=-5mm,yshift=-5mm]bot5.center) to [out=-285,in=-72] ([xshift=-5mm,yshift=-5mm]44.center);

\end{tikzpicture} 
\caption{The graph $G_{v,w}^{(3)}=G_{s_2,s_1s_3s_2}$.}
\label{examp3}
 \end{figure}
 
 The final term, $x_1(a_4)$, comes from the third run of $\bm{w_0}$ and tells us to add an edge from strand $1$ to strand $2$ in column $3$, as illustrated in \Cref{examp4}.

 \begin{figure}[H]
 \begin{tikzpicture}[node distance={10.5 mm}, thick, main/.style = {draw, circle,minimum size=2 mm}, 
blank/.style={circle, draw=green!0, fill=green!0, very thin, minimum size=3.5mm},]

\node[main] (1) {$1'$};
\node[main] (2) [above of=1] {$3'$};
\node[main] (3) [above of = 2] {$2'$}; 
\node[main] (4) [above of=3] {$4'$};
\node (blank)[above of = 4]{};
\node[main] (44) [right  of=blank] {$4$};
\node[main] (33) [right of = 44] {$3$};
\node[main] (22) [right of = 33] {$2$};
\node[main] (11) [right of = 22] {$1$};
\node (bot1)[right of =2]{};
\node (top1)[right of = 4]{};
\node (bot2)[right of = 4]{};
\node (bot5) [right of = 3]{};
\node (bot4)[below of = 22]{};
\node (colend) [below of = 33]{};
\draw[-{Latex[length=3mm]}] (1) -- (11);
\draw[-{Latex[length=3mm]}] (2) -- (22);
\draw[-{Latex[length=3mm]}] (3) -- (33);
\draw[-{Latex[length=3mm]}] (4) -- (44);
\draw[-{Latex[length=3mm]}] ([xshift=-5mm,yshift=-5mm]bot1.center) to [out=-260,in=-102] ([xshift=-5mm,yshift=-5mm]top1.center);
\draw[] ([xshift=-5mm,yshift=-5mm]bot5.center) -- ([xshift=-5mm,yshift=-5mm]colend.center);
\draw[-{Latex[length=3mm]}] ([xshift=-5mm,yshift=-5mm]bot5.center) to [out=-285,in=-72] ([xshift=-5mm,yshift=-5mm]44.center);
\draw[-{Latex[length=3mm]}] ([xshift=-5mm,yshift=-5mm]bot4.center) -- ([xshift=-5mm,yshift=-5mm]22.center) node [pos=0.6, right] {};
\end{tikzpicture} 
\caption{The complete construction of $G_{v,w}=G_{v,w}^{(4)}$.}
 \label{examp4}
 \end{figure}

 \end{example}
 
 We now explain the combinatorial value of the graphs $G_{v,w}$.
 \begin{theorem}\label{pathcollections}
 For $v\leq w\in \mathfrak{S}_n$, and a flag $F\in\mathcal{R}_{v,w}^{>0}$, the \Pl coordinate $P_I(F)$ is non-zero if and only if there is a non-intersecting path collection from $\{1',2',\ldots,|I|'\}$ to $I$ in $G_{v,w}$.  
 \end{theorem}
 
 To prove this, we add weights on $G_{v,w}$ to obtain a weighted digraph $G_{v,w}(\bm{a})$ which, through the LGV construction, will relate $G_{v,w}$ to the Marsh-Rietsch parameterization of $\textnormal{Fl}_n^{\geq 0}$. 
 
 \begin{definition}\label{weightedG}
  Let $G_{v,w}(\bm{a})$ be the weighted digraph on the underlying digraph $G_{v,w}$ obtained by modifying the inductive steps in the construction of $G_{v,w}$ in \Cref{constructG} as follows:
  
  \begin{primenumerate}
   \item If $M_j=x_{i_j}(a_j)$, assign the vertical edge which is added in \Cref{firststep} of \Cref{constructG} a weight of $a_j$.
   \item If $M_j=\dot{s}_{i_j}$ comes from run $r$ of $w_0$, then, after appropriately modifying the strands and incident edges as in \Cref{secondstep} of \Cref{constructG}, assign the diagonal edge between $(i_j, r-2)_{G_{v,w}^{(j-1)}}$ and $(i_j, r-1)_{G_{v,w}^{(j-1)}}$ a weight of $-1$. 
  \end{primenumerate}
Any edges not already assigned a weight by this procedure will be of weight $1$. 
 \end{definition}
 
 \begin{example}\label{toymodelweighted}
 We continue \Cref{toymodel2}, with $M_{v,w}(\bm{a})=x_1(a_1)x_3(a_2)\dot{s}_2x_1(a_4)$. Modifying the steps applied to obtain \Cref{examp2}, \Cref{examp3}, and \Cref{examp4} appropriately, we obtain the weighted graphs illustrated in \Cref{weightedgraph}.
 
 \begin{figure}[H]
  \begin{tikzpicture}[node distance={10.5 mm}, thick, main/.style = {draw, circle,minimum size=2 mm}, 
blank/.style={circle, draw=green!0, fill=green!0, very thin, minimum size=3.5mm},]

\node[main] (1) {$1'$};
\node[main] (2) [above of=1] {$2'$};
\node[main] (3) [above of = 2] {$3'$}; 
\node[main] (4) [above of=3] {$4'$};
\node (blank)[above of = 4]{};
\node[main] (44) [right  of=blank] {$4$};
\node[main] (33) [right of = 44] {$3$};
\node[main] (22) [right of = 33] {$2$};
\node[main] (11) [right of = 22] {$1$};
\node (bot1)[right of =2]{};
\node (top1)[above of = bot1]{};
\node (bot2)[right of = 4]{};
\draw[-{Latex[length=3mm]}] (1) -- (11);
\draw[-{Latex[length=3mm]}] (2) -- (22);
\draw[-{Latex[length=3mm]}] (3) -- (33);
\draw[-{Latex[length=3mm]}] (4) -- (44);
\draw[-{Latex[length=3mm]}] ([xshift=-5mm,yshift=-5mm]bot1.center) -- ([xshift=-5mm,yshift=-5mm]top1.center) node [pos=0.6, right] {$a_1$};
\draw[-{Latex[length=3mm]}] ([xshift=-5mm,yshift=-5mm]bot2.center) -- ([xshift=-5mm,yshift=-5mm]44.center) node [pos=0.6, right] {$a_2$};
\end{tikzpicture} 
\begin{tikzpicture}[node distance={10.5 mm}, thick, main/.style = {draw, circle,minimum size=2 mm}, 
blank/.style={circle, draw=green!0, fill=green!0, very thin, minimum size=3.5mm},]

\node[main] (1) {$1'$};
\node[main] (2) [above of=1] {$3'$};
\node[main] (3) [above of = 2] {$2'$}; 
\node[main] (4) [above of=3] {$4'$};
\node (blank)[above of = 4]{};
\node[main] (44) [right  of=blank] {$4$};
\node[main] (33) [right of = 44] {$3$};
\node[main] (22) [right of = 33] {$2$};
\node[main] (11) [right of = 22] {$1$};
\node (bot1)[right of =2]{};
\node (top1)[right of = 4]{};
\node (bot2)[right of = 4]{};
\node (bot5) [right of = 3]{};
\node (bot4)[right of =bot5]{};
\node (top4)[above of = bot4]{};
\node (col3)[below of = 22]{};
\node (colend) [below of = 33]{};
\draw[-{Latex[length=3mm]}] (1) -- (11);
\draw[-{Latex[length=3mm]}] (2) -- (22);
\draw[-{Latex[length=3mm]}] (3) -- (33);
\draw[-{Latex[length=3mm]}] (4) -- (44);
\draw[-{Latex[length=3mm]}] ([xshift=-5mm,yshift=-5mm]bot1.center) to [out=-260,in=-102] ([xshift=-5mm,yshift=-5mm]top1.center);
\path (bot1)--(top1) node [pos=0.05, left] {$a_1$};
\draw[line width=0.75mm, color=dark-gray] ([xshift=-5mm,yshift=-5mm]bot5.center) -- ([xshift=-5mm,yshift=-5mm]colend.center);
\draw[-{Latex[length=3mm]}] ([xshift=-5mm,yshift=-5mm]bot5.center) to [out=-285,in=-72] ([xshift=-5mm,yshift=-5mm]44.center);
\path (top4) -- (44) node [pos=0.3, left] {$a_2\;\;\;\;\;$};
\end{tikzpicture} 
 \begin{tikzpicture}[node distance={10.5 mm}, thick, main/.style = {draw, circle,minimum size=2 mm}, 
blank/.style={circle, draw=green!0, fill=green!0, very thin, minimum size=3.5mm},]

\node[main] (1) {$1'$};
\node[main] (2) [above of=1] {$3'$};
\node[main] (3) [above of = 2] {$2'$}; 
\node[main] (4) [above of=3] {$4'$};
\node (blank)[above of = 4]{};
\node[main] (44) [right  of=blank] {$4$};
\node[main] (33) [right of = 44] {$3$};
\node[main] (22) [right of = 33] {$2$};
\node[main] (11) [right of = 22] {$1$};
\node (bot1)[right of =2]{};
\node (top1)[right of = 4]{};
\node (bot2)[right of = 4]{};
\node (bot5) [right of = 3]{};
\node (bot4)[right of =bot5]{};
\node (top4)[above of = bot4]{};
\node (col3)[below of = 22]{};
\node (colend) [below of = 33]{};
\draw[-{Latex[length=3mm]}] (1) -- (11);
\draw[-{Latex[length=3mm]}] (2) -- (22);
\draw[-{Latex[length=3mm]}] (3) -- (33);
\draw[-{Latex[length=3mm]}] (4) -- (44);
\draw[-{Latex[length=3mm]}] ([xshift=-5mm,yshift=-5mm]bot1.center) to [out=-260,in=-102] ([xshift=-5mm,yshift=-5mm]top1.center);
\path (bot1)--(top1) node [pos=0.05, left] {$a_1$};
\draw[line width=0.75mm, color=dark-gray] ([xshift=-5mm,yshift=-5mm]bot5.center) -- ([xshift=-5mm,yshift=-5mm]colend.center);
\draw[-{Latex[length=3mm]}] ([xshift=-5mm,yshift=-5mm]bot5.center) to [out=-285,in=-72] ([xshift=-5mm,yshift=-5mm]44.center);
\path (top4) -- (44) node [pos=0.3, left] {$a_2\;\;\;\;\;$};
\draw[-{Latex[length=3mm]}] ([xshift=-5mm,yshift=-5mm]col3.center) -- ([xshift=-5mm,yshift=-5mm]22.center) node [pos=0.6, right] {$a_4$};
\end{tikzpicture} 
\caption{From left to right, the weighted digraphs $G_{id,s_1s_3}(a_1,a_2)$, $G_{s_2,s_1s_3s_2}(a_1,a_2)$, and $G_{s_2,s_1s_3s_2s_1}(a_1,a_2,a_4)$. For legibility, the diagonal edge with weight $-1$ is indicated by a bolded line segment and edges with weight 1 are not marked.}
 \label{weightedgraph}
 \end{figure}

 \end{example}
  \begin{proposition}\label{LGVMR} For any $v\leq w\in \mathfrak{S}_n$ and $\bm{a}\in \mathbb{R}_{>0}^{\ell(w)-\ell(v)}$, the matrix obtained from $G_{v,w}(\bm{a})$ via the LGV construction coincides with $M_{v,w}(\bm{a})$. 
 \end{proposition}
 
 \begin{proof}

 To verify this, we use induction on $\ell(w)$. If $\ell(w)=0$, both constructions yield the identity matrix. 
 
 We now let $\ell=\ell(w)>0$. Let $\bm{w}=s_{i_1}\cdots s_{i_{\ell}}$ be the positive distinguished subexpression for $w$ in $\bm{w_0}$ and let $\bm{v}=s_{{i_{j_1}}}\cdots s_{i_{j_{m}}}$ be the positive distinguished subexpression for $v$ in $\bm{w}$. Let $M=M_{v,w}(\bm{a})$. Our strategy is to compare the paths in $G_{v,w}(\bm{a})$ with those in $G_{v',w'}(\bm{a'})$ for $w'=ws_{i_{\ell}}<w$ and for some appropriate $v'\leq v$ and $\bm{a'}$. Let $N$ be the matrix obtained from $G_{v,w}(\bm{a})$ via the LGV construction. By induction, $M_{v',w'}(\bm{a'})$ is the matrix obtained from the LGV construction on $G_{v',w'}(\bm{a'})$. Since $M_{v,w}(\bm{a})$ is constructed inductively, we will be able to see that $N$ differs from $M_{v',w'}(\bm{a'})$ in exactly the same way as $M_{v,w}(\bm{a})$ differs from $M_{v',w'}(\bm{a'})$. We consider two cases, depending on the value of ${j_{m}}$.
 
 Suppose ${j_{m}}\neq \ell$. Then, $M=M'x_{i_{\ell}}(a_{\ell})$. Note that $\bm{w'}=s_{i_1}\cdots s_{i_{\ell-1}}$ is the positive distinguished subexpression for $w'=ws_{i_{\ell}}$ in $\bm{w_0}$. Let $\bm{a}'$ be $\bm{a}$ with its last coordinate removed. Note that $\bm{v}$ is a positive distinguished subexpression of $\bm{w'}$. By induction, the matrix obtained from $G_{v,w'}(\bm{a}')$ via the LGV construction is $M'=M_{v,w'}(\bm{a'})$. To obtain $G_{v,w}(\bm{a})$ from $G_{v,w'}(\bm{a}')$, we add a single vertical edge $e$ of weight $a_{\ell}$ between strands $i_{\ell}$ and $i_{\ell}+1$ in some column $c$. Note that by construction, $e$ is the last edge which gets added to $G_{v,w}(\bm{a})$ and, as a result, there are no other vertical edges weakly above and to the right of it. For an example, compare $G_{s_2, s_1s_3s_2}(a_1,a_2)$ and $G_{s_2, s_1s_3s_2s_1}(a_1,a_2,a_4)$ in \Cref{weightedgraph}. Let $N$ be the matrix obtained from $G_{v,w}(\bm{a})$ by the LGV construction so that the matrix entry $N_{\alpha,\beta}$ is the sum of weights of paths in $G_{v,w}(\bm{a})$ from source vertex $\alpha'$ to sink vertex $\beta$. Similarly, $M'_{\alpha,\beta}$ is the sum of weights of paths in $G_{v,w'}(\bm{a'})$ from $\alpha'$ to $\beta$. Since there are no edges weakly above and to the right of $e$, any path in $G_{v,w}(\bm{a})$ which uses $e$ must continue along diagonal strand $i_{\ell}+1$ until vertex $i_{\ell}+1$. It follows that all paths in $G_{v,w}(\bm{a})$ which use $e$ terminate on strand $i_{\ell}+1$ and so, if $\alpha,\beta\in[n]$ and $\beta\neq i_{\ell}+1$, then $N_{\alpha,\beta}=M'_{\alpha,\beta}$. Once again using the fact that there are no vertical edges weakly above and to the right of $e$, paths terminate at $i$ in $G_{v,w'}(\bm{a'})$ if and only if they pass through $(i_{\ell},c)_{G_{v,w}}$, the intersection of strand $i_{\ell}$ and column $c$. Thus, for each path $p$ from $\alpha'$ to $i_{\ell}$ in $G_{v,w'}(\bm{a'})$, we obtain a path from $\alpha'$ to $i_{\ell}+1$ in $G_{v,w}(\bm{a})$ which is identical to $p$ until reaching the bottom of $e$, uses edge $e$, and continues to the end of strand $i+1$. This new path has weight $a_{\ell}\omega(p)$. Also, any path from $\alpha'$ to $i_{\ell}+1$ in $G_{v,w'}(\bm{a'})$ is still a path from $\alpha'$ to $i_{\ell}+1$ in $G_{v,w}(\bm{a})$. Thus, for $\alpha\in[n]$, $N_{\alpha,i_{\ell}+1}=a_{\ell}M'_{\alpha,i_{\ell}}+M'_{\alpha,i_{\ell}+1}$ and we may conclude that $N=M'x_{i_{\ell}}(a_{\ell})=M,$ as desired. 
 
 Now, suppose $j_{m}=\ell(w)$ so that  $M=M'\dot{s}_{i_{\ell}}$. If $\bm{w'}$ is as in the previous case, then $\bm{v'}=s_{i_{j_1}}\cdots s_{i_{j_{m-1}}}$ is the positive distinguished subexpression for a permutation $v'=vs_{i_{\ell}}$ in  $\bm{w'}$. By induction, the matrix obtained from $G_{v',w'}(\bm{a})$ via the LGV construction is $M'=M_{v',w'}(\bm{a})$. We obtain $G_{v,w}(\bm{a})$ from $G_{v',w'}(\bm{a})$ by swapping strands $i_{\ell}$ and $i_{\ell}+1$ and by adding a section $d$ of weight $-1$ on strand $i_{\ell}$ between columns $c$ and $c+1$ for some $c$. Note that by construction, adding $d$ is the last step in the construction of $G_{v,w}(\bm{a})$ and, as a result, there are no vertical edges weakly above and to the right of it. For example, compare $G_{id,s_1s_3}(a_1,a_2)$ and $G_{s_2,s_1s_3s_2}(a_1,a_2)$ in \Cref{weightedgraph}. Let $N$ be the matrix obtained from $G_{v,w}(\bm{a})$ by the LGV construction. Recall that any vertical edge starting or terminating on strand $i_{\ell}$ in $G_{v',w'}(\bm{a})$ will start or terminate on strand $i_{\ell}+1$ in $G_{v,w}(\bm{a})$, respectively, and any vertical edge starting or terminating on strand $i_{\ell}+1$ in $G_{v',w'}(\bm{a})$ will start or terminate on strand $i_{\ell}$ in $G_{v,w}(\bm{a})$, respectively. Accordingly, any path $p$ in $G_{v',w'}(\bm{a})$ that terminates at $\beta\neq i_{\ell},i_{\ell}+1$ will correspond to a path $q$ in $G_{v,w}(\bm{a})$ that also terminates at $\beta$ and is identical to $p$ except $q$ uses strand $i_{\ell}$ whenever $p$ uses strand $i_{\ell}+1$, and $q$ uses strand $i_{\ell}+1$ whenever $p$ uses strand $i_{\ell}$. Since $q$ does not terminate on strand $i_{\ell}$ and there are no vertical edges to the right of $d$, it does not use $d$ and so $\omega(q)=\omega(p)$. Thus, for ${\alpha,\beta}\in [n]$ and $\beta\neq i_{\ell},i_{\ell}+1$,  $N_{\alpha,\beta}=M'_{\alpha,\beta}$. By similar reasoning, for each path $p$ terminating at $i_{\ell}$ in  $G_{v',w'}(\bm{a})$ there is a corresponding path $q$ terminating at $i_{\ell}+1$ in $G_{v,w}(\bm{a})$ with $\omega(q)=\omega(p)$. Thus, for $\alpha\in[n]$, $N_{\alpha,i_{\ell}+1}=M'_{\alpha,i_{\ell}}$. Similarly, for each path $p$ terminating at $i_{\ell}+1$ in $G_{v',w'}(\bm{a})$, there is a corresponding path $q$ terminating at $i_{\ell}$ in $G_{v,w}(\bm{a})$. However, since there are no vertical edges to the right of $d$, $q$ must pass through $d$. Thus, $\omega(q)=-\omega(p)$ and, for $\alpha\in[n]$, $N_{\alpha,i_{\ell}}=-M'_{\alpha,i_{\ell}+1}$. We conclude that $N=M'\dot{s}_{i_{\ell}}=M$, as desired.
 
 \end{proof}

\begin{lemma}\label{vertexinversion}
    Read from bottom to top, the vertices on the left of $G_{v,w}$ are $v(1)',v(2)',\ldots,v(n)'$.
\end{lemma} 
\begin{proof}
    This follows from the fact that the $s_i$ which are in $\bm{v}$ cause us to swap the primed vertices on strands $i$ and $i+1$. 
\end{proof}
 
 \ptitle{Useful results for later}
 We now state and prove a previously known result {\cite[Proposition 5.1]{RW}}, but we give a new proof using the machinery introduced in this section. We adopt the notation $[k]'\coloneqq \{1',2',\ldots,k'\}$.
 
 \begin{lemma}\label{sumonly}
 Let $v\leq w\in \mathfrak{S}_n$, $k\in [n]$ and $\bm{a}\in\mathbb{R}_{>0}^{\ell(w)-\ell(v)}$. Then, for any $I\in {\binom{[n]}{k}}$ the \Pl coordinate $P_I(M_{v,w}(\bm{a}))$ is a subtraction free polynomial combination of the weights $\bm{a}$. Moreover, each non-intersecting path collection from $[|I|]'$ to $I$ in $G_{v,w}(\bm{a})$ contributes a term to $P_I(M_{v,w}(\bm{a}))$. 
 \end{lemma}
 \begin{proof} From \Cref{cor:LGVcor} and \Cref{LGVMR}, we see that each such non-intersecting path collection does indeed contribute a term to $P_I$. What needs to be verified is that all such terms have positive coefficients. A non-intersecting path collection is determined by its vertical edges. Thus, there is no cancellation in $P_I$ between the monomials corresponding to different path collections. Suppose towards a contradiction that a non-intersecting path collection $P$ contributed a term with a negative coefficient in \Cref{LGVcor}, then we could fix the weights of all vertical edges not involved in $P$ and make the weights of the vertical edges which are involved in $P$ large. Then, for those weights, the expression in \Cref{LGVcor} would be negative. However, \Cref{cor:LGVcor} and \Cref{LGVMR} tell us that this expression is a \Pl coordinate of a flag in $\Fl_n^{\geq0}$ and thus must be positive for any positive weights by \Cref{forwarddir}. 
 \end{proof}
 
 \begin{proof}[Proof of \Cref{pathcollections}]
 This is immediate from \Cref{sumonly}.
\end{proof}

\subsection{Extremal Non-Zero \Pl Coordinates}

 We define a special subset of the \Pl coordinates of a flag which we call \textit{extremal non-zero \Pl coordinates}. For a flag $F\in \textnormal{Fl}_n^{\geq 0}$, the indices of these coordinates will depend only on which cell $\mathcal{R}_{v,w}^{>0}$ contains $F$. In future sections, we will prove a number of useful facts about the extremal non-zero \Pl coordinates of a flag, including that the extremal non-zero \Pl coordinates determine all of the other \Pl coordinates and also form a positivity test in $\mathcal{R}_{v,w}^{>0}$. We will also see the beginnings of the parallel story about the complete flag Dressian take shape in this section.

\ptitle{Description and intuition of $\Xi$} \ptitle{Description and intuition of $\Xi$} Let $P\in \mathbb{R}\mathbb{P}^{\binom{[n]}{1}}\times\cdots\times \mathbb{R}\mathbb{P}^{\binom{[n]}{n-1}}$ and let $S=2^{[n]}\setminus\{\emptyset, [n]\}$ be the set of nonempty proper subsets of $[n]$. We define a map $\Xi_P: S\rightarrow S$.  Intuitively, when applied to the index of a non-zero \Pl coordinate $P_I$, this map finds the largest member of $I$ that can be increased without making the corresponding \Pl coordinate $0$ and increases it maximally. 

\begin{definition}\label{ximap}
 Let $P\in \mathbb{R}\mathbb{P}^{\binom{[n]}{1}}\times\cdots\times \mathbb{R}\mathbb{P}^{\binom{[n]}{n-1}}$ and $I\in S$. Let $B= \left\{i\;|\;\exists\; j\notin I,  \;j>i, \; P_{(I\setminus i)\cup j}\neq 0\right\}$. If $B$ is non-empty, define $b=\max B$ and $a=\max_{j\notin I}\left\{j\;|\; P_{(I\setminus b)\cup j}\neq 0\right\}$. Then, 
\begin{align*}
    \Xi_P(I)=
    \begin{cases}
    (I\setminus b)\cup a &\;\;\textnormal{if } I \textnormal{ is the index of a non-zero \Pl coordinate and $B$ is non-empty,}\\
    I &\;\;\textnormal{otherwise.}  
    \end{cases}
\end{align*}

\end{definition}

\begin{remark}\label{rem:ximatroidstructure}
Recall that the indices of non-zero \Pl coordinates of a point in the flag variety form the bases of a flag matroid. Thus, for $P\in\textnormal{Fl}_n$, $\Xi_P$ acts by basis exchange on the constituent matroids of a flag matroid. Also note that, by \Cref{bruhatinterval}, for a flag $F\in\Fl_n^{\geq 0}$, the map $\Xi_{P(F)}$ depends only on the cell $\mathcal{R}_{v,w}^{>0}$ containing $F$. We denote this map by $\Xi_{v,w}$.

\end{remark}
We similarly define $\Xi_P^*$ to be the map going the other way. Applied to the index of a non-zero \Pl coordinate $I$, it tries to find the smallest member of $I$ which can be lowered without making the corresponding \Pl coordinate $0$ and lowers it maximally.

\begin{definition}

Let $P\in \mathbb{R}\mathbb{P}^{\binom{[n]}{1}}\times\cdots\times \mathbb{R}\mathbb{P}^{\binom{[n]}{n-1}}$ and $I\in S$. If $B^*= \left\{i\;|\;\exists\; j\notin I,  \;i>j, \; P_{(I\setminus i)\cup j}\neq 0\right\}$ is non-empty, define $b^*=\min_{i\in I} B^*$ and $a^*=\min_{j\notin I}\left\{j\;|\; P_{(I\setminus b^*)\cup j}\neq 0\right\}$. Then, 
\begin{align*}
    \Xi^*_P(I)=
    \begin{cases}
    (I\setminus b^*)\cup a^* &\;\;\textnormal{if } I \textnormal{ is the index of a non-zero \Pl coordinate and $B^*$ is non-empty,}\\
    I &\;\;\textnormal{otherwise.}  
    \end{cases}
\end{align*}

\end{definition}

Again, this can be thought of as a basis exchange. Also, as above, if $P=P(F)$ for a flag $F$ in some cell $\mathcal{R}_{v,w}^{> 0}$, we can denote the corresponding map by $\Xi^*_{v,w}$.

\begin{defbis}{ximap}
   Let $P\in \mathbb{T}\mathbb{P}^{\binom{[n]}{1}}\times\cdots\times \mathbb{T}\mathbb{P}^{\binom{[n]}{n-1}}$ and $I\in {\binom{[n]}{k}}$ for some $k\in [n]$. If $B= \left\{i\;|\;\exists\; j\notin I,  \;i<j, \; P_{(I\setminus i)\cup j}\neq \infty\right\}$ is non-empty, define $b=\max_{i\in I} B$ and $a=\max_{j\notin I}\left\{j\;|\; P_{(I\setminus b)\cup j}\neq \infty\right\}$. Then, 
\begin{align*}
    \Xi_P(I)=
    \begin{cases}
    (I\setminus b)\cup a &\;\;\textnormal{if } I \textnormal{ is the index of a non-infinite \Pl coordinate and $B$ is non-empty,}\\
    I &\;\;\textnormal{otherwise.}  
    \end{cases}
\end{align*}

\end{defbis}

We observe that if $P\in \textnormal{TrFl}_{v,w}^{>0}=\textnormal{Im}(\Trop\Phi_{v,w})$, then $\Xi_P=\Xi_{v,w}$, since the non-infinite coordinates of $P$ are exactly the same as the non-zero coordinates of a flag $F\in \mathcal{R}_{v,w}^{>0}$ by \Cref{tropicalbruhatinterval}.

 The extremal non-zero \Pl coordinates will be given as the coordinates indexed by certain $\Xi$ orbits. To properly define them, we will first need a preliminary result related to matroids:

 \begin{lemma}[{\cite[Theorem 1.3.1]{Coxeter}}] \label{coxeter} 
 Any matroid (and in particular, positroid) has a unique Gale minimal basis and a unique Gale maximal basis. 
 \end{lemma}
 
Note that the Gale minimal and maximal bases referenced in the previous lemma must simply be the lexicographically minimal and maximal bases, respectively. We use the fact that for $F\in \textnormal{Fl}_n$, the indices of non-zero \Pl coordinates of $F$ form the bases of a flag matroid.

 \begin{definition}\label{extremal}
  Given a flag $F\in \textnormal{Fl}_n$ such that $P_I(F)\geq 0$ for all $\emptyset\neq I\subset[n]$, let $I_k$ be the Gale minimal index of size $k$ such that $P_{I_k}(F)\neq 0$.
  The set of indices of the \textbf{extremal non-zero \Pl coordinates}, referred to as \textbf{extremal indices}, of $P(F)$ consists of the $\Xi_{P(F)}$ orbits of the $I_k$ for $k\in[n-1]$.
  
    \end{definition}
    
In \cite{BEZ} it is shown that the support of a point in the flag Dressian is a flag matroid. Thus, the following is well-defined:

\begin{defbis}{extremal}\label{tropextremalsdefinition}
 Given a point $P\in \textnormal{FlDr}_n^{\geq0}$, let $I_k$ be the Gale minimal index of size $k$ such that $P_{I_k}\neq \infty$.
  The set of indices of the \textbf{extremal non-infinite \Pl coordinates}, referred to as \textbf{extremal indices}, of $P$ consists of the $\Xi_{P}$ orbits of the $I_k$ for $k\in[n-1]$.
 
\end{defbis}
  
As we remarked in \Cref{rem:ximatroidstructure}, the extremal indices of a flag in $\textnormal{Fl}_n^{\geq 0}$ depend only on the cell $\mathcal{R}_{v,w}^{>0}$ containing that flag. Similarly, all points in $\textnormal{TrFl}_{v,w}^{>0}$ have the same extremal indices as a flag in $\mathcal{R}_{v,w}^{>0}$.

By \Cref{bruhatinterval}, in $\mathcal{R}_{v,w}^{>0}$, the Gale minimal indices of non-zero \Pl coordinates are $\{v^{-1}(1),\ldots, v^{-1}(k)\}$ for $k\in [n-1]$ while the Gale maximal indices of non-zero \Pl coordinates are $\{w^{-1}(1),\ldots,w^{-1}(k)\}$ for $k\in[n-1]$. We will use $\Xi_{v,w}$ as a way to interpolate between the non-zero (or non-infinite, in the tropical case) \Pl coordinates of a flag with Gale minimal and Gale maximal indices: The basis exchange axiom for matroids guarantees that repeated applications of $\Xi_{v,w}$ must eventually yield the unique Gale maximal index of a non-zero (non-infinite) \Pl coordinate of a flag in $\mathcal{R}_{v,w}^{>0}$ (a point in $\textnormal{TrFl}_{v,w}^{> 0}$).  

We note that one can define a dual notion of extremal non-zero (non-infinite) \Pl coordinates using the $\Xi^*$ orbit of the Gale maximal indices of non-zero (non-infinite) \Pl coordinates. While these dual extremal indices are in fact different in some cases, every result which we prove would work equally with the dual version of extremal non-zero (non-infinite) \Pl coordinates.

 \ptitle{Example of extremal indices and \Pl coordinates}
   \begin{example}
 We continue with \Cref{toymodel}. Recall that we had 
 \begin{equation*}
     M=\begin{pmatrix}
    1&a_4&a_1&0\\
    0&0&1&0\\
    0&-1&0&a_2 \\
    0&0&0&1
    \end{pmatrix}.
 \end{equation*}
 
 Let $P=P(M)$ and $\Xi=\Xi_P$. The non-zero \Pl coordinate with Gale minimal index of size $1$ is $P_1=1$. Thus, $P_1$ is an extremal non-zero \Pl coordinate. Applying $\Xi$ to $\{1\}$, we replace $1$ with the maximal single element index of a non-zero \Pl coordinate, obtaining $\Xi(\{1\})=\{3\}$. Thus, $P_3=a_1$ is also an extremal non-zero \Pl coordinate.
 
 The non-zero \Pl coordinate with Gale minimal index of size $2$ is $P_{13}=1$. Thus, $P_{13}$ is an extremal non-zero \Pl coordinate. Applying $\Xi$ to $\{1,3\}$ first tries to replace $3$ with something bigger. However, that is not possible since $P_{14}=0$. Thus, it tries to replace $1$ with something bigger, yielding $\Xi(\{1,3\})=\{2,3\}$. Thus $P_{23}=a_4$ is also an extremal non-zero \Pl coordinate. The index $\{2,3\}$ is the Gale maximal index of size $2$ of a non-zero \Pl coordinate.
 
 The non-zero \Pl coordinate with Gale minimal index of size $3$ is $P_{123}=1$. Thus, $P_{123}$ is an extremal non-zero \Pl coordinate. Applying $\Xi$, we obtain $\Xi(\{1,2,3\})=\{1,3,4\}$ and $\Xi(\{1,3,4\})=\{2,3,4\}$. Thus $P_{134}=a_2$ and $P_{234}=a_2a_4$ are both extremal non-zero \Pl coordinates. 
 \end{example}
 
 \begin{example}\label{maxsizecoords}
 For any $F\in \textnormal{Fl}_n$, any non-zero \Pl coordinate with index of size $n-1$ is an extremal non-zero \Pl coordinate. To see this, note that Gale order is a total order on subsets of size $n-1$ and in this case, $\Xi_{P(F)}$ simply acts by replacing the index of a non-zero \Pl coordinate by the next largest such index in Gale order.  
 \end{example}

 The next few results will help us get a handle on which indices are actually extremal.

 \begin{lemma}\label{positroidexch}
 For $F\in \textnormal{Fl}_n$, let $P=P(F)$ and suppose $P_I\geq 0$ for all $\emptyset\neq I\subset [n]$. Let $S$ be any subset of $[n]$ of size at most $n-4$. Let $a<b<c<d$ with $a,b,c,d\notin S$. Then, both $P_{S\cup ac}\neq 0$ and $P_{S\cup bd}\neq 0$ if and only if either both $P_{S\cup ab}\neq 0$ and $P_{S\cup cd}\neq 0$, or both $P_{S\cup ad}\neq 0$ and $P_{S\cup bc}\neq 0$. 

 \end{lemma}
 \begin{proof}
 Since $P_I\geq 0$ for all $\emptyset\neq I\subset[n]$, the lemma follows from the three-term \Pl relation in \Cref{rem:threetermrelations}.
 \end{proof}

\begin{lemma}\label{claim} Let $F\in \textnormal{Fl}_n$ such that $P_I(F)\geq0$ for all $\emptyset\neq I\subset[n]$ and let $\Xi=\Xi_{P(F)}$. Let $I$ be the index of a non-zero \Pl coordinate. Let $s$ be such that $\Xi^{s-1}(I)\neq \Xi^s(I)=\Xi^{s+1}(I)$. For $1\leq r\leq s$, define $\alpha_r$ and $\beta_r$ by $\Xi^{r}(I)=\left(\Xi^{r-1}(I)\setminus \beta_r\right)\cup \alpha_r$. If $1\leq t<r\leq s$, then $\beta_r<\beta_t$.\end{lemma}
 
 \begin{proof} It suffices to show that $\beta_{t+1}<\beta_t$ for $1\leq t< s$. Note that $t<s$ simply ensures that $\Xi^{t+1}(I)\neq \Xi^t(I)$. By definition, we have $\alpha_t>\beta_t$ and $\alpha_{t+1}>\beta_{t+1}$. Also, $\beta_{t+1}\neq \alpha_t$ since in that case, $\Xi$ should replace $\beta_t$ by $\alpha_{t+1}$ rather than $\alpha_t$, by the maximality of $a$ in the definition of $\Xi$. Let $S=\Xi^{t-1}(I)$, so that $S$ and $\Xi^2(S)=\left(S\setminus \beta_t\beta_{t+1}\right)\cup \alpha_t\alpha_{t+1}$ are both indices of non-zero \Pl coordinates. Suppose $\beta_t<\beta_{t+1}$. Then we are in one of the following three cases:
 
 \begin{enumerate}
     \item[Case 1:] If $\beta_t<\beta_{t+1}<\alpha_t<\alpha_{t+1}$, then \Cref{positroidexch} applied to $S\setminus \beta_t\beta_{t+1}$ implies that $(S\setminus \beta_{t})\cup \alpha_{t+1}$ is the index of a non-zero \Pl coordinate as well. However, this contradicts the maximality of $a$ in the definition of $\Xi$. 
     \item[Case 2:] If $\beta_t<\beta_{t+1}<\alpha_{t+1}<\alpha_t$, then \Cref{positroidexch} applied to $S\setminus \beta_t\beta_{t+1}$ implies that $(S\setminus \beta_{t+1})\cup \alpha_{t+1}$ is the index of a non-zero \Pl coordinate as well. This contradicts the maximality of $b$ in the definition of $\Xi$. 
     \item[Case 3:] If $\beta_t<\alpha_t<\beta_{t+1}<\alpha_{t+1}$ then \Cref{positroidexch} applied to $S\setminus \beta_t\beta_{t+1}$ implies that either $(S\setminus \beta_t)\cup \alpha_{t+1}$ or $(S\setminus \beta_{t+1})\cup \alpha_{t+1}$ must be the index of a non-zero \Pl coordinate. These both lead to contradictions, as in cases 1 and 2, respectively.
     
 \end{enumerate}
\end{proof}

\begin{lemma}\label{addins}
Let $F\in \mathcal{R}_{v,w}^{>0}$. Let $I$ be an extremal index with $|I|=k$ and $\Xi_{v,w}(I)=(I\setminus a)\cup b$. Then $b\in \{w^{-1}(1),\ldots,w^{-1}(k)\}$. In particular, once $b$ is added to the index by an application of $\Xi_{v,w}$, it is never removed by subsequent applications of $\Xi_{v,w}$. 

\end{lemma}
\begin{proof}
As noted after \Cref{bruhatinterval}, repeated applications of $\Xi_{v,w}$ eventually yield $\{w^{-1}(1),\ldots,w^{-1}(k)\}$. \Cref{claim} implies that once $b$ is added it cannot be removed by applications of $\Xi_{v,w}$ since $b>a$. The result then follows.
\end{proof}

 \section{The Totally Positive Complete Flag Variety and its Tropicalization}
\label{secttopcell} 
 In the cell decomposition of $\textnormal{Fl}_n^{\geq 0}$, the top dimensional cell is $\mathcal{R}_{id,w_0}^{>0}$. By \Cref{bruhatinterval}, $\mathcal{R}_{id,w_0}^{>0}$ consists of those points where all of the \Pl coordinates are positive. (This is also straightforward to deduce from $G_{id,w_0}$.) In this section, we give a detailed combinatorial description of this space and of its tropicalization, the totally positive tropical complete flag variety. In the next section, we will generalize the arguments presented here to address the totally nonnegative complete flag variety and the totally nonnegative tropical complete flag variety.

\subsection{Graphical Description of Extremal Indices}
 
 We first describe the extremal coordinates of flags $F\in\mathcal{R}^{>0}_{id,w_0}.$ We start with an example.
 
 \begin{example}Let $n=5$ and $F\in \mathcal{R}^{>0}_{id,w_0}$, so $P_I(F)\neq 0$ for all $\emptyset\neq I\subset [5]$. We will work out the extremal indices of size $2$. $P_{12}\neq 0$ is the non-zero \Pl coordinate with Gale minimal index of size $2$, and so it is an extremal non-zero \Pl coordinate. Applying $\Xi_{id,w_0}$, we obtain the index of a non-zero \Pl coordinate constructed by removing the $2$ and replacing it with the maximal possible element of $[5]$. Since all \Pl coordinates are non-zero, we can replace the $2$ by a $5$. Thus, $P_{15}$ is also an extremal non-zero \Pl coordinate. Applying $\Xi_{id,w_0}$ again, we increase the $1$ in the index maximally to obtain that $P_{45}$ is an extremal non-zero \Pl coordinate as well. Similarly, the extremal indices of size $3$ are $\{\{1,2,3\},\{1,2,5\},\{1,4,5\},\{3,4,5\}\}$.
 \end{example}
 
 We can extend this example to a general fact: 
 
 \begin{proposition}\label{topcellextremals}
 The extremal indices of a flag $F\in \mathcal{R}^{>0}_{id,w_0}\subset \textnormal{Fl}_n$ are precisely the complements of intervals in $[n]$. Explicitly, these are the subsets of the form $\{1,2,\ldots, m-1,m, k,k+1,\ldots, n\}$ for $0\leq m<k-1\leq n$, excluding $m=0, k=n$. 
 \end{proposition}
 
 \begin{proof}
 Start with the Gale minimal extremal index of size $m+(n-k)+1$, which is $\{1,2,\ldots, m+n-k+1\}$. Now, bearing in mind that all \Pl coordinates are non-zero in this case, apply $\Xi_{id,w_0}$ to this index $n-k+1$ times. Each time we replace the largest entry of the index by the largest element of $[n]$ not in the index. After $(n-k)+1$ applications of $\Xi_{id,w_0}$, this yields $\{1,2,\ldots, m-1,m, k,k+1,\ldots, n\}$, proving the result.
 \end{proof}
 
 \begin{propbis}{topcellextremals}\label{troptopcellextremals}
  
 The extremal indices of a point $P\in \textnormal{TrFl}^{>0}_{id,w_0}\subset\textnormal{TrFl}_n$ are precisely the complements of intervals in $[n]$. Explicitly, these are the subsets of the form $\{1,2,\ldots, m-1,m, k,k+1,\ldots, n\}$ for $0\leq m<k-1\leq n$, excluding $m=0, k=n$.

 \end{propbis}
 \begin{proof}
 The extremal indices of a point in $\textnormal{TrFl}_{id,w_0}^{>0}$ are identical to those of a flag in $\mathcal{R}_{id,w_0}^{>0}$.
 \end{proof}
 \begin{definition}
  We call a path collection in $G_{v,w}$ \textbf{diagonal} if none of the paths it contains use any vertical edges.
  \end{definition}

 \begin{definition}\label{defleftgreedy}
Consider a source set $A\subset[n]'$ in $G_{v,w}$. Label the sources in $A$ from top to bottom as $a_1',\ldots, a_m'$. Construct a non-intersecting path collection as follows: Start with the path originating from $a_1'$ which takes every vertical edge it can. After having added a path originating at $a_i'$ to the path collection, we add in the path originating at $a_{i+1}'$ which takes every vertical edge it can without intersecting any of the paths already in the collection. The path collection obtained after adding in the path originating at $a_m'$ will be called the \textbf{greedy path collection} originating at $A$. Less formally, this is a non-intersecting path collection with source set $A$ where each path takes every vertical edge available to it without intersecting a path that originates above it.
\end{definition}

\begin{remark}\label{rem:lexmax}
    To any path collection $C$ (or path, thought of as a path collection with only one path) in $G$, we can associate a tuple $\epsilon(C)=(e_1,\ldots, e_{n-1})$ where $e_i$ is the number of vertical edges used by $C$ in column $i$ of $G$. Observe that the greedy path collection with source $S$ is equivalently the non-intersecting path collection $C$ which uniquely maximizes $\epsilon(C)$ in lexicographic order amongst all paths originating at $S$; Once we have determined the greedy path collection through columns $\{1,\ldots, i-1\}$, any path that can use a vertical edge in column $i$ without blocking another path of the collection, must do so. 
\end{remark}

 From our explicit description of the extremal indices in \Cref{topcellextremals} and \Cref{troptopcellextremals}, we can immediately make a few observations.
 
 \begin{proposition}\label{topcellcor}
 Let $I$ be an extremal index of a flag in $\mathcal{R}_{id, w_0}^{>0}$ (or, equivalently, of a point in $\textnormal{TrFl}_{id,w_0}^{>0}$). There is a unique collection originating from $[|I|]'$ with sink set $I$ in the graph $G_{id,w_0}$. It is the union of a diagonal path collection and a greedy path collection.

  Conversely, the sink set of any non-intersecting path collection in $G_{id,w_0}$ that is the union of a greedy path collection originating from the top $q$ vertices of $[|I|]'$ and a diagonal path collection originating from the remaining vertices of $[|I|]'$ for some $0\leq q\leq |I|$ is an extremal index.
 \end{proposition}
 
 \begin{proof}
By \Cref{topcellextremals}, we are interested in non-intersecting path collections in $G_{id,w_0}$ with origin set $\{1',2',\ldots, (m+n-k+1)'\}$ and sink set $\{1,2,\ldots, m-1,m, k,k+1,\ldots, n\}$ for $0\leq m<k\leq n$ (other than $m=0,k=n$). The only path which terminates at $1$ is the diagonal path originating at $1'$. Once we have added this path to our path collection, the only path which avoids it and terminates at $2$ is the diagonal path originating at $2'$. Continuing in this way, we obtain a unique non-intersecting path collection consisting of diagonal paths which terminates at $\{1,\ldots,m\}$. Note that $G_{id,w_0}$ is the graph which has the maximum number of vertical edges, that is, $n-c$ vertical edges in column $c$, as in \Cref{exampletopcellextremals}. Given that there must be a path terminating at $n$ and that we want our paths to be non-intersecting in the planar graph $G_{id,w_0}$, the path originating at the topmost source in our source set, namely $(m+n-k+1)'$, must terminate at $n$. Thus, we must include the unique path between these two vertices. This happens to be a greedy path. Similarly, once we have added this path to our collection, in order to have a non-intersecting path collection, the path terminating at $n-1$ must originate at the second source from the top, namely $(m+n-k)'$. There is a unique such path which does not intersect the path from $(m+n-k+1)'$ to $n$ and these two paths together form a greedy path collection. Continuing in this way we obtain the greedy path collection which originates at $\{(m+1)',\ldots, (m+n-k+1)'\}$ and terminates at $\{k,\ldots,n\}$. Putting these two collections together gives the desired unique non-intersecting path collection. See \Cref{exampletopcellextremals} for an example of this construction.

 Finally, to see the converse, observe that any non-intersecting path collection originating from $|I|'$ which consists of a greedy path collection on top of a diagonal path collection appears in this proof for appropriate values of $m$ and $k$.
  
 \end{proof}
  
     \begin{figure}[H]
 \begin{tikzpicture}[node distance={10.5 mm}, thick, main/.style = {draw, circle,minimum size=2 mm}, 
blank/.style={circle, draw=green!0, fill=green!0, very thin, minimum size=3.5mm},]

\node[main] (1) {$1'$};
\node[main] (2) [above of=1] {$2'$};
\node[main] (3) [above of = 2] {$3'$}; 
\node[main] (4) [above of=3] {$4'$};
\node[main](5) [above of = 4] {$5'$};
\node (blank)[above of = 5]{};
\node[main](55) [right of = blank]{$5$};
\node(54) [below of = 55]{};
\node(53) [below of = 54]{};
\node(52) [below of = 53]{};
\node(51) [below of = 52]{};
\node[main] (44) [right  of=55] {$4$};
\node(43) [below of = 44]{};
\node(42) [below of = 43]{};
\node(41) [below of = 42]{};
\node[main] (33) [right of = 44] {$3$};
\node(32) [below of = 33]{};
\node(31) [below of = 32]{};
\node[main] (22) [right of = 33] {$2$};
\node(21) [below of = 22]{};
\node[main] (11) [right of = 22] {$1$};

\draw[-{Latex[length=3mm]}] (5) -- (55);
\draw[](4)--([xshift=-5mm,yshift=-5mm]54.center);
\draw[-{Latex[length=3mm]}] ([xshift=-5mm,yshift=-5mm]54.center) -- (44);
\draw[dashed] (3) -- ([xshift=-5mm,yshift=-5mm]53.center);
\draw[] ([xshift=-5mm,yshift=-5mm]53.center) -- ([xshift=-5mm,yshift=-5mm]43.center);
\draw[-{Latex[length=3mm]}] ([xshift=-5mm,yshift=-5mm]43.center)--(33);
\draw[dashed] (2) -- ([xshift=-5mm,yshift=-5mm]52.center);
\draw[dashed] ([xshift=-5mm,yshift=-5mm]52.center) -- (42.center);
\draw[] ([xshift=-5mm,yshift=-5mm]42.center) -- ([xshift=-5mm,yshift=-5mm]32.center);
\draw[-{Latex[length=3mm]}] ([xshift=-5mm,yshift=-5mm]32.center) -- (22);
\draw[dashed] (1) -- ([xshift=-5mm,yshift=-5mm]51.center);
\draw[dashed] ([xshift=-5mm,yshift=-5mm]51.center) -- ([xshift=-5mm,yshift=-5mm]41.center);
\draw[-{Latex[length=3mm]},dashed] ([xshift=-5mm,yshift=-5mm]41.center)--(11);

\draw[-{Latex[length=3mm]},dashed] ([xshift=-5mm,yshift=-5mm]54.center) -- ([xshift=-5mm,yshift=-5mm]55.center);
\draw[-{Latex[length=3mm]},dashed] ([xshift=-5mm,yshift=-5mm]53.center) -- ([xshift=-5mm,yshift=-5mm]54.center);
\draw[-{Latex[length=3mm]}] ([xshift=-5mm,yshift=-5mm]52.center) -- ([xshift=-5mm,yshift=-5mm]53.center);
\draw[-{Latex[length=3mm]}] ([xshift=-5mm,yshift=-5mm]51.center) -- ([xshift=-5mm,yshift=-5mm]52.center);
\draw[-{Latex[length=3mm]},dashed] ([xshift=-5mm,yshift=-5mm]43.center) -- ([xshift=-5mm,yshift=-5mm]44.center);
\draw[-{Latex[length=3mm]},dashed] ([xshift=-5mm,yshift=-5mm]42.center) -- ([xshift=-5mm,yshift=-5mm]43.center);
\draw[-{Latex[length=3mm]}] ([xshift=-5mm,yshift=-5mm]41.center) -- ([xshift=-5mm,yshift=-5mm]42.center);
\draw[-{Latex[length=3mm]}] ([xshift=-5mm,yshift=-5mm]31.center) -- ([xshift=-5mm,yshift=-5mm]32.center);
\draw[-{Latex[length=3mm]}] ([xshift=-5mm,yshift=-5mm]32.center) -- ([xshift=-5mm,yshift=-5mm]33.center);
\draw[-{Latex[length=3mm]}] ([xshift=-5mm,yshift=-5mm]21.center) -- ([xshift=-5mm,yshift=-5mm]22.center);

\end{tikzpicture} 
\caption{A non-intersecting path collection, indicated by dashed edges, in $G_{id, w_0}$ for $n=5$ whose sink set is an extremal index. Note that this is the union of the diagonal path originating at $1'$ and the greedy path collection originating at $\{2',3'\}$.} 
\label{exampletopcellextremals}
 \end{figure}

\begin{cor}\label{monomialtopcellcor}
Let $P=\Phi_{id,w_0}(\bm{a})$. For each extremal index $I$ for $\mathcal{R}_{id,w_0}^{>0}$, the \Pl coordinate $P_I$ is a monomial in the $\{a_i\}$. 
\end{cor}
\begin{proof}
Since there is a unique non-intersecting path collection originating from $[|I|]'$ with sink set $I$ by \Cref{topcellcor}, this follows from \Cref{LGVcor}.
\end{proof}

\begin{definition}\label{def:topcellgraphextremal}
 We define a \textbf{graph extremal path collection} to be a path collection of the form described in \Cref{topcellcor}, that is, a non-intersecting path collection at $[k]'$ in $G_{id,w_0}$ for some $k$ whose sink set is an extremal index. 
\end{definition} 

\subsection{Determining Non-Extremal \Pl Coordinates}
 We now show that the extremal non-zero and non-infinite \Pl coordinates determine the rest of the \Pl coordinates in the totally positive flag variety and totally positive Dressian, respectively. Recall the three-term incidence \Pl relations from \Cref{rem:threetermrelations}.
 
 \begin{theorem}\label{topcell3termgenerate}
 For any flag $F$ in $\mathcal{R}_{id,w_0}^{> 0}$, the extremal non-zero \Pl coordinates of $F$ uniquely determine the other non-zero \Pl coordinates of $F$ by three-term incidence \Pl relations. 
 \end{theorem}
 
 \begin{proof}
 Let $P=P(F)$. Our goal is to determine $P_I$ for any $\emptyset\neq I\subset[n]$ using just the extremal non-zero \Pl coordinates and the three-term incidence \Pl relations. This proof will proceed by triple induction. First, we will work by reverse induction on the size of the index $I$. For subsets of a fixed size, we will use reverse induction on a statistic $t$, which we will define shortly. Finally, for subsets with equal size and equal values of $t$, we will use reverse induction on Gale order. 

We now define the statistic $t$ by the following property: Among all non-intersecting path collections from $[|I|]'$ to $I$, the non-diagonal path with the lowest source vertex originates at the source vertex which is $(t(I)+1)^\textnormal{st}$ from the bottom in $[|I|]'$. In particular, we have diagonal paths originating from the bottom-most $t(I)$ vertices of $[|I|]'$ in any such path collection. This is also the largest integer $\tau$ such that $[\tau]\subset I$. From the latter description, it is straightforward to verify that we may define $t$ in terms of any given non-intersecting path collection, rather than minimizing over all of them.
 
 All non-zero\footnote{\label{note1}The adjective non-zero is redundant here, but we will wish to make use of this proof for general $G_{v,w}$ later.} \Pl coordinates of size $n-1$ are extremal, as explained in \Cref{maxsizecoords}. Suppose $t(I)=|I|$. Any path collection from $[|I|]'$ to $I$ in $G_{id,w_0}$ must be a diagonal path collection. Since there are no arrows directed downwards in $G_{id,w_0}$, the diagonal path collection originating from $[|I|]'$ has Gale minimal sink set among all path collections originating from $[|I|]'$. Thus, $I$ is extremal. These serve as base cases for our induction.

Let $\emptyset\neq I\subset [n]$ with $|I|=k$ and $t(I)<|I|$ be such that $P_{I}\neq 0$. Let $C$ be a non-intersecting path collection from $[|I|]'$ to $I$ in $G_{id,w_0}$. By definition of $t$, the path $p$ in $C$ originating from the source which is $t(I)+1$ from the bottom of $[|I|]'$ is not diagonal. Say that $p$ originates on strand $a$ and terminates at sink vertex $b$. Since $p$ is not diagonal and there are no vertical edges directed downwards in $G_{id,w_0}$, $b>a$. Note that the path collection $D$ which is identical to $C$ except that we replace $p$ by a diagonal path $q$ terminating at sink vertex $a$ is non-intersecting; paths originating below $q$ in $D$ cannot intersect it, since they are diagonal, and paths originating above $q$ in $D$ cannot intersect it, since $q$ is diagonal and there are no vertical edges directed downwards in $G_{id,w_0}$. Thus, $P_{(I\setminus b)\cup a}\neq 0$. Also, since the indices of non-zero \Pl coordinates form a flag matroid, there exists $c\in[n]$ such that $c\notin (I\setminus b)\cup a$ and $P_{(I\setminus b)\cup ac}\neq 0$ by \Cref{def:flagmatroid}.  Choose $c$ to be minimal with this property. We consider three cases. 

First, suppose that $c>b$. Using the three-term incidence \Pl relation $P_IP_{(I\setminus b)\cup ac}=P_{(I\setminus b)\cup c}P_{I\cup a}+P_{(I\setminus b)\cup a}P_{I\cup c}$, we can determine $P_I$ as long as all other terms in this relation are already determined by induction. Each of $(I\setminus b)\cup ac$, $I\cup a$ and $I\cup c$ are of larger cardinality than $I$ and so the corresponding \Pl coordinates are already determined by induction. The remaining two terms have indices of the same size as $I$, so we move to our next layers of induction, first induction on $t$ and then reverse induction on Gale order. Recall that we defined $a$ to be the strand on which the lowest non-diagonal path of $C$ lies. Thus, $t((I\setminus b)\cup a)>t(I)$ and $P_{(I\setminus b)\cup a}$ is determined by induction. Finally, since the path terminating at $b$ is non-diagonal and since replacing it by $c>b$ does not affect the number of diagonal paths in our path collection, we have $t((I\setminus b)\cup c)= t(I)$. Additionally, $(I\setminus b)\cup c>I$ in Gale order, so $P_{(I\setminus b)\cup c}$ is determined by induction. 

Next, suppose that $c<a$. Since the path originating on strand $a$ in $C$ is the non-diagonal path with the lowest origin in $C$, and there are no edges directed downwards in $G_{id,w_0}$, $|I\cap [a-1]|=t(I)$, that is, there are precisely $t(I)$ paths (all of them diagonal) originating and terminating on or below strand $a-1$. In particular, precisely $t(I)$ vertices of $[|I|]'$ lie on or below strand $a-1$. We consider the three-term incidence \Pl relation $P_{(I\setminus b)\cup a}P_{I\cup c}=P_IP_{(I\setminus b)\cup ac}+P_{(I\setminus b)\cup c}P_{I\cup a}$. We first show that $P_{(I\setminus b)\cup c}=0$. Note that $|((I\setminus b)\cup c)\cap [a-1]|=t(I) +1$ whereas only $t(I)$ of the vertices in $[|I|]'$ lie weakly below strand $a-1$. Since there are no vertical edges directed downwards in $G_{id,w_0}$, it is impossible for any path collection to have more sinks than sources on the bottom $a-1$ strands. Thus, there are no path collections from $[|I|]'$ to ${(I\setminus b)\cup c}$ in $G_{id,w_0}$ and so $P_{(I\setminus b)\cup c}=0$. Since $P_{(I\setminus b)\cup ac}\neq0,$ it suffices to show that $P_{(I\setminus b)\cup ac}$, $P_{I\cup c}$, and $P_{(I\setminus b)\cup a}$ are already determined by the induction. The first two have indices of larger cardinality than $I$, and so are already determined by induction. The last is of the same cardinality, but as in the previous case, $t((I\setminus b)\cup a)>t(I)$ and so $P_{(I\setminus b)\cup a}$ is already determined by induction. 

Finally, we have the case where $a<c<b$. We begin by observing that in this case, $P_{I\cup a}\neq 0$. To see this, we use the three-term incidence \Pl relation $P_{(I\setminus b)\cup c}P_{I\cup a}=P_{I}P_{(I\setminus b)\cup ac}+P_{(I\setminus b)\cup a}P_{I\cup c}$. Note that by assumption, $P_I\neq 0$ and by choice of $c$, $P_{(I\setminus b)\cup ac}\neq 0$. Since all \Pl coordinates are nonnegative by \Cref{forwarddir}, we must have that $P_{(I\setminus b)\cup c}P_{I\cup a}\neq 0$ and so, $P_{I\cup a}\neq 0$. 

Next, observe that we may assume $I$ is not extremal since, if it were, there would be nothing to prove. Thus, $\Xi_{id,w_0}(I)=I\setminus b'\cup d'$ for some $b',d'\in [n]$ with $b'<d'$. 

We next show by contradiction that $b'>a$. Note that $b'\neq a$ since $b'\in I$ and $a\notin I$. Suppose $b'<a$. Then, $b'$ is the sink of a diagonal path in $C$. By definition of $\Xi_{id,w_0}$, this means that it is impossible to replace the sink of any path in $C$ which is not diagonal by something larger. In other words, the non-diagonal part of $C$ has a sink set which is Gale maximal among all non-intersecting path collections originating from the top $|I|-t(I)$ vertices of $[|I|]'$. Observe that in the planar graph $G_{id,w_0}$, greedy path collections achieve Gale maximal sink sets amongst all non-intersecting path collections with the same source set. Thus, the non-diagonal paths of $C$ achieve the same sink set as a greedy path collection. By \Cref{topcellcor}, this implies that the sink set $I$ of $C$ is extremal, contradicting our assumption that $I$ is not extremal. 

Thus, we have $a<b'<d'$. We then have the three-term incidence \Pl relation $P_{I}P_{(I\setminus b')\cup ad'}=P_{(I\setminus b')\cup a}P_{I\cup d'}+P_{(I\setminus b')\cup d'}P_{I\cup a}$. Recall that we began the $a<c<b$ case by showing that $P_{I\cup a}\neq 0$. Moreover, since $\Xi_{id,w_0}(I)=(I\setminus b')\cup d'$,  $P_{(I\setminus b')\cup d'}\neq 0$. Since, by \Cref{forwarddir}, all \Pl coordinates are nonnegative, we must have $P_{(I\setminus b')\cup ad'}\neq 0$. We are just left to show that all terms other than $P_I$ in this relation are already determined by induction. Observe that $|(I\setminus b')\cup ad'|=|I\cup d'|=|I\cup a|>|I|$. Also, while $|(I\setminus b')\cup a|=|I|$, we have $t((I\setminus b')\cup a)>t(I)$. Finally, while $|(I\setminus b')\cup d'|=|I|$ and $t((I\setminus b')\cup d')=t(I)$, we have $(I\setminus b')\cup d'>I$ in Gale order. Thus, all the corresponding \Pl coordinates are already determined by induction.

\end{proof}

\begin{cor}\label{postest}
 For any flag $F$ in $\mathcal{R}^{> 0}_{id,w_0}$, the extremal \Pl coordinates serve as a positivity test. Explicitly, if the extremal non-zero \Pl coordinates of $F$ are positive, then so are all the \Pl coordinates of $F$.
 \end{cor} 
 
 \begin{proof}
 Observe that whenever an unknown $P_I$ appears in an equation as described in the proof of \Cref{topcell3termgenerate}, it is uniquely determined in a subtraction free way by the other coordinates appearing in that equation.  
 \end{proof} 
 
 \begin{thmbis}{topcell3termgenerate}\label{tropicaltopcell3termgenerate}

 For any point $P$ in $\textnormal{TrFl}_{id,w_0}^{> 0}$, the extremal non-infinite \Pl coordinates of $P$ uniquely determine the other non-infinite \Pl coordinates of $P$ by three-term tropical incidence \Pl relations.

 \end{thmbis}

 \begin{proof}
 The tropicalization of the three-term relation $P_IP_{(I\setminus b)\cup ac}=P_{(I\setminus b)\cup c}P_{I\cup a}+P_{(I\setminus b)\cup a}P_{I\cup c}$ is $P_I+P_{(I\setminus b)\cup ac}=\min\{P_{(I\setminus b)\cup c}+P_{I\cup a},P_{(I\setminus b)\cup a}+P_{I\cup c}\}$. Observe that a coordinate appearing on the left hand side of this equation is uniquely determined by the other five coordinates appearing in the equation. Similarly, for any equality of monomials $P_{I_1}P_{I_2}=P_{I_3}P_{I_4}$, the tropical relation becomes $P_{I_1}+P_{I_2}=P_{I_3}+P_{I_4}$ and any one term can be uniquely determined by the other three. Since in the proof of \Cref{topcell3termgenerate}, unknowns $P_I$ always appear on the left hand side of a three-term incidence \Pl relation or in an equality of monomials, the proof carries through to the tropical case. 
 \end{proof}

\subsection{Determining the Parameters}

 Note that, for a flag $F\in \mathcal{R}_{id,w_0}^{>0}$, the extremal non-zero \Pl coordinates must be algebraically independent. To see this, observe that the dimension of $\mathcal{R}^{>0}_{id,w_0}$ is $\ell(w_0)-\ell(id)={\binom{n}{2}}$. Also, we have $N\coloneqq{\binom{n}{2}}+n$ extremal non-zero \Pl coordinates. We may ignore the $n$ extremal non-zero \Pl coordinates $P_{[k]}$, since we fix these to be $1$ in order to specify the projective scaling of our coordinates. Specifying the remaining $\binom{n}{2}$ many extremal non-zero \Pl coordinates determines a point in the $\binom{n}{2}$ dimensional space $\mathcal{R}_{id,w_0}^{>0}$, by \Cref{topcell3termgenerate}, and every point in $\mathcal{R}_{id,w_0}^{>0}$ can be determined in this way. Thus, they must be algebraically independent. Recall that $\Phi_{id,w_0}$ maps the weights $\bm{a}$ to the \Pl coordinates of the flag determined by $G_{id,w_0}(\bm a)$ via the LGV construction. Our next major goal is to show that the inverse $\Psi_{id,w_0}$ of $\Phi_{id,w_0}$ is well behaved. Our next result is a technical tool we will need for that purpose.

\begin{lemma}\label{lem:constructionstar}
    Let $v\leq w\in \mathfrak{S}_n$. Let $C$ be a non-intersecting path collection of $N$ paths in $G_{v,w}$ with source set $O$ which is the union of a greedy path collection originating weakly above strand $q$ and a diagonal path collection originating from the rest of $O$ for some $q$. Let $S$ denote the sink set of $C$. For any $o\in O$, let $D$ be the path collection with source set $O\setminus o$ which is the union of a greedy path collection originating weakly above strand $q$ and a diagonal path collection originating from the rest of $O\setminus o$. Then, away from vertical edges of $G_{v,w}$, $D$ is locally identical to $C$ with a single path removed. Accordingly, the sink set of $D$ is $S\setminus \kappa$ for some $\kappa$. 
\end{lemma}
\begin{proof}
The second claim of the Lemma is immediate from the first claim. We prove the first claim. 

Let $D^{(0)}$ be the path collection identical to $C$ but with the path originating on $o$ removed. If the path originating on $o$ were diagonal, we are done, as $D=D^{(0)}$ has the desired properties. Otherwise, the path originating on strand $o$ is in the greedy part of $C$. In this case, we will inductively define a sequence of non-intersecting path collections $D^{(0)},D^{(1)},\cdots, D^{(z)}$ that interpolate between $D^{(0)}$ and a path collection $D=D^{(z)}$ with the desired properties. We begin with a base step. 

If $D^{(0)}$ has the desired properties, we are done. Otherwise, it must be the case that the non-diagonal paths of $D^{(0)}$ do not form a greedy path collection. There must be a path $q_1$ in the non-diagonal part of $D^{(0)}$ which approaches the origin of a vertical edge $e_1$ along a strand but does not use $e_1$, even though there is no other path of $D^{(0)}$ using the strand at the terminus of $e_1$ and blocking $q_1$ from using $e_1$. Assume $e_1$ is leftmost with this property and that it lies in column $c_1$. This can only happen once in column $c_1$ since to form $D^{(0)}$ from $C$, we only removed a single path. Since the non-diagonal part of $C$ is greedy, it must be the case that some path $p'_1$ of $C$ used the strand at the terminus of $e_1$. Note that, by the inductive construction of greedy path collections in \Cref{defleftgreedy}, $q_1$ must originate below $p_1'$. Let $D^{(1)}$ be the non-intersecting path collection obtained from $D^{(0)}$ by replacing $q_1$ by the path $\hat{q}_1$, also originating on strand $\sigma_1$, which is identical to $q_1$ until reaching $e_1$, uses edge $e_1$, and then finishes identically to $p'_1$. For an example, see \Cref{fig:removingapath}. Note that, immediately to the left of column $c_1$, $D^{(1)}$ is identical to $C$ with a path removed and, to the right of column $c_1$, $D^{(1)}$ is identical to $C$ but with a different path removed. If the non-diagonal part of $D^{(1)}$ is greedy, $D=D^{(1)}$ has the desired properties. Otherwise, we iterate this argument, as follows.

We establish some inductive hypotheses: Suppose we constructed $D^{(i)}$ from $D^{(i-1)}$ by (1) identifying the unique edge $e_{i}$ in column $c_{i}$ where there was a path of $D^{(i-1)}$ which used the strand at the origin of edge $e_{i}$ but did not use edge $e_{i}$ itself, despite the fact that there is no path of $D^{(i-1)}$ using the strand at the terminus of $e_{i}$ and (2) modifying a path $q_{i}$ of $D^{(i-1)}$ into a new path $\hat{q}_i$ which is identical to $q_i$ to the left of column $c_i$ and uses edge $e_i$ in column $c_i$. Further suppose that to the right of column $c_i$, $D^{(i)}$ is identical to $C$ with some path $p'_{i+1}$ removed. Finally, suppose the non-diagonal part of $D^{(i)}$ is not greedy. With these assumptions, we construct $D^{(i+1)}$. Find the leftmost column $c_{i+1}$ where there is a path $q_{i+1}$ originating on some strand $\sigma_{i+1}$ in $D^{(i)}$ that uses the strand at the bottom of a vertical edge $e_{i+1}$ and no path in $D^{(i)}$ that uses the strand at the top of $e_{i+1}$. Column $c_{i+1}$ is necessarily to the right of column $c_{i}$ since in going from $D^{(i-1)}$ to $D^{(i)}$, we do not alter anything to the left of column $c_i$ and we observed that $e_i$ was the unique edge with its defining properties in column $c_i$. The edge $e_{i+1}$ is the unique edge in column $c_{i+1}$ with these properties since, to the right of $c_i$, $D^{(i)}$ is identical to $C$ with a path $p_{i+1}'$ removed. Since $C$ is greedy, $p'_{i+1}$ must use the strand at the terminus of $e_{i+1}$. Let $D^{(i+1)}$ be the non-intersecting path collection obtained from $D^{(i)}$ by replacing $q_{i+1}$ by the path $\hat{q}_{i+1}$, also originating on strand $\sigma_{i+1}$, which is identical to $q_{i+1}$ until reaching $e_{i+1}$, uses edge $e_{i+1}$, and then finishes identically to $p'_{i+1}$. Note that to the right of column $c_{i+1}$, $D^{(i+1)}$ is identical to $C$ with a path $p'_{i+2}\neq p'_{i+1}$ removed. Also observe that by the inductive construction of greedy path collections in \Cref{defleftgreedy}, we must have that $p'_{i+2}$ originates below $p'_{i+1}$. If the non-diagonal part of $D^{(i+1)}$ is greedy, then $D=D^{(i+1)}$ has the desired properties and we are done. Otherwise, we satisfy all the inductive hypotheses laid out at the beginning of this paragraph and we can continue to iterate this argument. Since $p'_{i+1}$ originates below $p'_i$ for each $i$, this process must eventually terminate.

\end{proof}

\begin{remark}\label{rem:constructionstarproof}
    We record here the observation that $p'_{i+1}$ originates below $p'_i$ for each $i$. This also implies that each $\sigma_i$ is distinct. We will use these facts later.
\end{remark}

     \begin{figure}[H]
 \begin{tikzpicture}[node distance={10.5 mm}, thick, main/.style = {draw, circle,minimum size=2 mm}, 
blank/.style={circle, draw=green!0, fill=green!0, very thin, minimum size=3.5mm},]

\node[main] (1) {$1'$};
\node[main] (2) [above of=1] {$2'$};
\node[main] (3) [above of = 2] {$3'$}; 
\node[main] (4) [above of=3] {$4'$};
\node[main](5) [above of = 4] {$5'$};
\node (blank)[above of = 5]{};
\node[main](55) [right of = blank]{$5$};
\node(54) [below of = 55]{};
\node(53) [below of = 54]{};
\node(52) [below of = 53]{};
\node(51) [below of = 52]{};
\node[main] (44) [right  of=55] {$4$};
\node(43) [below of = 44]{};
\node(42) [below of = 43]{};
\node(41) [below of = 42]{};
\node[main] (33) [right of = 44] {$3$};
\node(32) [below of = 33]{};
\node(31) [below of = 32]{};
\node[main] (22) [right of = 33] {$2$};
\node(21) [below of = 22]{};
\node[main] (11) [right of = 22] {$1$};

\draw[] (5) -- ([xshift=-5mm,yshift=-5mm]55.center);
\draw[dashed,-{Latex[length=3mm]}]([xshift=-5mm,yshift=-5mm]55.center)--(55);
\draw[dashed](4)--([xshift=-5mm,yshift=-5mm]54.center);
\draw[] ([xshift=-5mm,yshift=-5mm]54.center) -- ([xshift=-5mm,yshift=-5mm]44.center);
\draw[-{Latex[length=3mm]}] ([xshift=-5mm,yshift=-5mm]44.center) -- (44);
\draw[dashed] (3) -- ([xshift=-5mm,yshift=-5mm]53.center);
\draw[dashed] ([xshift=-5mm,yshift=-5mm]53.center) -- ([xshift=-5mm,yshift=-5mm]43.center);
\draw[] ([xshift=-5mm,yshift=-5mm]43.center) -- ([xshift=-5mm,yshift=-5mm]33.center);
\draw[-{Latex[length=3mm]}] ([xshift=-5mm,yshift=-5mm]33.center)--(33);
\draw[dashed] (2) -- ([xshift=-5mm,yshift=-5mm]42.center);
\draw[dashed] ([xshift=-5mm,yshift=-5mm]42.center) -- ([xshift=-5mm,yshift=-5mm]32.center);
\draw[-{Latex[length=3mm]}] ([xshift=-5mm,yshift=-5mm]32.center) -- (22);
\draw[dashed] (1) -- ([xshift=-5mm,yshift=-5mm]51.center);
\draw[dashed] ([xshift=-5mm,yshift=-5mm]51.center) -- ([xshift=-5mm,yshift=-5mm]41.center);
\draw[-{Latex[length=3mm]},dashed] ([xshift=-5mm,yshift=-5mm]41.center)--(11);

\draw[-{Latex[length=3mm]},dashed] ([xshift=-5mm,yshift=-5mm]54.center) -- ([xshift=-5mm,yshift=-5mm]55.center);
\draw[-{Latex[length=3mm]}] ([xshift=-5mm,yshift=-5mm]53.center) -- ([xshift=-5mm,yshift=-5mm]54.center);
\draw[-{Latex[length=3mm]}] ([xshift=-5mm,yshift=-5mm]52.center) -- ([xshift=-5mm,yshift=-5mm]53.center);
\draw[-{Latex[length=3mm]}] ([xshift=-5mm,yshift=-5mm]51.center) -- ([xshift=-5mm,yshift=-5mm]52.center);
\draw[-{Latex[length=3mm]},dashed] ([xshift=-5mm,yshift=-5mm]43.center) -- ([xshift=-5mm,yshift=-5mm]44.center);
\draw[-{Latex[length=3mm]}] ([xshift=-5mm,yshift=-5mm]42.center) -- ([xshift=-5mm,yshift=-5mm]43.center);
\draw[-{Latex[length=3mm]}] ([xshift=-5mm,yshift=-5mm]41.center) -- ([xshift=-5mm,yshift=-5mm]42.center);
\draw[-{Latex[length=3mm]}] ([xshift=-5mm,yshift=-5mm]31.center) -- ([xshift=-5mm,yshift=-5mm]32.center);
\draw[-{Latex[length=3mm]},dashed] ([xshift=-5mm,yshift=-5mm]32.center) -- ([xshift=-5mm,yshift=-5mm]33.center);
\draw[-{Latex[length=3mm]}] ([xshift=-5mm,yshift=-5mm]21.center) -- ([xshift=-5mm,yshift=-5mm]22.center);

\end{tikzpicture} 
 \begin{tikzpicture}[node distance={10.5 mm}, thick, main/.style = {draw, circle,minimum size=2 mm}, 
blank/.style={circle, draw=green!0, fill=green!0, very thin, minimum size=3.5mm},]

\node[main] (1) {$1'$};
\node[main] (2) [above of=1] {$2'$};
\node[main] (3) [above of = 2] {$3'$}; 
\node[main] (4) [above of=3] {$4'$};
\node[main](5) [above of = 4] {$5'$};
\node (blank)[above of = 5]{};
\node[main](55) [right of = blank]{$5$};
\node(54) [below of = 55]{};
\node(53) [below of = 54]{};
\node(52) [below of = 53]{};
\node(51) [below of = 52]{};
\node[main] (44) [right  of=55] {$4$};
\node(43) [below of = 44]{};
\node(42) [below of = 43]{};
\node(41) [below of = 42]{};
\node[main] (33) [right of = 44] {$3$};
\node(32) [below of = 33]{};
\node(31) [below of = 32]{};
\node[main] (22) [right of = 33] {$2$};
\node(21) [below of = 22]{};
\node[main] (11) [right of = 22] {$1$};

\draw[] (5) -- ([xshift=-5mm,yshift=-5mm]55.center);
\draw[-{Latex[length=3mm]}]([xshift=-5mm,yshift=-5mm]55.center)--(55);
\draw[](4)--([xshift=-5mm,yshift=-5mm]54.center);
\draw[] ([xshift=-5mm,yshift=-5mm]54.center) -- ([xshift=-5mm,yshift=-5mm]44.center);
\draw[-{Latex[length=3mm]}] ([xshift=-5mm,yshift=-5mm]44.center) -- (44);
\draw[dashed] (3) -- ([xshift=-5mm,yshift=-5mm]43.center);
\draw[] ([xshift=-5mm,yshift=-5mm]43.center) -- ([xshift=-5mm,yshift=-5mm]33.center);
\draw[-{Latex[length=3mm]}] ([xshift=-5mm,yshift=-5mm]33.center)--(33);
\draw[dashed] (2) -- ([xshift=-5mm,yshift=-5mm]32.center);
\draw[-{Latex[length=3mm]}] ([xshift=-5mm,yshift=-5mm]32.center) -- (22);
\draw[dashed] (1) -- ([xshift=-5mm,yshift=-5mm]51.center);
\draw[dashed] ([xshift=-5mm,yshift=-5mm]51.center) -- ([xshift=-5mm,yshift=-5mm]41.center);
\draw[-{Latex[length=3mm]},dashed] ([xshift=-5mm,yshift=-5mm]41.center)--(11);

\draw[-{Latex[length=3mm]}] ([xshift=-5mm,yshift=-5mm]54.center) -- ([xshift=-5mm,yshift=-5mm]55.center);
\draw[-{Latex[length=3mm]}] ([xshift=-5mm,yshift=-5mm]53.center) -- ([xshift=-5mm,yshift=-5mm]54.center) node [pos=0.6, right] {$e_1$};
\draw[-{Latex[length=3mm]}] ([xshift=-5mm,yshift=-5mm]52.center) -- ([xshift=-5mm,yshift=-5mm]53.center);
\draw[-{Latex[length=3mm]}] ([xshift=-5mm,yshift=-5mm]51.center) -- ([xshift=-5mm,yshift=-5mm]52.center);
\draw[-{Latex[length=3mm]},dashed] ([xshift=-5mm,yshift=-5mm]43.center) -- ([xshift=-5mm,yshift=-5mm]44.center);
\draw[-{Latex[length=3mm]}] ([xshift=-5mm,yshift=-5mm]42.center) -- ([xshift=-5mm,yshift=-5mm]43.center);
\draw[-{Latex[length=3mm]}] ([xshift=-5mm,yshift=-5mm]41.center) -- ([xshift=-5mm,yshift=-5mm]42.center);
\draw[-{Latex[length=3mm]}] ([xshift=-5mm,yshift=-5mm]31.center) -- ([xshift=-5mm,yshift=-5mm]32.center);
\draw[-{Latex[length=3mm]},dashed] ([xshift=-5mm,yshift=-5mm]32.center) -- ([xshift=-5mm,yshift=-5mm]33.center);
\draw[-{Latex[length=3mm]}] ([xshift=-5mm,yshift=-5mm]21.center) -- ([xshift=-5mm,yshift=-5mm]22.center);

\end{tikzpicture} 

\vspace{0.7 cm}

\begin{tikzpicture}[node distance={10.5 mm}, thick, main/.style = {draw, circle,minimum size=2 mm}, 
blank/.style={circle, draw=green!0, fill=green!0, very thin, minimum size=3.5mm},]

\node[main] (1) {$1'$};
\node[main] (2) [above of=1] {$2'$};
\node[main] (3) [above of = 2] {$3'$}; 
\node[main] (4) [above of=3] {$4'$};
\node[main](5) [above of = 4] {$5'$};
\node (blank)[above of = 5]{};
\node[main](55) [right of = blank]{$5$};
\node(54) [below of = 55]{};
\node(53) [below of = 54]{};
\node(52) [below of = 53]{};
\node(51) [below of = 52]{};
\node[main] (44) [right  of=55] {$4$};
\node(43) [below of = 44]{};
\node(42) [below of = 43]{};
\node(41) [below of = 42]{};
\node[main] (33) [right of = 44] {$3$};
\node(32) [below of = 33]{};
\node(31) [below of = 32]{};
\node[main] (22) [right of = 33] {$2$};
\node(21) [below of = 22]{};
\node[main] (11) [right of = 22] {$1$};

\draw[] (5) -- ([xshift=-5mm,yshift=-5mm]55.center);
\draw[-{Latex[length=3mm]}]([xshift=-5mm,yshift=-5mm]55.center)--(55);
\draw[](4)--([xshift=-5mm,yshift=-5mm]54.center);
\draw[] ([xshift=-5mm,yshift=-5mm]54.center) -- ([xshift=-5mm,yshift=-5mm]44.center);
\draw[-{Latex[length=3mm]}] ([xshift=-5mm,yshift=-5mm]44.center) -- (44);
\draw[dashed] (3) -- ([xshift=-5mm,yshift=-5mm]53.center);
\draw[] ([xshift=-5mm,yshift=-5mm]53.center) -- ([xshift=-5mm,yshift=-5mm]33.center);
\draw[-{Latex[length=3mm]}] ([xshift=-5mm,yshift=-5mm]33.center)--(33);
\draw[dashed] (2) -- ([xshift=-5mm,yshift=-5mm]32.center);
\draw[-{Latex[length=3mm]}] ([xshift=-5mm,yshift=-5mm]32.center) -- (22);
\draw[dashed] (1) -- ([xshift=-5mm,yshift=-5mm]51.center);
\draw[dashed] ([xshift=-5mm,yshift=-5mm]51.center) -- ([xshift=-5mm,yshift=-5mm]41.center);
\draw[-{Latex[length=3mm]},dashed] ([xshift=-5mm,yshift=-5mm]41.center)--(11);

\draw[-{Latex[length=3mm]},dashed] ([xshift=-5mm,yshift=-5mm]54.center) -- ([xshift=-5mm,yshift=-5mm]55.center);
\draw[-{Latex[length=3mm]},dashed] ([xshift=-5mm,yshift=-5mm]53.center) -- ([xshift=-5mm,yshift=-5mm]54.center) node [pos=0.6, right] {$e_1$};
\draw[-{Latex[length=3mm]}] ([xshift=-5mm,yshift=-5mm]52.center) -- ([xshift=-5mm,yshift=-5mm]53.center);
\draw[-{Latex[length=3mm]}] ([xshift=-5mm,yshift=-5mm]51.center) -- ([xshift=-5mm,yshift=-5mm]52.center);
\draw[-{Latex[length=3mm]}] ([xshift=-5mm,yshift=-5mm]43.center) -- ([xshift=-5mm,yshift=-5mm]44.center);
\draw[-{Latex[length=3mm]}] ([xshift=-5mm,yshift=-5mm]42.center) -- ([xshift=-5mm,yshift=-5mm]43.center)node [pos=0.6, right] {$e_2$};
\draw[-{Latex[length=3mm]}] ([xshift=-5mm,yshift=-5mm]41.center) -- ([xshift=-5mm,yshift=-5mm]42.center);
\draw[-{Latex[length=3mm]}] ([xshift=-5mm,yshift=-5mm]31.center) -- ([xshift=-5mm,yshift=-5mm]32.center);
\draw[-{Latex[length=3mm]},dashed] ([xshift=-5mm,yshift=-5mm]32.center) -- ([xshift=-5mm,yshift=-5mm]33.center);
\draw[-{Latex[length=3mm]}] ([xshift=-5mm,yshift=-5mm]21.center) -- ([xshift=-5mm,yshift=-5mm]22.center);

\end{tikzpicture} 
\begin{tikzpicture}[node distance={10.5 mm}, thick, main/.style = {draw, circle,minimum size=2 mm}, 
blank/.style={circle, draw=green!0, fill=green!0, very thin, minimum size=3.5mm},]

\node[main] (1) {$1'$};
\node[main] (2) [above of=1] {$2'$};
\node[main] (3) [above of = 2] {$3'$}; 
\node[main] (4) [above of=3] {$4'$};
\node[main](5) [above of = 4] {$5'$};
\node (blank)[above of = 5]{};
\node[main](55) [right of = blank]{$5$};
\node(54) [below of = 55]{};
\node(53) [below of = 54]{};
\node(52) [below of = 53]{};
\node(51) [below of = 52]{};
\node[main] (44) [right  of=55] {$4$};
\node(43) [below of = 44]{};
\node(42) [below of = 43]{};
\node(41) [below of = 42]{};
\node[main] (33) [right of = 44] {$3$};
\node(32) [below of = 33]{};
\node(31) [below of = 32]{};
\node[main] (22) [right of = 33] {$2$};
\node(21) [below of = 22]{};
\node[main] (11) [right of = 22] {$1$};

\draw[] (5) -- ([xshift=-5mm,yshift=-5mm]55.center);
\draw[-{Latex[length=3mm]}]([xshift=-5mm,yshift=-5mm]55.center)--(55);
\draw[](4)--([xshift=-5mm,yshift=-5mm]54.center);
\draw[] ([xshift=-5mm,yshift=-5mm]54.center) -- ([xshift=-5mm,yshift=-5mm]44.center);
\draw[-{Latex[length=3mm]}] ([xshift=-5mm,yshift=-5mm]44.center) -- (44);
\draw[dashed] (3) -- ([xshift=-5mm,yshift=-5mm]53.center);
\draw[] ([xshift=-5mm,yshift=-5mm]53.center) -- ([xshift=-5mm,yshift=-5mm]33.center);
\draw[-{Latex[length=3mm]}] ([xshift=-5mm,yshift=-5mm]33.center)--(33);
\draw[dashed] (2) -- ([xshift=-5mm,yshift=-5mm]42.center);
\draw[-{Latex[length=3mm]}] ([xshift=-5mm,yshift=-5mm]42.center) -- (22);
\draw[dashed] (1) -- ([xshift=-5mm,yshift=-5mm]51.center);
\draw[dashed] ([xshift=-5mm,yshift=-5mm]51.center) -- ([xshift=-5mm,yshift=-5mm]41.center);
\draw[-{Latex[length=3mm]},dashed] ([xshift=-5mm,yshift=-5mm]41.center)--(11);

\draw[-{Latex[length=3mm]},dashed] ([xshift=-5mm,yshift=-5mm]54.center) -- ([xshift=-5mm,yshift=-5mm]55.center);
\draw[-{Latex[length=3mm]},dashed] ([xshift=-5mm,yshift=-5mm]53.center) -- ([xshift=-5mm,yshift=-5mm]54.center) node [pos=0.6, right] {$e_1$};
\draw[-{Latex[length=3mm]}] ([xshift=-5mm,yshift=-5mm]52.center) -- ([xshift=-5mm,yshift=-5mm]53.center);
\draw[-{Latex[length=3mm]}] ([xshift=-5mm,yshift=-5mm]51.center) -- ([xshift=-5mm,yshift=-5mm]52.center);
\draw[-{Latex[length=3mm]},dashed] ([xshift=-5mm,yshift=-5mm]43.center) -- ([xshift=-5mm,yshift=-5mm]44.center);
\draw[-{Latex[length=3mm]},dashed] ([xshift=-5mm,yshift=-5mm]42.center) -- ([xshift=-5mm,yshift=-5mm]43.center)node [pos=0.6, right] {$e_2$};
\draw[-{Latex[length=3mm]}] ([xshift=-5mm,yshift=-5mm]41.center) -- ([xshift=-5mm,yshift=-5mm]42.center);
\draw[-{Latex[length=3mm]}] ([xshift=-5mm,yshift=-5mm]31.center) -- ([xshift=-5mm,yshift=-5mm]32.center);
\draw[-{Latex[length=3mm]}] ([xshift=-5mm,yshift=-5mm]32.center) -- ([xshift=-5mm,yshift=-5mm]33.center);
\draw[-{Latex[length=3mm]}] ([xshift=-5mm,yshift=-5mm]21.center) -- ([xshift=-5mm,yshift=-5mm]22.center);

\end{tikzpicture} \caption{The top left subfigure shows a path collection $C$ (dashed) which is a union of diagonal paths and a greedy path collection in $G_{id,w_0}$ for $n=5$. The next subfigures show how to construct $D$ from $C$ for $o=4'$, step by step. The top right shows the path collection $D^{(0)}$ obtained by deleting the path originating at $4'$ from $C$. It also indicates the edge $e_1$ in column $c_1=1$ which has a path using the strand at its origin but no path using the strand at its terminus in $D^{(0)}$. The bottom left subfigure shows the path collection $D^{(1)}$ where we replace the path $q_1$ originating on strand $\sigma_1=3$ in $D^{(0)}$ by a path $\hat{q}_1$ that uses $e_1$ and then terminates identically to the path $p_1'$ originating on strand $4$ in $C$. It also indicates the edge $e_2$. The bottom right subfigure shows the path collection $D^{(2)}$ where we replace the path $q_2$ originating on strand $\sigma_2=2$ in $D^{(1)}$ by a path $\hat{q}_2$ that uses $e_2$ and then terminates identically to the path $p_2'$ originating at $3'$ in $C$. The path collection $D^{(2)}$ is a union of diagonal paths and a greedy path collection. Its sink set is the same as the sink set of $C$ but with $3$ removed, as guaranteed by \Cref{lem:constructionstar}.} 
\label{fig:removingapath}
 \end{figure}

\begin{proposition}\label{prop:oneextraweight}
    Let $P=\Phi_{id,w_0}(\bm{a})$. Let $N$ be the number of extremal indices of a flag in $\mathcal{R}_{id,w_0}^{>0}$. Then, there are sets $A_i\subset[N]$ and a total order $\prec$ on the extremal indices $\{E_1\prec \cdots \prec E_N\}$ for $\mathcal{R}_{id,w_0}^{>0}$ such that $P_{E_{i}}=\prod_{k\in A_i}a_k$ and $\left|A_i\setminus \left(\bigcup_{j<i}A_j\right)\right|\leq 1$.
\end{proposition}

\begin{proof}
   The fact that each $P_{E_i}$ can be expressed as $P_{E_{i}}=\prod_{k\in A_i}a_k$ for some $A_i$ is a direct restatement of \Cref{monomialtopcellcor}. 

   Note that by \Cref{topcellcor}, each extremal index corresponds to a unique extremal path collection in $G_{id,w_0}$. Using this correspondence, the proposition is equivalent to the following statement: There exists an order $\prec$ on the extremal path collections such that if one goes through the extremal path collections in $\prec$ order, each one uses at most a single vertical edge which was not used previously. 

    We now define the total order $\prec$. If $|E_i|>|E_j|,$ then $E_i\prec E_j$. For extremal indices of the same size, $\prec$ is the lexicographic order. Equivalently, $E_i\prec \Xi_{id,w_0}(E_i)$. By abuse of notation, we will use the order $\prec$ on both indices and the corresponding unique non-intersecting path collections. 
   
   For $\alpha\in [N]$, we now show that $\left|A_\alpha\setminus \left(\bigcup_{\beta<\alpha}A_\beta\right)\right|\leq 1$. Let $C_\alpha$ be the extremal non-intersecting path collection with sink set $E_\alpha$. Suppose $C_\alpha$ consists of $k$ paths. If $C_\alpha$ were diagonal, we would have $A_\alpha=\emptyset$, so we may assume that $E_\alpha$ is not the Gale minimal extremal index of size $k$. As a result, $C_\alpha$ uses the same number of paths as the path collection $C_{\alpha-1}$ immediately preceding it in $\prec$ order. Suppose the bottom-most non-diagonal path of $C_\alpha$ originates on strand $\sigma$. Denote this path by $p$. The set $A_\alpha\setminus A_{\alpha-1}$ must index the weights on edges used by $p$, since all non-diagonal paths in $C_\alpha$ other than $p$ appear in $C_{\alpha-1}$ as well. 

Consider the non-intersecting path collection $C_{\alpha'}\prec C_\alpha$ which has $k+1$ paths and which is the union of a greedy path collection originating on or above strand $\sigma$ and a diagonal path collection originating below strand $\sigma$. Let $p'$ be the path originating on strand $\sigma$ in $C_{\alpha'}$. We will now analyze the difference between the greedy part of $C_{\alpha'}$ and that of $C_{\alpha}$ in order to show that $p$ uses at most a single vertical edge which is not used by $C_{\alpha'}$, proving that $|A_\alpha\setminus (A_{\alpha-1}\cup A_{{\alpha'}})|\leq 1$, which suffices to prove the proposition.

If the source vertex $(k+1)'$ lies below strand $\sigma$, the greedy part of $C_{\alpha'}$ is identical to the greedy part of $C_\alpha$ and so $p'$ is identical to $p$\footnote{This never happens in $G_{id,w_0}$, where $(k+1)'$ lies above all of $[k]'$. However for general $G_{v,w}$, this may occur, and we will want to reuse this proof.}. Otherwise, the vertex $(k+1)'$ lies above strand $\sigma$. It is still possible that $p=p'$, in which case we have nothing to show. Thus, we assume $p\neq p'$. We first determine the set of vertical edges used by $C_\alpha$ but not by $C_{\alpha'}$, and then show that at most one of these edges is used by $p$.

Note that $C_{\alpha}$ is precisely the path $D$ obtained by applying the construction of \Cref{lem:constructionstar} to the non-intersecting path collection $C=C_{\alpha'}$, with $o=(k+1)'$. The vertical edges used by $C_{\alpha}$ but not by $C_{\alpha'}$ are precisely the edges $\{e_1,\ldots, e_d\}$, in the notation of the proof of \Cref{lem:constructionstar}. We continue to use the notation of that proof. When we construct the path collection $D^{(i)}$, we replace a path $q_i$, originating on strand $\sigma_i$, which does not use any of the edges in $\{e_i, e_{i+1},\ldots, e_d\}$ by a path $\hat{q}_i$, also originating on strand $\sigma_i$, which uses the edge $e_i$ but not any of $\{e_{i+1},e_{i+2},\ldots, e_d\}$. Moreover, by \Cref{rem:constructionstarproof}, $\sigma_i\neq \sigma_j$ whenever $i\neq j$. Therefore, at the point in the construction where $\sigma_i=\sigma$, $p'$ gets replaced by $p$, which uses exactly one of the edges in $\{e_1,\ldots, e_d\}$, namely, $e_i$. 

\end{proof}

 \begin{cor}\label{inverseLaurent}
The map $\Psi_{id,w_0}$, which is inverse to $\Phi_{id,w_0}$ on $\mathcal{R}_{id,w_0}^{>0}$, can be expressed as Laurent monomials in the extremal non-zero \Pl coordinates of a flag in $\mathcal{R}^{>0}_{id,w_0}$. 
\end{cor}

\begin{proof}

We use the notation of \Cref{prop:oneextraweight}. For each $i\in [N]$, let ${k_i}$ be the unique element of $A_i\setminus \left(\bigcup_{j<i}A_j\right)$ if this set is nonempty and let $k_i$ be $0$ otherwise. We let $a_0=1$ so that $a_{k_i}$ is defined for all $i\in [N]$. We now show by induction that we can express $a_{k_i}$ as a Laurent monomial in $\{P_{E_1},\ldots P_{E_i}\}$. 

As a base case, $E_1$ is Gale minimal. Thus, the corresponding extremal path collection is a diagonal path collection and $A_1=\emptyset$. It follows that $a_{k_1}=1$ is the empty Laurent monomial in the extremal \Pl coordinates. 

For the induction step, suppose that for $l<i$, $a_{k_l}$ is expressible as a Laurent monomial in $\{P_{E_{1}},\ldots, P_{E_{l}}\}$. If $a_{k_i}=1$, there is nothing to show. Otherwise, \Cref{prop:oneextraweight} gives a way to express $a_{k_i}$ as a Laurent monomial in $P_{E_i}$ and $\{a_{k_j}\mid j<i\}$. By induction, this gives a way to express $a_{k_i}$ as a Laurent monomial in $\{P_{E_{1}},\ldots, P_{E_{i}}\}$. 

By \Cref{topcell3termgenerate}, the extremal \Pl coordinates determine all other \Pl coordinates. Thus, they must also determine the weights $a_i$ appearing in $G_{id,w_0}(\mathbf{a})$. Accordingly, every weight $a_i$ appears in some extremal path collection and so all the weights are determined as Laurent monomials in the extremal non-zero \Pl coordinates by the induction in the previous paragraph.

\end{proof}

See \Cref{monomialsidw_0} for a concrete example of this result. 

\begin{propbis}{inverseLaurent}\label{tropinverse}
 The map $\Trop\Psi_{id,w_0}$ is inverse to $\Trop\Phi_{id,w_0}$ and can be written in terms of sums and differences of extremal \Pl coordinates.
\end{propbis}

\begin{proof}
By \Cref{inverseLaurent}, the map $\Psi_{id,w_0}$ expresses the weights $a_i$ as Laurent monomials in the extremal non-zero \Pl coordinates of $F$. By \Cref{topcellcor}, the map $\Phi_{id,w_0}$ expresses the extremal \Pl coordinates as monomials in $\bm{a}$. We know that $\Phi_{id,w_0}$ and $\Psi_{id,w_0}$ are inverses. Since tropicalizing converts products and quotients to sums and differences, respectively, we will have that $\Trop \Psi_{id,w_0}$ can be written as sums and differences of extremal \Pl coordinates and, moreover, that $\Trop \Psi_{id,w_0}$ and $\Trop \Phi_{id,w_0}$ are inverses.
\end{proof}

\begin{remark}Based on the preceding propositions, we can apply $\Psi_{id,w_{0}}$, thought of as a collection of Laurent monomials in the extremal non-zero \Pl coordinates, to any point whose extremal non-zero \Pl coordinates coincide with the extremal non-zero \Pl coordinates of a flag in $\mathcal{R}_{id,w_0}^{>0}$. Similarly, we can apply $\Trop\Psi_{id,w_{0}}$, thought of as a collection of sums and differences in the extremal non-infinite \Pl coordinates, to any point whose extremal non-infinite \Pl coordinates coincide with the extremal non-infinite \Pl coordinates of a point in $\textnormal{TrFl}_{id,w_0}^{>0}$.

\end{remark}
 
We now have all the tools we need to prove our main result of this section. The key insight is that, by the preceding remark, a flag $F$ with all \Pl coordinates positive can be mapped by $\Psi_{id,w_0}$ to a weight vector $\bm{a}$. From there, it is not difficult to show that $F=\Phi_{id,w_0}(\bm{a})\in \mathcal{R}_{id,w_0}^{>0}$. 
 
 \begin{theorem}\label{topcellnonnegPluckCoords}
The totally positive complete flag variety $\mathcal{R}^{> 0}_{id,w_0}$ equals the set $\{F\in \textnormal{Fl}_n|\;P_I(F)> 0\; \forall \; \emptyset\neq I\subset [n]\}$.
\end{theorem}

\begin{proof}
We already established in \Cref{forwarddir} that for any $F$ in $\mathcal{R}^{> 0}_{id,w_0}$, we have $P_I(F)> 0$ for any $\emptyset\neq I\subset[n]$. We are left to prove the reverse direction. 

Let $F$ be any flag in $\textnormal{Fl}_n$ such that $P_I(F)> 0$ for all $\emptyset\neq I\subset [n]$. We prove that $F\in\mathcal{R}_{id,w_0}^{>0}$. To do so, we show that there exist some positive real weights $\bm{a}$ such that  $\Phi_{id,w_0}(\bm{a})=P(F)$. Apply the map $\Psi_{id,w_0}$ to the extremal non-zero \Pl coordinates of $F$ to get a collection of weights $\bm{a}$. Let $\Phi_{id,w_0}(\bm{a})=Q\in\mathbb{R}\mathbb{P}^{\binom{n}{1}-1}\times\cdots \times \mathbb{R}\mathbb{P}^{\binom{n}{n-1}-1}$. By construction, $P(F)$ and $Q$ agree on their extremal non-zero \Pl coordinates. By \Cref{topcell3termgenerate}, the three-term incidence \Pl relations determine all the other \Pl coordinates in terms of the extremal non-zero \Pl coordinates and so $P(F)$ and $Q$ agree on all \Pl coordinates. Thus, $F$ lies in $\mathcal{R}_{id,w_0}^{>0}$.
\end{proof}

Before tropicalizing this result, we present an example.

\begin{example}
 \label{monomialsidw_0}
 For $n=3$, the graph $G_{id,w_0}$ is given in \Cref{graphidw_0}. The extremal indices in $\mathcal{R}_{id,w_0}^{>0}$ are $\{1\}$, $\{3\}$, $\{1,2\}$, $\{1,3\}$ and $\{2,3\}$. Let $F$ be a flag such that $P(F)_I>0$ for all $\emptyset\neq I\subset[n]$.
 
      \begin{figure}[H]
 \begin{tikzpicture}[node distance={10.5 mm}, thick, main/.style = {draw, circle,minimum size=2 mm}, 
blank/.style={circle, draw=green!0, fill=green!0, very thin, minimum size=3.5mm},]

\node[main] (1) {$1'$};
\node[main] (2) [above of=1] {$2'$};
\node[main] (3) [above of = 2] {$3'$}; 
\node (blank)[above of = 3]{};
\node[main] (33) [right of = blank] {$3$};
\node[main] (22) [right of = 33] {$2$};
\node[main] (11) [right of = 22] {$1$};
\node (first) [right of = 2]{};
\node (second) [right of = 3]{};
\node (third)[right of = second]{};
\draw[-{Latex[length=3mm]}] (1) -- (11);
\draw[-{Latex[length=3mm]}] (2) -- (22);
\draw[-{Latex[length=3mm]}] (3) -- (33);

\draw[-{Latex[length=3mm]}] ([xshift=-5mm,yshift=-5mm]first.center) -- ([xshift=-5mm,yshift=-5mm]second.center) node [pos=0.6, right] {$a_1$};

\draw[-{Latex[length=3mm]}] ([xshift=-5mm,yshift=-5mm]second.center) -- ([xshift=-5mm,yshift=-5mm]33.center) node [pos=0.6, right] {$a_2$};

\draw[-{Latex[length=3mm]}] ([xshift=-5mm,yshift=-5mm]third.center) -- ([xshift=-5mm,yshift=-5mm]22.center) node [pos=0.6, right] {$a_3$};
\end{tikzpicture} 
\caption{The graph $G_{id, w_0}$ for $n=3$.}
\label{graphidw_0}
 \end{figure}
 
 First, we will determine the weights $a_i$ in terms of the extremal \Pl coordinates. We do so by working through the extremal indices in the order $\prec$ introduced in \Cref{prop:oneextraweight} and solving for the $a_i$. We start with the lexicographically minimal extremal index of size $2$, which is $\{1,2\}$. The unique non-intersecting path collection with source set $[2]'$ and sink set $\{1,2\}$ has weight $1$. The Gale-next extremal index of size 2 is $\{1,3\}$. The weight of the unique non-intersecting path collection with source set $[2]'$ and sink set $\{1,3\}$ is just $a_2$. Thus, we have $a_2=P_{13}$. The Gale-next extremal index of size $2$ is $\{2,3\}$. The unique non-intersecting path collection with source set $[2]'$ and sink set $\{2,3\}$ has weight $a_2a_3$. Thus, $a_3=\frac{P_{23}}{P_{13}}$. Next, we move on to the extremal indices of size $1$. The Gale minimal extremal index of size $1$ is $\{1\}$. The unique path from $[1]'$ to $\{1\}$ has weight $1$. The Gale-next extremal index of size $1$ is $\{3\}$. The unique path from $[1]'$ to $\{3\}$ has weight $a_1a_2$, so $a_1=\frac{P_3}{P_{13}}$. This determines all the weights in $G_{id,w_0}$ as Laurent monomials in the extremal non-zero \Pl coordinates. 
 
 We now check that if we apply the LGV construction, the \Pl coordinates $Q$ which we obtain are the same as the \Pl coordinates $P$ we started with. The only interesting coordinate to check is the non-extremal \Pl coordinate $Q_2$. Indeed, $Q_2=a_1+a_3=\frac{P_3}{P_{13}}+\frac{P_{23}}{P_{13}}=P_2$, where the last equality follows from the three-term incidence \Pl relation $P_{2}P_{13}=P_{1}P_{23}+P_{3}P_{12}$ with $P_1=P_{12}=1$. 
 
\end{example}

\begin{thmbis}{topcellnonnegPluckCoords}\label{maintheoremtrop}
The totally positive tropical complete flag variety $\textnormal{TrFl}_n^{>0}$ equals the totally positive complete flag Dressian $\textnormal{FlDr}_n^{>0}$. 
\end{thmbis}

\begin{proof}
It is clear by definition that $\textnormal{TrFl}_n^{>0}\subseteq \textnormal{FlDr}_n^{>0}$. By \Cref{tropicalbruhatinterval}, $\textnormal{TrFl}_{id,w_0}^{>0}$ is the subset of $\textnormal{TrFl}_n^{\geq0}$ where all coordinates are real, that is, $\textnormal{TrFl}_{id,w_0}^{>0}=\textnormal{TrFl}_n^{>0}$. Thus, it will suffice to show that $\textnormal{FlDr}_n^{>0}\subseteq \textnormal{TrFl}^{>0}_{id,w_0}$. 

Let $P\in \textnormal{FlDr}_n^{>0}$. We must show that there exists some real weights $\bm{a}$ such that $\Trop \Phi_{id,w_0}(\bm{a})=P$. Since all coordinates of $P$ are real, it has the same extremal indices as a point in $\textnormal{TrFl}_{id,w_0}^{>0}$. Thus, we can apply the map $\Trop \Psi_{id,w_0}$ to the extremal non-infinite \Pl coordinates of $P$ to get a collection of weights $\bm{a}$. Consider $Q=\Trop\Phi_{id,w_0}(\bm{a})\in\mathbb{T}\mathbb{P}^{\binom{n}{1}-1}\times\cdots \times \mathbb{T}\mathbb{P}^{\binom{n}{n-1}-1}$. By \Cref{tropinverse}, $P$ and $Q$ agree on their extremal non-infinite \Pl coordinates. By \Cref{tropicaltopcell3termgenerate}, the three-term incidence \Pl relations determine all the other \Pl coordinates in terms of the extremal non-infinite \Pl coordinates and so $P$ and $Q$ agree on all \Pl coordinates. Thus, $P$ lies in $\textnormal{TrFl}_{id,w_0}^{>0}=\textnormal{TrFl}_{n}^{>0}$.  
\end{proof}

We now take a moment to comment on similarities and differences between the extremal non-zero \Pl coordinates and the \textit{standard chamber minors} defined and studied in \cite{MR}. Marsh and Rietsch express the parameters in their parameterization of the totally nonnegative complete flag variety as Laurent monomials in the standard chamber minors and express the standard chamber minors as monomials in the parameters. We accomplish similar things with the extremal non-zero \Pl coordinates. However, in general, the standard chamber minors differ from the extremal non-zero \Pl coordinates. These differences are best illustrated by an example. 

\begin{example}
We consider $\textnormal{Fl}_4^{>0}$ and the corresponding weighted digraph $G_{id,w_0}(\bm{a})$, illustrated in \Cref{stdchamberfig}.

     \begin{figure}[H]
 \begin{tikzpicture}[node distance={10.5 mm}, thick, main/.style = {draw, circle,minimum size=2 mm}, 
blank/.style={circle, draw=green!0, fill=green!0, very thin, minimum size=3.5mm},]

\node[main] (1) {$1'$};
\node[main] (2) [above of=1] {$2'$};
\node[main] (3) [above of = 2] {$3'$}; 
\node[main] (4) [above of=3] {$4'$};
\node (blank)[above of = 4]{};
\node[main] (44) [right  of=blank] {$4$};
\node(43) [below of = 44]{};
\node(42) [below of = 43]{};
\node(41) [below of = 42]{};
\node[main] (33) [right of = 44] {$3$};
\node(32) [below of = 33]{};
\node(31) [below of = 32]{};
\node[main] (22) [right of = 33] {$2$};
\node(21) [below of = 22]{};
\node[main] (11) [right of = 22] {$1$};

\draw[](4)--(44);
\draw[] (3) --  ([xshift=-5mm,yshift=-5mm]43.center);
\draw[-{Latex[length=3mm]}] ([xshift=-5mm,yshift=-5mm]43.center)--(33);
\draw[] (2) -- (42.center);
\draw[] ([xshift=-5mm,yshift=-5mm]42.center) -- ([xshift=-5mm,yshift=-5mm]32.center);
\draw[-{Latex[length=3mm]}] ([xshift=-5mm,yshift=-5mm]32.center) -- (22);
\draw[] (1) --  ([xshift=-5mm,yshift=-5mm]41.center);
\draw[-{Latex[length=3mm]}] ([xshift=-5mm,yshift=-5mm]41.center)--(11);

\draw[-{Latex[length=3mm]}] ([xshift=-5mm,yshift=-5mm]43.center) -- ([xshift=-5mm,yshift=-5mm]44.center)node [pos=0.65, right] {$a_3$};
\draw[-{Latex[length=3mm]}] ([xshift=-5mm,yshift=-5mm]42.center) -- ([xshift=-5mm,yshift=-5mm]43.center)node [pos=0.65, right] {$a_2$};
\draw[-{Latex[length=3mm]}] ([xshift=-5mm,yshift=-5mm]41.center) -- ([xshift=-5mm,yshift=-5mm]42.center)node [pos=0.65, right]{$a_1$};
\draw[-{Latex[length=3mm]}] ([xshift=-5mm,yshift=-5mm]31.center) -- ([xshift=-5mm,yshift=-5mm]32.center)node [pos=0.65, right]{$a_4$};
\draw[-{Latex[length=3mm]}] ([xshift=-5mm,yshift=-5mm]32.center) -- ([xshift=-5mm,yshift=-5mm]33.center)node [pos=0.65, right]{$a_5$};
\draw[-{Latex[length=3mm]}] ([xshift=-5mm,yshift=-5mm]21.center) -- ([xshift=-5mm,yshift=-5mm]22.center)node [pos=0.65, right]{$a_6$};

\end{tikzpicture} 
\caption{The graph $G_{id, w_0}(\bm{a})$ for $n=4$.}
\label{stdchamberfig}
 \end{figure}
Reading off the extremal non-zero \Pl coordinates and ignoring those which are simply $1$, we obtain $\{a_1a_2a_3, a_2a_3, a_2a_3a_4a_5, a_3, a_3a_5, a_3a_5a_6\}$. Meanwhile, in the notation of \cite{MR}, one can calculate the standard chamber minor $\Delta^{v_{(k-1)}\omega_{i_k}}_{w_{(k-1)}\omega_{i_{k}}}(z)$ for $k=5$. In this case, $v_{(4)}=id$, $w_{(4)}=s_1s_2s_3s_1$ and $i_5=2$. Inputting this data and working through the computation, we obtain the result $a_2a_3a_4$. Note that this is not an extremal non-zero \Pl coordinate and the simplest way to express it as a Laurent monomial in the extremal non-zero \Pl coordinates requires three terms. Moreover, there is no \Pl coordinate (not even a non-extremal one) which equals $a_2a_3a_4$. This highlights, even in the relatively simple setting of the top cell, that there are subtle but potentially important differences between the extremal non-zero \Pl coordinates and the standard chamber minors of Marsh and Rietsch. 
\end{example}

\section{The Totally Nonnegative Complete Flag Variety and its Tropicalization}\label{sectgeneral}
In this section, we prove results analogous to those in the previous section but in the more general setting of $\mathcal{R}_{v,w}^{>0}$ for any $v\leq w\in \mathfrak{S}_n$. The three subsections of this section exactly mirror the three subsections of the previous section, with analogous results appearing in corresponding subsections.

\subsection{Graphical Description of Extremal Indices}\label{generealextremals}

We start by generalizing a basic fact about $G_{id,w_0}$ that followed immediately from its definition: We show that in the graphs $G_{v,w}$, all vertical edges are directed upwards.

\begin{lemma} \label{vertarrow} For any $v\leq w\in \mathfrak{S}_n$, there are no vertical edges in $G_{v,w}$ directed downwards.
\end{lemma}

\begin{proof}
We will prove this by induction on $\ell(v)$. If $\ell(v)=0$, we never swap any strands in the construction of $G_{v,w}$ and so the result is clear. 

Fix $v\leq w\in \mathfrak{S}_n$ with $\ell(v)=d+1$. Suppose that whenever $\ell(u)=d$, the graph $G_{u,w}$ has no edges directed downwards. Let $\bm{w}$ be the positive distinguished subexpression for $w$ in $\bm{w_0}$ and let $\bm{v}$ be the positive distinguished subexpression for $v$ in $\bm{w}$. Let $u$ be obtained from $v$ by removing the last simple reflection in $\bm{v}$. Then $\ell(u)=d$, $u\leq w$, and the positive distinguished subexpression for $u$ in $\bm{w}$ is the subexpression obtained by removing the final simple reflection of $\bm{v}$. 

Suppose towards a contradiction that there is an edge $e$ directed downwards in $G_{v,w}$. We will construct four non-intersecting path collections in $G_{u,w}$ and $G_{v,w}$, each originating at $[k]'$ for some $k$. We will show that the product of their contributions to \Cref{LGVcor} is negative, contradicting \Cref{sumonly}. By induction, there are no vertical edges directed downwards in $G_{u,w}$. Thus, the final simple reflection $s^*$ in $\bm{v}$ must have caused $e$ to flip. Suppose that $e$ has weight $a_e$ and that $s^*=s_r$. In such a situation $e$ was directed from strand $r$ to strand $r+1$ in $G_{u,w}$ whereas it is directed from strand $r+1$ to strand $r$ in $G_{v,w}$. Suppose that, in $G_{u,w}$, the source vertex on strand $r$ is $k'$ and the source vertex on strand $r+1$ is $l'$. By the reducedness of $\bm{v}$ and by \Cref{vertexinversion}, swapping strands $r$ and $r+1$ of $G_{u,w}$ must introduce an inversion in the labels of the source vertices of $G_{u,w}$, so we must have $k<l$. In $G_{v,w}$, they swap; the source vertex on strand $r+1$ is $k'$ and the source vertex on strand $r$ is $l'$. This is illustrated in \Cref{figdownarrows}.

By \Cref{weightedG}, the edges of $G_{v,w}$ with weight different from $1$ are the vertical edges and certain sections of diagonal strands. The contribution of a non-intersecting path collection to a minor, via the LGV construction, also contains a sign due to the relative ordering of the source vertices, as indicated by the sign term in \Cref{LGVcor}. 

Let $C_1$ be the diagonal path collection originating from $[k]'$ in $G_{u,w}$. Let $C_2$ be the non-intersecting path collection in $G_{u,w}$ originating from $[k]'$ in which all paths are diagonal except the one originating at vertex $k'$. This one should follow strand $r$ until reaching edge $e$, use edge $e$, and then terminate along strand $r+1$. This is possible since $l>k$ and so there is no path originating from source vertex $l'$ on strand $r+1$ in this path collection. Let $C_3$ be the diagonal path collection originating from $[k]'$ in $G_{v,w}$. Let $C_4$ be the non-intersecting path collection in $G_{v,w}$ originating from $[k]'$ in which all paths are diagonal except the one originating at vertex $k'$. This one should follow strand $r+1$ until reaching edge $e$, use edge $e$, and then terminate along strand $r$. This is possible since $l>k$ and so there is no path originating from source vertex $l'$ on strand $r$ in this path collection. We now compute the product of the contributions of $C_1$, $C_2$, $C_3$ and $C_4$ to \Cref{LGVcor} and see that it is negative, contradicting the fact that all non-intersecting path collections originating from $[k]'$ in both $G_{u,w}$ and $G_{v,w}$ contribute a positive term to \Cref{LGVcor}.

Aside from flipping the edge $e$, the application of $s^*$ adds a weight $-1$ section of strand $r$ in $G_{v,w}$. This section must be to the right of edge $e$ since edge $e$ was added to the graph prior to the application of $s^*$ in the construction of $G_{v,w}$. Consider the product $W$ of the contributions of each of the non-intersecting path collections $C_1,C_2,C_3$ and $C_4$ to \Cref{LGVcor}. These four path collections all use mostly the same edges. They differ in their use of the edge $e$ and of certain sections of strands $r$ and $r+1$. 

Let $\delta\in\{-1,1\}$ represent the weight of the diagonal path collection originating from $[k-1]'$, which is common to all the $C_i$ for $i\in [4]$. This contributes a factor of $\delta^4=1$ to $W$. All these path collections maintain the same relative ordering of the source vertices to the sink vertices, so they all contribute the same sign $\epsilon\in \{-1,1\}$ to \Cref{LGVcor}. This contributes a factor of $\epsilon^4=1$ to $W$. Strand $r+1$ has the same total weight in both $G_{u,w}$ and $G_{v,w}$. The part of strand $r+1$ to the left of $e$ is used by $C_3$ and by $C_4$ while the part to the right of $e$ is used by $C_2$ and $C_3$. If the total weight of strand $r+1$ in $G_{u,w}$ is $\delta_{r+1}\in \{-1,1\}$, then this strand contributes $\delta_{r+1}^2=1$ to $W$. The weights on strand $r$ to the left of $e$ are identical in $G_{u,w}$ and in $G_{v,w}$. This section gets used by $C_1$ and by $C_2$. The section of strand $r$ to the right of edge $e$ picks up an extra weight of $-1$ in $G_{v,w}$. This section of the graph gets used by $C_2$ in $G_{u,w}$ and by $C_4$ in $G_{v,w}$. Thus, if the total weight of strand $r$ in $G_{u,w}$ is $\delta_r\in \{-1,1\}$, this strand contributes $-\delta_r^2=-1$ to $W$. Finally, we still need to account for the vertical edge $e$ which gets used by both $C_2$ and $C_4$, contributing $a_e^2$ to $W$. Thus, $W=\epsilon^2\delta^4\delta_{r+1}^2(-\delta_r^2)a_e^2=-a_e^2$. Edge weights are non-zero, so this is a contradiction to the fact that each non-intersecting path collection contributes positively to \Cref{LGVcor}. Thus, it is impossible that there is an edge oriented downwards in $G_{v,w}$.

\begin{figure}[H]
 \begin{subfigure}{0.4\textwidth}
  \scalebox{.6}{
 \begin{tikzpicture}[node distance={14 mm}, very thick, main/.style = {draw, circle,minimum size=16 mm}, 
blank/.style={circle, draw=green!0, fill=green!0, very thin, minimum size=3.5mm},]
\node[main] (1) {\large $k'$};
\node (2) [above of=1] {};
\node[main](3) [above of = 2] {\large $l'$}; 
\node[] (4) [above of=3] {};
\node(5) [above of = 4] {};
\node (blank)[above of = 5]{};
\node(55) [right of = blank]{};
\node(54) [below of = 55]{};
\node(53) [below of = 54]{};
\node(52) [below of = 53]{};
\node(51) [below of = 52]{};
\node(44) [right  of=55] {};
\node(43) [below of = 44]{};
\node(42) [below of = 43]{};
\node(41) [below of = 42]{};
\node[main] (33) [right of = 44] {\large $r+1$};
\node(32) [below of = 33]{};
\node(31) [below of = 32]{};
\node (22) [right of = 33] {};
\node(21) [below of = 22]{};
\node[main](11) [right of = 22] {\large $r$};

\draw[-{Latex[length=3mm]}] (3)--(33);
\draw[-{Latex[length=3mm]}] (1) -- (11);

\draw[-{Latex[length=3mm]}] ([xshift=-5mm,yshift=-5mm]41.center) -- ([xshift=-5mm,yshift=-5mm]43.center) node[midway,xshift=2mm]{\Large$e$};
\end{tikzpicture}}
\label{flip1}
\caption{}
\end{subfigure}
 \begin{subfigure}{0.4\textwidth}
  \scalebox{.6}{
 \begin{tikzpicture}[node distance={14 mm}, very thick, main/.style = {draw, circle,minimum size=16 mm}, 
blank/.style={circle, draw=green!0, fill=green!0, very thin, minimum size=3.5mm},]
\node[main] (1) {\large $l'$};
\node (2) [above of=1] {};
\node[main](3) [above of = 2] {\large $k'$}; 
\node[] (4) [above of=3] {};
\node(5) [above of = 4] {};
\node (blank)[above of = 5]{};
\node(55) [right of = blank]{};
\node(54) [below of = 55]{};
\node(53) [below of = 54]{};
\node(52) [below of = 53]{};
\node(51) [below of = 52]{};
\node(44) [right  of=55] {};
\node(43) [below of = 44]{};
\node(42) [below of = 43]{};
\node(41) [below of = 42]{};
\node[main] (33) [right of = 44] {\large $r+1$};
\node(32) [below of = 33]{};
\node(31) [below of = 32]{};
\node (22) [right of = 33] {};
\node(21) [below of = 22]{};
\node[main](11) [right of = 22] {\large $r$};

\draw[-{Latex[length=3mm]}] (3)--(33);
\draw[-{Latex[length=3mm]}] (1) -- (11);

\draw[-{Latex[length=3mm]}] ([xshift=-5mm,yshift=-5mm]43.center) -- ([xshift=-5mm,yshift=-5mm]41.center) node[midway,xshift=2mm]{\Large$e$};
\end{tikzpicture}}
\label{flip2}
\caption{}
\end{subfigure}

\caption{Illustration of strands $r$ and $r+1$ of $G_{u,w}$ (left) and $G_{v,w}$ (right), as described in the proof of \Cref{vertarrow}. In this figure, $l>k$.}
\label{figdownarrows}
 \end{figure}

\end{proof}

We begin by understanding the extremal \Pl coordinates of a point in an arbitrary cell $\mathcal{R}_{v,w}^{>0}$ of the totally nonnegative complete flag variety. This is where we will make real use of the graphs $G_{v,w}$ introduced earlier. At a basic level, what changes from \Cref{secttopcell} is that we can now have coordinates which are zero. Thus, it will be helpful to introduce the following terminology:

\begin{definition}\label{support}
 For a point $P\in\mathbb{R}\mathbb{P}^{\binom{n}{1}-1}\times\cdots \times \mathbb{R}\mathbb{P}^{\binom{n}{n-1}-1}$, the \textbf{support} of $P$, denoted $\supp(P)$, is the set of indices of the non-zero coordinates of $P$.
\end{definition} 

\begin{defbis}{support}\label{tropicalsupport}
For a point $P\in\mathbb{T}\mathbb{P}^{\binom{n}{1}-1}\times\cdots \times \mathbb{T}\mathbb{P}^{\binom{n}{n-1}-1}$, the \textbf{support} of $P$, denoted $\supp(P)$, is the set of indices of the non-infinite coordinates of $P$.   
\end{defbis}

Recall that by tropicalizing the parameterization $\Phi_{v,w}$ of $\mathcal{R}_{v,w}^{>0}$, we obtain a parameterization of a cell $\textnormal{TrFl}_{v,w}^{>0}$ of $\textnormal{TrFl}_n^{\geq 0}$. Thus, the points of $\textnormal{TrFl}_{v,w}^{>0}$ have the same support as flags in $\mathcal{R}_{v,w}^{>0}$. Since $\textnormal{TrFl}_n^{\geq 0}\subset \textnormal{FlDr}_n^{\geq0}$, there exist points in $\textnormal{FlDr}_n^{\geq 0}$ with that same support as well. This justifies the following definition:

\begin{definition}
 We denote by $\textnormal{FlDr}_{v,w}^{> 0}$ the set $\{P\in \textnormal{FlDr}
_n^{\geq0}\mid \supp(P)=\supp(F) \textnormal{ for some, equivalent any, } F\in \mathcal{R}_{v,w}^{>0}\}$. 
   
\end{definition}

We now present a graphical description of extremal coordinates, analogous to \Cref{topcellcor}, which offers a more combinatorial way to think about the extremal indices of a totally nonnegative flag. The key fact is that for any $v\leq w\in \mathfrak{S}_n$ and any extremal index $I$ of $\mathcal{R}_{v,w}^{>0}$, there is a unique non-intersecting path collection in $G_{v,w}$ with source set $\{1',2',\cdots, |I|'\}$ and sink set $I$. 

 Recall the greedy path collections of \Cref{defleftgreedy}. We now introduce a related notion. These concepts are in general different, although we will see that they coincide in cases of interest. 

 \begin{definition}\label{leftextremedef}
  
 Consider a source set $A=\{a_i'\}_{i=1}^m$ ordered from top to bottom in $G_{v,w}$. Choose the leftmost possible sink $b_1$ such that there exists a path from $a_1'$ to $b_1$. Having determined $b_i$ for $i\leq j\leq m-1$, choose the leftmost possible sink $b_{j+1}$ subject to the constraint that there is a non-intersecting path collection from $\{a_i'\}_{i=1}^{j+1}$ to $\{b_i\}_{i=1}^{j+1}$. Then, any non-intersecting path collection from $A=\{a_i'\}_{i=1}^m$ to $\{b_i\}_{i=1}^m$ is called \textbf{left extreme}. See \Cref{leftextremefig} for an illustration.
\end{definition}

    \begin{figure}[H]
     
 \begin{tikzpicture}[node distance={10.5 mm}, thick, main/.style = {draw, circle,minimum size=2 mm}, 
blank/.style={circle, draw=green!0, fill=green!0, very thin, minimum size=3.5mm},]

\node[main] (1) {$1'$};
\node[main] (2) [above of=1] {$2'$};
\node[main] (3) [above of = 2] {$3'$}; 
\node[main] (4) [above of=3] {$4'$};
\node[main](5) [above of = 4] {$5'$};
\node (blank)[above of = 5]{};
\node[main](55) [right of = blank]{$5$};
\node(54) [below of = 55]{};
\node(53) [below of = 54]{};
\node(52) [below of = 53]{};
\node(51) [below of = 52]{};
\node[main] (44) [right  of=55] {$4$};
\node(43) [below of = 44]{};
\node(42) [below of = 43]{};
\node(41) [below of = 42]{};
\node[main] (33) [right of = 44] {$3$};
\node(32) [below of = 33]{};
\node(31) [below of = 32]{};
\node[main] (22) [right of = 33] {$2$};
\node(21) [below of = 22]{};
\node[main] (11) [right of = 22] {$1$};

\draw[-{Latex[length=3mm]},dashed] (5) -- (55);
\draw[](4)--([xshift=-5mm,yshift=-5mm]54.center);
\draw[-{Latex[length=3mm]},dashed] ([xshift=-5mm,yshift=-5mm]54.center) -- (44);
\draw[dashed] (3) -- ([xshift=-5mm,yshift=-5mm]53.center);
\draw[] ([xshift=-5mm,yshift=-5mm]53.center) -- ([xshift=-5mm,yshift=-5mm]43.center);
\draw[-{Latex[length=3mm]},dashed] ([xshift=-5mm,yshift=-5mm]43.center)--(33);
\draw[-{Latex[length=3mm]}] (2) -- (22);
\draw[dashed] (1) -- ([xshift=-5mm,yshift=-5mm]41.center);
\draw[-{Latex[length=3mm]}] ([xshift=-5mm,yshift=-5mm]41.center)--(11);

\draw[-{Latex[length=3mm]}] ([xshift=-5mm,yshift=-5mm]54.center) -- ([xshift=-5mm,yshift=-5mm]55.center);
\draw[-{Latex[length=3mm]},dashed] ([xshift=-5mm,yshift=-5mm]53.center) -- ([xshift=-5mm,yshift=-5mm]54.center);
\draw[-{Latex[length=3mm]}] ([xshift=-5mm,yshift=-5mm]52.center) -- ([xshift=-5mm,yshift=-5mm]53.center);
\draw[-{Latex[length=3mm]}] ([xshift=-5mm,yshift=-5mm]51.center) -- ([xshift=-5mm,yshift=-5mm]52.center);
\draw[-{Latex[length=3mm]}] ([xshift=-5mm,yshift=-5mm]43.center) -- ([xshift=-5mm,yshift=-5mm]44.center);
\draw[-{Latex[length=3mm]},dashed] ([xshift=-5mm,yshift=-5mm]42.center) -- ([xshift=-5mm,yshift=-5mm]43.center);
\draw[-{Latex[length=3mm]}, dashed] ([xshift=-5mm,yshift=-5mm]41.center) -- ([xshift=-5mm,yshift=-5mm]42.center);

\end{tikzpicture}
\label{a1}

\caption{The dashed paths form a left extreme path collection with source set $\{a_1'=5',a_2'=3', a_3'=1'\}$, and sink set $\{b_1=5, b_2=4,b_3=3\}$. Note that this is not a greedy path collection because the path originating at $1'$ does not use the first vertical edge it passes.}
\label{leftextremefig}
 \end{figure}

 \begin{definition}\label{def:graphextremal}
 Let $v\leq w\in \mathfrak{S}_n$. Fix $0\leq i\leq k<n$ and let $A=[k]'$. We define a \textbf{graph extremal path collection} to be the union of a left extreme path collection with source set the topmost $i$ vertices of $A$ and the diagonal path collection with source set the bottom-most $k-i$ vertices of $A$ in the graph $G_{v,w}$. Its sink set will be called a \textbf{graph extremal index}. 
\end{definition} 

\begin{definition} \label{def:Rvwpositroid}
    By \Cref{bruhatinterval}, the support of any flag in $\Fl_n^{\geq 0}$ is determined by the cell $\mathcal{R}_{v,w}^{>0}$ it belongs to. This support is a realizable flag positroid, which we denote by $\bm{\mathcal{M}^{v,w}}$. The realizable flag positroid $\bm{\mathcal{M}^{v,w}}$ is also the support of $\textnormal{TrFl}_{v,w}^{>0}$.
\end{definition}

\begin{proposition}\label{galemax}
    The sink set of the left extreme part of a graph extremal path collection is Gale maximal among the sink sets of all non-intersecting path collections originating from the same source set, namely, the topmost $i$ vertices of $[k]'$ for some $1\leq i\leq k\leq n$. 
\end{proposition}

\begin{proof}

    Fix $v\leq w\in \mathfrak{S}_n$ and $1\leq k\leq n$.  Let $(\mathcal{M}_1,\ldots, \mathcal{M}_{n-1})$ be the constituent positroids of $\bm{\mathcal{M}^{v,w}}$. Suppose the source vertices $[k]'$ in $G_{v,w}$ lie on strands $j_1,\ldots,j_k$, ordered from bottom to top. Let $S_i$ be the sink set of a left extreme path collection originating from the top $i$ vertices of $[k]'$, so that $S_i\cup\{j_1,\ldots, j_{k-i}\}$ is a graph extremal index. Since there are no arrows directed downwards in $G_{v,w}$ by \Cref{vertarrow}, any non-intersecting path collection originating from the topmost $i$ vertices of $[k]'$ with sink set $S$ can be extended to a non-intersecting path collection from $[k]'$ to $S\cup\{j_1,\ldots, j_{k-i}\}$ by adding diagonal paths on strands $j_1,\ldots, j_{k-i}$. Thus, to prove the proposition, it suffices to show that $S_i$ is Gale maximal among all sets $S$ such that $S\cup\{j_1,\ldots, j_{k-i}\}$ is a basis of $\mathcal{M}_k$. Equivalently, it suffices to show that $S_i$ is the Gale maximal basis of $\mathcal{N}_{i}=\mathcal{M}_k/ \{j_1,\ldots, j_{k-i}\}$, the contraction of $\mathcal{M}_k$ by $\{j_1,\ldots, j_{k-i}\}$. By \cite[Lemma 7.3.3]{Oxl}, $\bm{\mathcal{N}}=(\mathcal{N}_1,\ldots, \mathcal{N}_k)$ is a flag matroid since for $i\in[k-1]$, $\mathcal{N}_i=\mathcal{N}_{i+1}/ \{j_{k-i}\}$. Thus, by \cite[Corollary 1.7.2]{Coxeter}, the Gale maximal bases $G_i$ of the matroids $\mathcal{N}_i$ form a flag. We must show $G_i=S_i$ for each $i\in [k]$. We will do so by induction. 
    
    By definition, $G_1=S_1$. For the induction step, suppose $S_i=G_i$. Since $(G_1,\ldots, G_k)$ form a flag, $G_{i+1}=G_i\cup a$ for some $a\in [n]$ and since the $G_i$ are Gale maximal bases, $a$ must be maximal subject to the condition that $G_{i+1}$ is a basis of $\mathcal{N}_{i+1}$. By definition of graph extremal indices, we also have $S_{i+1}=S_i\cup a'$, where $a'$ is maximal subject to the condition that there is a non-intersecting path collection from the topmost $i+1$ vertices of $[k]'$ to $S_{i+1}$. However, $S_i=G_i$ and there exists a non-intersecting path collection from the top $i+1$ vertices of $[k]'$ to $S_{i+1}$ if and only if $S_{i+1}$ is a basis of $\mathcal{N}_{i+1}$. Thus, $a'=a$ and $S_{i+1}=G_{i+1}$. We are then finished by induction. 
\end{proof}

\Cref{def:graphextremal} is formulated precisely such that the following holds: 

\begin{proposition}\label{prop:graphextremalisextremal}
Let $v\leq w\in \mathfrak{S}_n$. The extremal indices of a flag in $\mathcal{R}_{v,w}^{>0}$ coincide precisely with graph extremal indices of $G_{v,w}$.
\end{proposition}

\begin{proof}
    Both extremal indices and graph extremal indices are constructed inductively. We will compare the two constructions. Note that the sink set of the diagonal path collection originating from $[k]'$ in $G_{v,w}$ is the Gale minimal extremal \Pl coordinate of size $k$, since there are no vertical edges directed downwards in $G_{v,w}$ by \Cref{vertarrow}. Denote this sink set by $I_k$. This is the base case for the construction of both the extremal indices and the graph extremal indices. 
    
    Consider an extremal path collection $C_l$, which is the union of a left extreme path collection originating from the topmost $l$ vertices of $[k]'$ and diagonal paths from the remaining $k-l$ vertices of $[k]'$. How do we construct $C_{l'}$, where $l'$ is the smallest number greater than $l$ such that $C_{l'}\neq C_l$? First, we identify the topmost path $p$ in the diagonal part of $C_l$ which satisfies the following property: If we construct a left extreme path collection originating from the vertices of $[k]'$ which lie weakly above the strand $\sigma$ where $p$ originates, then the path originating on $\sigma$ is not diagonal. Then, $C_{l'}$ is the union of the left extreme path collection originating from the vertices of $[k]'$ weakly above $\sigma$ and diagonal paths originating from the vertices of $[k]'$ below $\sigma$. By \Cref{leftextremedef}, the sink set of $C_{l'}$ equals $(S_l\setminus\sigma)\cup\tau$, where $S_l$ is the sink set of $C_l$ and $\tau$ is maximal subject to the condition that there is a non-intersecting path collection from $[k]'$ to $S_l$. Equivalently, $\tau$ is maximal subject to the condition that $(S_l\setminus\sigma)\cup \tau$ is the index of a non-zero \Pl coordinate of a flag in $\mathcal{R}_{v,w}^{>0}$. We remark that since the path $p$ originating on strand $\sigma$ in $C_l$ is diagonal, we necessarily have that $\sigma$ is in the sink set of $C_r$ for all $r\leq l$. Thus, $\sigma$ is the maximal element in the intersection of the sink sets of $\{C_r\mid r\leq l\}$ with the property that there exists a path collection with sink set $(C_l\setminus\sigma)\cup \tau$ for some $\tau$.

    Now consider $\Xi^i(I_k)$. How do we construct $\Xi^{i+1}(I_k)$? By \Cref{ximap}, we identify the largest element $b$ of $\Xi^{i}(I_k)$ which can be increased, and replace it by $a$, where $a$ is maximal subject to the condition that $(\Xi^{i}(I_k)\setminus b)\cup a$ is the index of a non-zero \Pl coordinate of a flag in $\mathcal{R}_{v,w}^{>0}$. Also, by \Cref{claim}, we know that $b\in \bigcap_{0\leq r\leq i}\Xi^{r}(I_k)$. 
    
    Comparing the previous two paragraphs, and assuming by induction that $\Xi^{r}(I_k)$ is the sink set of some $C_{r'}$ for each $r<i$ and for some $r'$ depending on $r$, we must have $\sigma=b$ and $\tau=a$. This proves that indeed, the graph extremal indices and extremal indices of a flag in $\mathcal{R}_{v,w}^{>0}$ coincide.

    \end{proof}

We will use the terms \textit{graph extremal indices} and \textit{extremal indices} interchangeably, depending whether we want to emphasize the combinatorics of the graph $G_{v,w}$ or the algebraic geometry of $\mathcal{R}_{v,w}^{>0}$.

    \begin{remark}
    Combining \Cref{galemax} and \Cref{prop:graphextremalisextremal}, we obtain a new description of extremal indices: Fix $v\leq w\in \mathfrak{S}_n$ and $k\in[n-1]$. Let $I_k$ be the Gale minimal basis of the rank $k$ constituent matroid of $\bm{\mathcal{M}^{v,w}}$, and for $i\in [k]$, let $J_i$ be the smallest $i$ elements of $I_k$. Then, extremal indices are sets of the form $J_i\cup S$ where $S$ is Gale maximal subject to the constraint that $J_i\cup S$ be a basis of the rank $k$ constituent matroid of $\bm{\mathcal{M}^{v,w}}$. 
   
\end{remark}

 The following result, which is the main result of this section, will play a key role in allowing us to relate the weights $a_i$ of the graph $G_{v,w}(\bm{a})$ to the extremal non-zero \Pl coordinates. 

  \begin{proposition}\label{extremunique} 
  
  Let $v\leq w\in \mathfrak{S}_n$, and let $F\in \mathcal{R}_{v,w}^{> 0}$. Suppose $I$ is an extremal index of $F$. There is a unique extremal path collection with sink set $I$ in the graph $G_{v,w}$. It is the union of a diagonal path collection and a greedy path collection.

  Conversely, the sink set of any non-intersecting path collection in $G_{v,w}$ that is the union of a greedy path collection originating from the top $q$ vertices of $[|I|]'$ and a diagonal path collection originating from the remaining vertices of $[|I|]'$ for some $0\leq q\leq |I|$ is an extremal index.
\end{proposition}

\begin{cor}\label{extremuniquecor}
Let $P=\Phi_{v,w}(\bm{a})$. For each extremal index $I$ for $\mathcal{R}_{v,w}^{>0}$, the \Pl coordinate $P_I$ is a monomial in the $\{a_i\}$. 
\end{cor}
\begin{proof}
Since there is a unique non-intersecting path collection originating from $[|I|]'$ with sink set $I$ by \Cref{extremunique}, this follows from \Cref{LGVcor}.
\end{proof}

 The proof of \Cref{extremunique} is long and technical, and relies on a number of other technical results. To maintain the flow of the larger argument, we limit ourselves to an example here and reserve the proof of this result for \Cref{sec:BigProof}.

  \begin{example}\label{ex:extremeunique}
  We study the extremal indices of size $3$ when $v=s_2s_3s_2$ and $w=s_1s_2s_3s_4s_1s_2s_3s_1$ in $\mathfrak{S}_5$. Let $G=G_{v,w}$ be the corresponding graph, illustrated in \Cref{examplegraphical}. We will observe that for each extremal index $I$ with $|I|=3$, there is a unique non-intersecting path collection from $[3]'$ to $I$ in $G$, which is the union of a greedy path collection and a diagonal path collection. The Gale minimal extremal index of size $3$ is the sink set of the diagonal path collection with source set $[3]'$. As indicated with dashed lines in the top left of \Cref{examplegraphical}, this is $I=\{1,3,4\}$. There are no other non-intersecting path collections from $[3]'$ to $\{1,3,4\}$ in $G$.

    \begin{figure}[H]
 \begin{tikzpicture}[node distance={10.5 mm}, thick, main/.style = {draw, circle,minimum size=2 mm}, 
blank/.style={circle, draw=green!0, fill=green!0, very thin, minimum size=3.5mm},]

\node[main] (1) {$1'$};
\node[main] (2) [above of=1] {$4'$};
\node[main] (3) [above of = 2] {$3'$}; 
\node[main] (4) [above of=3] {$2'$};
\node[main](5) [above of = 4] {$5'$};
\node (blank)[above of = 5]{};
\node[main](55) [right of = blank]{$5$};
\node(54) [below of = 55]{};
\node(53) [below of = 54]{};
\node(52) [below of = 53]{};
\node(51) [below of = 52]{};
\node[main] (44) [right  of=55] {$4$};
\node(43) [below of = 44]{};
\node(42) [below of = 43]{};
\node(41) [below of = 42]{};
\node[main] (33) [right of = 44] {$3$};
\node(32) [below of = 33]{};
\node(31) [below of = 32]{};
\node[main] (22) [right of = 33] {$2$};
\node(21) [below of = 22]{};
\node[main] (11) [right of = 22] {$1$};

\draw[] (5) -- ([xshift=-5mm,yshift=-5mm]55.center);
\draw[-{Latex[length=3mm]}] ([xshift=-5mm,yshift=-5mm]55.center) -- (55);
\draw[dashed] (4) -- ([xshift=-5mm,yshift=-5mm]54.center);
\draw[dashed] ([xshift=-5mm,yshift=-5mm]54.center)--([xshift=-5mm,yshift=-5mm]44.center);
\draw[-{Latex[length=3mm]},dashed] ([xshift=-5mm,yshift=-5mm]44.center) -- (44);
\draw[dashed] (3) -- ([xshift=-5mm,yshift=-5mm]43.center);
\draw[-{Latex[length=3mm]}, dashed] ([xshift=-5mm,yshift=-5mm]43.center)--(33);
\draw[] (2) -- ([xshift=-5mm,yshift=-5mm]32.center);
\draw[-{Latex[length=3mm]}] ([xshift=-5mm,yshift=-5mm]32.center)--(22);
\draw[dashed](1)--([xshift=-5mm,yshift=-5mm]31.center);
\draw[-{Latex[length=3mm]}, dashed] ([xshift=-5mm,yshift=-5mm]31.center)--(11);

\draw[-{Latex[length=3mm]}] ([xshift=-5mm,yshift=-5mm]54.center) -- ([xshift=-5mm,yshift=-5mm]55.center);
\draw[-{Latex[length=3mm]}] ([xshift=-5mm,yshift=-5mm]51.center) to [out=-285,in=-78] ([xshift=-5mm,yshift=-5mm]54.center);
\draw[-{Latex[length=3mm]}] ([xshift=-5mm,yshift=-5mm]43.center) -- ([xshift=-5mm,yshift=-5mm]44.center);
\draw[-{Latex[length=3mm]}] ([xshift=-5mm,yshift=-5mm]41.center) to [out=-285,in=-78] ([xshift=-5mm,yshift=-5mm]43.center);
\draw[-{Latex[length=3mm]}] ([xshift=-5mm,yshift=-5mm]31.center) -- ([xshift=-5mm,yshift=-5mm]32.center);

\end{tikzpicture} 
 \begin{tikzpicture}[node distance={10.5 mm}, thick, main/.style = {draw, circle,minimum size=2 mm}, 
blank/.style={circle, draw=green!0, fill=green!0, very thin, minimum size=3.5mm},]

\node[main] (1) {$1'$};
\node[main] (2) [above of=1] {$4'$};
\node[main] (3) [above of = 2] {$3'$}; 
\node[main] (4) [above of=3] {$2'$};
\node[main](5) [above of = 4] {$5'$};
\node (blank)[above of = 5]{};
\node[main](55) [right of = blank]{$5$};
\node(54) [below of = 55]{};
\node(53) [below of = 54]{};
\node(52) [below of = 53]{};
\node(51) [below of = 52]{};
\node[main] (44) [right  of=55] {$4$};
\node(43) [below of = 44]{};
\node(42) [below of = 43]{};
\node(41) [below of = 42]{};
\node[main] (33) [right of = 44] {$3$};
\node(32) [below of = 33]{};
\node(31) [below of = 32]{};
\node[main] (22) [right of = 33] {$2$};
\node(21) [below of = 22]{};
\node[main] (11) [right of = 22] {$1$};

\draw[] (5) -- ([xshift=-5mm,yshift=-5mm]55.center);
\draw[-{Latex[length=3mm]},dashed] ([xshift=-5mm,yshift=-5mm]55.center) -- (55);
\draw[dashed] (4) -- ([xshift=-5mm,yshift=-5mm]54.center);
\draw[] ([xshift=-5mm,yshift=-5mm]54.center)--([xshift=-5mm,yshift=-5mm]44.center);
\draw[-{Latex[length=3mm]}] ([xshift=-5mm,yshift=-5mm]44.center) -- (44);
\draw[dashed] (3) -- ([xshift=-5mm,yshift=-5mm]43.center);
\draw[-{Latex[length=3mm]},dashed] ([xshift=-5mm,yshift=-5mm]43.center)--(33);
\draw[] (2) -- ([xshift=-5mm,yshift=-5mm]32.center);
\draw[-{Latex[length=3mm]}] ([xshift=-5mm,yshift=-5mm]32.center)--(22);
\draw[dashed](1)--([xshift=-5mm,yshift=-5mm]31.center);
\draw[-{Latex[length=3mm]},dashed] ([xshift=-5mm,yshift=-5mm]31.center)--(11);

\draw[-{Latex[length=3mm]},dashed] ([xshift=-5mm,yshift=-5mm]54.center) -- ([xshift=-5mm,yshift=-5mm]55.center);
\draw[-{Latex[length=3mm]}] ([xshift=-5mm,yshift=-5mm]43.center) -- ([xshift=-5mm,yshift=-5mm]44.center);
\draw[-{Latex[length=3mm]}] ([xshift=-5mm,yshift=-5mm]51.center) to [out=-285,in=-78] ([xshift=-5mm,yshift=-5mm]54.center);
\draw[-{Latex[length=3mm]}] ([xshift=-5mm,yshift=-5mm]41.center) to [out=-285,in=-78] ([xshift=-5mm,yshift=-5mm]43.center);
\draw[-{Latex[length=3mm]}] ([xshift=-5mm,yshift=-5mm]31.center) -- ([xshift=-5mm,yshift=-5mm]32.center);

\end{tikzpicture} 

\vspace{0.7 cm}

\begin{tikzpicture}[node distance={10.5 mm}, thick, main/.style = {draw, circle,minimum size=2 mm}, 
blank/.style={circle, draw=green!0, fill=green!0, very thin, minimum size=3.5mm},]

\node[main] (1) {$1'$};
\node[main] (2) [above of=1] {$4'$};
\node[main] (3) [above of = 2] {$3'$}; 
\node[main] (4) [above of=3] {$2'$};
\node[main](5) [above of = 4] {$5'$};
\node (blank)[above of = 5]{};
\node[main](55) [right of = blank]{$5$};
\node(54) [below of = 55]{};
\node(53) [below of = 54]{};
\node(52) [below of = 53]{};
\node(51) [below of = 52]{};
\node[main] (44) [right  of=55] {$4$};
\node(43) [below of = 44]{};
\node(42) [below of = 43]{};
\node(41) [below of = 42]{};
\node[main] (33) [right of = 44] {$3$};
\node(32) [below of = 33]{};
\node(31) [below of = 32]{};
\node[main] (22) [right of = 33] {$2$};
\node(21) [below of = 22]{};
\node[main] (11) [right of = 22] {$1$};

\draw[] (5) -- ([xshift=-5mm,yshift=-5mm]55.center);
\draw[-{Latex[length=3mm]},dashed] ([xshift=-5mm,yshift=-5mm]55.center) -- (55);
\draw[dashed] (4) -- ([xshift=-5mm,yshift=-5mm]54.center);
\draw[] ([xshift=-5mm,yshift=-5mm]54.center)--([xshift=-5mm,yshift=-5mm]44.center);
\draw[-{Latex[length=3mm]}, dashed] ([xshift=-5mm,yshift=-5mm]44.center) -- (44);
\draw[dashed] (3) -- ([xshift=-5mm,yshift=-5mm]43.center);
\draw[-{Latex[length=3mm]}] ([xshift=-5mm,yshift=-5mm]43.center)--(33);
\draw[] (2) -- ([xshift=-5mm,yshift=-5mm]42.center);
\draw[-{Latex[length=3mm]}] ([xshift=-5mm,yshift=-5mm]42.center)--(22);
\draw[dashed](1)--([xshift=-5mm,yshift=-5mm]31.center);
\draw[-{Latex[length=3mm]},dashed] ([xshift=-5mm,yshift=-5mm]31.center)--(11);

\draw[-{Latex[length=3mm]},dashed] ([xshift=-5mm,yshift=-5mm]54.center) -- ([xshift=-5mm,yshift=-5mm]55.center);
\draw[-{Latex[length=3mm]}, dashed] ([xshift=-5mm,yshift=-5mm]43.center) -- ([xshift=-5mm,yshift=-5mm]44.center);
\draw[-{Latex[length=3mm]}] ([xshift=-5mm,yshift=-5mm]51.center) to [out=-285,in=-78] ([xshift=-5mm,yshift=-5mm]54.center);
\draw[-{Latex[length=3mm]}] ([xshift=-5mm,yshift=-5mm]41.center) to [out=-285,in=-78] ([xshift=-5mm,yshift=-5mm]43.center);
\draw[-{Latex[length=3mm]}] ([xshift=-5mm,yshift=-5mm]31.center) -- ([xshift=-5mm,yshift=-5mm]32.center);

\end{tikzpicture} 
\begin{tikzpicture}[node distance={10.5 mm}, thick, main/.style = {draw, circle,minimum size=2 mm}, 
blank/.style={circle, draw=green!0, fill=green!0, very thin, minimum size=3.5mm},]

\node[main] (1) {$1'$};
\node[main] (2) [above of=1] {$4'$};
\node[main] (3) [above of = 2] {$3'$}; 
\node[main] (4) [above of=3] {$2'$};
\node[main](5) [above of = 4] {$5'$};
\node (blank)[above of = 5]{};
\node[main](55) [right of = blank]{$5$};
\node(54) [below of = 55]{};
\node(53) [below of = 54]{};
\node(52) [below of = 53]{};
\node(51) [below of = 52]{};
\node[main] (44) [right  of=55] {$4$};
\node(43) [below of = 44]{};
\node(42) [below of = 43]{};
\node(41) [below of = 42]{};
\node[main] (33) [right of = 44] {$3$};
\node(32) [below of = 33]{};
\node(31) [below of = 32]{};
\node[main] (22) [right of = 33] {$2$};
\node(21) [below of = 22]{};
\node[main] (11) [right of = 22] {$1$};

\draw[] (5) -- ([xshift=-5mm,yshift=-5mm]55.center);
\draw[-{Latex[length=3mm]},dashed] ([xshift=-5mm,yshift=-5mm]55.center) -- (55);
\draw[dashed] (4) -- ([xshift=-5mm,yshift=-5mm]54.center);
\draw[] ([xshift=-5mm,yshift=-5mm]54.center)--([xshift=-5mm,yshift=-5mm]44.center);
\draw[-{Latex[length=3mm]}, dashed] ([xshift=-5mm,yshift=-5mm]44.center) -- (44);
\draw[dashed] (3) -- ([xshift=-5mm,yshift=-5mm]43.center);
\draw[-{Latex[length=3mm]}] ([xshift=-5mm,yshift=-5mm]43.center)--(33);
\draw[] (2) -- ([xshift=-5mm,yshift=-5mm]42.center);
\draw[] ([xshift=-5mm,yshift=-5mm]42.center) -- ([xshift=-5mm,yshift=-5mm]32.center);
\draw[-{Latex[length=3mm]}, dashed] ([xshift=-5mm,yshift=-5mm]32.center)--(22);
\draw[dashed](1)--([xshift=-5mm,yshift=-5mm]31.center);
\draw[-{Latex[length=3mm]}] ([xshift=-5mm,yshift=-5mm]31.center)--(11);

\draw[-{Latex[length=3mm]},dashed] ([xshift=-5mm,yshift=-5mm]54.center) -- ([xshift=-5mm,yshift=-5mm]55.center);
\draw[-{Latex[length=3mm]}, dashed] ([xshift=-5mm,yshift=-5mm]43.center) -- ([xshift=-5mm,yshift=-5mm]44.center);
\draw[-{Latex[length=3mm]}] ([xshift=-5mm,yshift=-5mm]51.center) to [out=-285,in=-78] ([xshift=-5mm,yshift=-5mm]54.center);
\draw[-{Latex[length=3mm]}] ([xshift=-5mm,yshift=-5mm]41.center) to [out=-285,in=-78] ([xshift=-5mm,yshift=-5mm]43.center);
\draw[-{Latex[length=3mm]}, dashed] ([xshift=-5mm,yshift=-5mm]31.center) -- ([xshift=-5mm,yshift=-5mm]32.center);

\end{tikzpicture} 
\caption{The four extremal path collections consisting of $3$ paths in $G_{v,w}$, indicated with dashed lines, for $v=s_2s_3s_2$ and $w=s_1s_2s_3s_4s_1s_2s_3s_1$ in $\mathfrak{S}_5$.}
\label{examplegraphical}
 \end{figure}

 To find the next extremal index, we find the topmost source vertex for which the path originating at that vertex does not attain the leftmost possible sink. In this case, that is $2'$. Swapping the diagonal path originating at $2'$ for a path which does attain the leftmost possible sink, we get the dashed paths in the top right of \Cref{examplegraphical}. Thus, the next extremal index is $\Xi_{v,w}(I)=\{1,3,5\}$, the sink set of these paths. Observe that this is indeed the unique non-intersecting path collection with source set $[3]'$ and sink set $\{1,3,5\}$, as asserted by \Cref{extremunique}, and consists of 2 diagonal paths and a greedy path. To find the extremal index $\Xi_{v,w}^2(I)$, we replace the path collection originating from $\{2',3'\}$ by a non-intersecting path collection with Gale maximal sink set. We obtain $\Xi_{v,w}^2(I)=\{1,4,5\}$. This is illustrated in the bottom left of \Cref{examplegraphical}. Similarly, to find the extremal index $\Xi_{v,w}^3(I)$, we replace the entire path collection by a non-intersecting path collection with Gale maximal sink set. We obtain $\Xi_{v,w}^3(I)=\{2,4,5\}$. This is illustrated in the bottom right of \Cref{examplegraphical}. Observe that the path collections illustrated in the two bottom panels of \Cref{examplegraphical} are both unions of greedy path collections and diagonal paths. Moreover, they are the unique non-intersecting path collections from $[3]'$ to their respective sink sets.

 \end{example}

 \subsection{Determining Non-Extremal \Pl Coordinates}\label{sect:3termgenerate}

Our next goal is to prove that the extremal non-zero and non-infinite \Pl coordinates determine the rest of the \Pl coordinates in the totally nonnegative flag variety and totally nonnegative Dressian, respectively.

 \begin{theorem}\label{3termgenerate}
 Let $v\leq w\in \mathfrak{S}_n$. For any flag $F$ in $\textnormal{Fl}_n^{\geq 0}$, the extremal non-zero \Pl coordinates of $F$ uniquely determine the other non-zero \Pl coordinates of $F$ by three-term incidence \Pl relations. 
 \end{theorem}
 
 \begin{proof}
The proof is identical to the proof of \Cref{topcell3termgenerate}, except for the following changes:

\begin{enumerate}
    \item The definition of the statistic $t(I)$ is unchanged, but what it represents changes slightly. It is now the largest integer $\tau$ such that, among the strands $\{\sigma_1<\sigma_2<\cdots< \sigma_k\}$ containing vertices in the source set $[k]'$ in $G_{v,w}$,  we have $\left\{\sigma_1,\ldots, \sigma_\tau\right\}\subset I$. 

    \item All subscripts $id,w_0$ should be replaced by subscripts $v,w$.

    \item Some justifications need to be changed. Every time we say there are no arrows directed downwards in $G_{v,w}$, we rely on \Cref{vertarrow}. Instead of \Cref{topcellcor}, we refer to \Cref{extremunique}.

    \item The proof of \Cref{topcell3termgenerate} uses planarity to justify that a non-intersecting path collection whose non-diagonal paths achieve a Gale maximal sink set is a union of a greedy path collection and a diagonal path collection. We replace this argument by a combination of \Cref{galemax} and \Cref{extremunique}. 
\end{enumerate}
 \end{proof}

 \begin{cor}\label{generalpostest}
 For any flag $F$ in $\textnormal{Fl}_n^{\geq 0}$, the extremal \Pl coordinates serve as a positivity test for the non-zero \Pl coordinates. Explicitly, if the extremal non-zero \Pl coordinates of $F$ are positive, then so are all the other non-zero \Pl coordinates of $F$.
 \end{cor} 
 
 \begin{proof}
This is similar to the proof of \Cref{postest}.
 \end{proof} 
 
\begin{thmbis}{3termgenerate}\label{tropical3termgenerate}
 Let $v\leq w\in \mathfrak{S}_n$. For any point $P$ in $\textnormal{FlDr}_{v,w}^{>0}$, the extremal non-infinite \Pl coordinates of $P$ uniquely determine the other non-infinite \Pl coordinates of $P$. 
 \end{thmbis}
 
 \begin{proof}
This is similar to the proof of \Cref{tropicaltopcell3termgenerate}.
 \end{proof}
 
\begin{definition}
    A \textbf{flag of subsets} of $[n]$ is a collection of subsets $(R_1,\ldots, R_n)$ of $[n]$ such that $R_1\subsetneq R_2\subsetneq\cdots \subsetneq R_n=[n]$. In particular, $|R_i|=i$.
\end{definition} 

\begin{proposition}\label{extremalflag}
    For any $v\leq w\in \mathfrak{S}_n$, each extremal index of a flag $F$ in $\mathcal{R}_{v,w}^{>0}$ lies in a flag of subsets of $[n]$ consisting of extremal indices of $F$. 
\end{proposition}

\begin{proof}
    Fix an extremal index $I$ with $|I|=t$. By \Cref{extremunique}, there is a corresponding unique non-intersecting path collection $C$. If $t=1$, then there is no need to check that a subset of $I$ is extremal. Otherwise, apply \Cref{lem:constructionstar} with $o=t'$ to obtain a new non-intersecting path collection $D$. The path collection $D$ is also extremal, by \Cref{extremunique}. Moreover, its sink set is of the form $I\setminus \kappa$ for some $\kappa\in I$.

    If $t=n-1$, there is no need to check that $I$ is contained in an extremal index. Otherwise, consider the extremal path collection $C'$ originating from $[t+1]'$ where a path is greedy if and only if the path originates weakly above the origin of the bottom-most greedy path of $C$. Say this path collection has sink set $J$. By \Cref{extremunique}, $J$ is extremal. Applying \Cref{lem:constructionstar} to $C'$ with $o=(t+1)'$, we recover precisely the original path collection $C$. Thus, $J$ contains $I$. 
    
    In summary, we have shown that for $|I|>1$, $I$ contains an extremal index of size $|I|-1$ and, for $|I|<n-1$, $I$ is contained in an extremal index of size $|I|+1$. Thus, there is a flag of extremal indices containing $I$.
\end{proof}

Armed with \Cref{3termgenerate}, \Cref{tropical3termgenerate}, and \Cref{extremalflag}, we revisit the two notions of flag positroid we introduced earlier, namely, realizable and synthetic. Recall that a flag positroid is realizable if it is the support of a point in $\textnormal{Fl}_{n}^{\geq{0}}$ and synthetic if it is the support of a point in $\textnormal{Fl}_n$ whose \Pl coordinates are all nonnegative. In particular, the support of any flag in $\mathcal{R}_{v,w}^{>0}$ is a realizable flag positroid which depends only on $v$ and $w$. We denote it by $\bm{\mathcal{M}^{v,w}}=(\mathcal{M}_1^{v,w},\ldots,\mathcal{M}_n^{v,w})$.
  
  \begin{lemma}\label{synthetic}
  The set of synthetic flag positroids on $[n]$ equals the set of realizable flag positroids on $[n]$.
  \end{lemma}
  \begin{proof}
  
  We have established in \Cref{forwarddir} that if $F\in \textnormal{Fl}_n^{\geq 0}$ then $P_I(F)\geq 0$ for all $\emptyset\neq I\subset [n]$, so all realizable flag positroids are synthetic flag positroids. 
  
  We now show that synthetic flag positroids are realizable. Let $F$ be a flag such that $P_I=P_I(F)\geq 0$ for all $\emptyset\neq I\subset[n]$. Let $\bm{\mathcal{M}}=(\mathcal{M}_1,\ldots, \mathcal{M}_n)$ be the synthetic flag positroid given by the indices of the non-zero coordinates of $P$. For any flag matroid, the Gale minimal bases of each size form a flag of subsets in $[n]$ and similarly for the Gale maximal bases \cite[Section 1.8]{Coxeter}. Consider the two flags of subsets $F_I=\left\{\{i_1\}\subset \{i_1,i_2\}\subset\cdots\subset \{i_1,\cdots, i_n\}=[n]\right\}$ and $F_J=\left\{\{j_1\}\subset \{j_1,j_2\}\subset\cdots\subset \{j_1,\cdots, j_n\}=[n]\right\}$ corresponding to the Gale minimal and Gale maximal bases of $\bm{\mathcal{M}}$, respectively. Let $v,w\in \mathfrak{S}_n$ be defined by $v^{-1}(k)=i_k$ and $w^{-1}(k)=j_k$ for $k\in[n]$. We will show that $\bm{\mathcal{M}}=\bm{\mathcal{M}^{v,w}}$. 
  
  We first show that the extremal indices of $P$ coincide with those of $Q=P(F')$ for a flag $F'\in\mathcal{R}_{v,w}^{> 0}$. In a complete flag matroid, by definition, each basis lies in a flag of subsets of $[n]$ consisting of bases of the constituent matroids. Thus, if we are to have $P_K\neq 0$ for some index $K$, we must (at least) have that $K=\{k_1,\ldots, k_m\}$ lies in some flag of subsets $F=\left\{\{k_1\}\subset \{k_1,k_2\}\subset\cdots\subset [n]\right\}$ satisfying $\{i_1,\ldots,i_t\}\leq \{k_1,\ldots,k_t\}\leq \{j_1,\ldots,j_t\}$ for each $t\in [n]$, by the minimality and maximality of $\{i_1,\ldots,i_t\}$ and $\{j_1,\ldots,j_t\}$, respectively. By \Cref{bruhatinterval}, $Q_K\neq 0$ for each such $K$. Thus, any basis of $\bm{\mathcal{M}}$ is guaranteed to be a basis of $\bm{\mathcal{M}^{v,w}}$. We are left to show that for any extremal index $L$ of $Q$, that is, of $\mathbf{\mathcal{M}^{v,w}}$, we have $P_L\neq 0$. 
  
  For indices of size $1$, the extremal indices are simply the minimal and maximal indices of non-zero \Pl coordinates. These are $i_1$ and $j_1$, respectively, for both $P$ and $Q$. For larger indices, we argue by contradiction. Suppose $l>1$ is minimal such that an extremal index $L_1$ of size $l$ of $Q$ satisfies $P_{L_1}=0$. Choose $L_1$ to be the Gale minimal such index. Since $L_1$ is an extremal index of $Q$, \Cref{extremalflag} tells us that $L_1\setminus d$ is extremal for some $d\in L_1$. By the minimality of $l$, we must have that $P_{{L_1}\setminus d}\neq 0$. We will use \Pl relations to derive a contradiction in two stages, first showing $d$ must be suitably large and then showing that such a $d$ cannot satisfy the \Pl equations.
  
  Recall that, for each $t\in [n]$, the Gale minimal and Gale maximal indices of size $t$ of non-zero \Pl coordinates of $P$ and of $Q$ coincide, being $\{i_1,\ldots, i_t\}$ and $\{j_1,\ldots, j_t\}$, respectively. Thus, $L_1$ cannot be either the Gale minimal 
 or the Gale maximal extremal index of size $l$. Thus, we know that $L_1$ is of the form $\Xi_{v,w}(L_2)=(L_2\setminus b)\cup c$ for some $c>b$ and some other extremal index $L_2$ of $Q$ which is Gale less than $L_1$. Moreover, since $L_1$ is not Gale maximal, we claim it is impossible that $b$ is the minimal element of $L_2$. By the characterization of extremal path collections in \Cref{extremunique}, the unique extremal path collection $C$ in $G_{v,w}$ with sink set $L_2$ is a union of diagonal paths and a greedy path collection. To construct the extremal path collection $C'$ with set $L_1$, we identify the topmost diagonal path of $C$ which is not already greedy and replace it with a greedy path. If $b$ were the minimal element of $L_2$, it would have to be the bottom-most diagonal path in $C$. This would imply $C'$ is a greedy path collection, which would in turn mean $L_1$ is Gale maximal. Since we assume $L_1$ is not Gale maximal, we must have that $b$ is not the minimal element of $L_2$.
 
 By Gale minimality of $L_1$, we have, $P_{L_2}\neq 0$. Since $b$ is not minimal in $L_2$, we can pick $a<b$ in $L_2$. Then, $a\in L_1=(L_2\setminus b)\cup c$. We have the three-term incidence \Pl relation $P_{L_1}P_{({L_1}\setminus ac)\cup b}=P_{({L_1}\setminus a)\cup b}P_{{L_1}\setminus c}+P_{L_2}P_{{L_1}\setminus a}$. Since all the \Pl coordinates are nonnegative, and $P_{L_1}=0$ and $P_{L_2}\neq 0$, we must have $P_{{L_1}\setminus a}=0$ for any $a\in L_1$ satisfying $a<b$. Thus, since $P_{L_1\setminus d}\neq 0$, we have $d>b$.      

Note that by \Cref{addins} applied to $Q$, $c$ must be in the Gale maximal extremal index $N=\{j_1,\cdots, j_l\}$ of size $l$ of $Q$. Since the Gale maximal indices of non-zero \Pl coordinates of $P$ and of $Q$ coincide, $P_N\neq 0$. By dual basis exchange between $N$ and $L_2$ in $\mathcal{M}_l$, there must be some $b'\in L_2\setminus N$ such that $P_{(L_2\setminus b')\cup c}\neq 0$. Since $P_{L_1}=0$, $b'\neq b$. Moreover, we cannot have $b'>b$ since $P_{(L_2\setminus b')\cup c}\neq 0$ implies $Q_{(L_2\setminus b')\cup c}\neq 0$ and this would contradict the fact that $\Xi_{v,w}(L_2)=(L_2\setminus b)\cup c$, which imposes a maximality condition on $b$. Thus, $b'<b$. We consider the incidence \Pl relation $P_{{L_1}}P_{{(L_1\setminus b'd)\cup b}}=P_{(L_1\setminus b')\cup b}P_{{L_1}\setminus d}+P_{({L_1}\setminus d)\cup b}P_{{L_1}\setminus b'}$. We know that $P_{L_1}=0$ but $P_{(L_1\setminus b')\cup b}\neq 0$ since $(L_1\setminus b')\cup b=(L_2\setminus b')\cup c$. This contradicts the assumption that $P_{L_1}=0$, since all the \Pl coordinates are nonnegative.

Finally, we prove $\bm{\mathcal{M}}=\bm{\mathcal{M}^{v,w}}$: As proven earlier, any basis of $\bm{\mathcal{M}}$ is a basis of $\bm{\mathcal{M}^{v,w}}$ by \Cref{bruhatinterval}. Let $I$ be the index of a non-zero \Pl coordinate of $Q$, that is, a basis of $\bm{\mathcal{M}^{v,w}}$. We show $P_I\neq 0$, that is, $I$ is a basis of $\bm{\mathcal{M}}$. By \Cref{3termgenerate}, the extremal non-zero \Pl coordinates of $F'\in\mathcal{R}_{v,w}^{> 0}$ determine $Q_I$ by three-term incidence \Pl relations. In particular, the proof of \Cref{3termgenerate} gives a way to express $Q_I$ as a subtraction-free rational expression in the extremal non-zero \Pl coordinates. Since $P$ is subject to the three-term relations too, $P_I$ has the same expressions as a subtraction-free rational expression in the extremal non-zero \Pl coordinates. Since the extremal non-zero \Pl coordinates of $F$ must be positive by definition, $P_I\neq 0$. Thus,  $\bm{\mathcal{M}}=\bm{\mathcal{M}^{v,w}}$ and so, $\bm{\mathcal{M}}$ is a realizable flag positroid.   
  \end{proof}
  
The previous theorem says that the support of any complete flag with nonnegative \Pl coordinates is a realizable flag positroid. The following result will analogously say that the support of any point in the totally nonnegative complete flag Dressian is a realizable flag positroid.

\begin{lemmabis}{synthetic}\label{tropsynthetic}
The totally nonnegative complete flag Dressian decomposes as $\textnormal{FlDr}_n^{\geq 0}=\sqcup_{v\leq w}\textnormal{FlDr}_{v,w}^{> 0}$, where the disjoint union is over all pairs $v\leq w\in \mathfrak{S}_n$.
\end{lemmabis}
\begin{proof}
By {\cite[Section 4.2]{BEZ}}, the support of any point in $\textnormal{FlDr}_n$ is a flag matroid. For the reverse direction, the previous proof applies with the following modifications: $P$ should lie in $\textnormal{FlDr}^{\geq 0}$, $Q$ should lie in $\textnormal{FlDr}^{\geq0}_{v,w}$, and all equations should be positively tropicalized, with instances of $0$ replaced by $\infty$. When we derive contradictions from incidence \Pl relations, it is in such a way that the tropical analogue also leads to a contradiction. The argument of the last paragraph holds with \Cref{tropical3termgenerate} in place of \Cref{3termgenerate}. 
\end{proof}
  
\subsection{Determining the Parameters}

Our next goal is to show that any flag in $\textnormal{Fl}_n$ with \Pl coordinates which are all nonnegative lies in $\textnormal{Fl}_n^{\geq 0}$. Recall the maps $\Phi_{v,w}$ and $\Psi_{v,w}$ from \Cref{parameterization}. Our argument will be similar to the top cell case in \Cref{inverseLaurent} but in this more general setting, we will have to exercise caution because our extremal non-zero \Pl coordinates are not always algebraically independent.  

We require the following preliminaries. For the rest of this section, we fix $v\leq w\in \mathfrak{S}_n$ and $G=G_{v,w}$.

\begin{definition}
    Let $r\in [n]$. Let $L_r$ be the greedy path collection in $G=G_{v,w}$ with source set $[r]'$. This is the path collection whose sink set is the Gale maximal extremal index of size $r$. 
\end{definition}

\begin{lemma}\label{addingedges}
    Suppose $t_2>t_1$ and there is a path which is blocked from using the edge $e$ in $L_{t_1}$ --- In other words, there is a path in $L_{t_1}$ which uses the diagonal strand below $e$ but does not turn onto edge $e$ because some other path uses the diagonal strand at the top of $e$. Then edge $e$ is not used in $L_{t_2}$.
    \end{lemma}
    
    \begin{proof}
  
    By repeated application of the last statement of \Cref{lem:constructionstar}, the section of the diagonal strand at the top of edge $e$ must also be used in $L_{t_2}$ and thus the edge $e$ cannot be used in $L_{t_2}$. 
\end{proof}

Recall the order $\prec$ on $2^{[n]}$ defined by $I\prec J$ if $|I|>|J|$ or if $|I|=|J|$ and $I$ is lexicographically smaller than $J$. Using this order, we inductively define a subset $S_{v,w}$ of the extremal non-zero indices of a flag in $\mathcal{R}_{v,w}^{>0}$. Start with $S_{v,w}=\emptyset$. Then, going through the extremal non-zero \Pl coordinates according to the total order $\prec$, add $I$ to $S_{v,w}$ if $P_I$ is algebraically independent of the extremal non-zero \Pl coordinates $P_J$ with $J\prec I$.

We first claim that the \Pl coordinates with index in $S_{v,w}$ determine the other extremal non-zero \Pl coordinates by three-term incidence \Pl relations. These are algebraically independent by definition and so, we will be able to prove that any flag with nonnegative coordinates is in fact a totally nonnegative flag in a way that is very similar to \Cref{topcellnonnegPluckCoords}. Note that it is clear by construction that the coordinates whose indices lie in $S_{v,w}$ determine the other extremal non-zero \Pl coordinates. What will be important here (for the tropical version of our results) is that we can deduce the other extremal non-zero \Pl coordinates using just the three-term incidence \Pl relations.

\begin{lemma}\label{Svwgenerate}
For any flag $F\in \mathcal{R}_{v,w}^{>0}$, any extremal non-zero \Pl coordinate of $F$ whose index is not in $S_{v,w}$ is determined by the coordinates whose indices are in $S_{v,w}$ using three-term incidence \Pl relations.
\end{lemma}
\begin{proof}

Let $P=P(F)$. Suppose $I$ with $|I|=t$ is the index of an extremal non-zero \Pl coordinate of $F$ which is not in $S_{v,w}$. Note that if $I_{min}$ is the $\prec$ minimal index of a non-zero \Pl coordinate, then $P_{I_{min}}$ is always forced to equal $1$ since the corresponding non-intersecting path collection is diagonal. Thus, it suffices by induction to show that $P_I$ can be expressed, via three-term incidence \Pl relations, in terms of extremal non-zero \Pl coordinates $P_{J}$ with $J\prec I$

By \Cref{extremunique}, we can let $C$ be the unique non-intersecting path collection in $G_{v,w}$ with sink set $I$. We will assume not all paths in $C$ are diagonal, or equivalently, that $I$ is not Gale minimal amongst the size $t$ indices of non-zero \Pl coordinates, since if it were, $P_I=1$ and so $P_I$ is determined without any equations. Let $K$ be a set of vertices such that the paths in $C$ originating from $K$ form a greedy and left extreme path collection (that is, the topmost $k$ sources of $C$ for some $0\leq k\leq n$). It is possible that $K$ is not uniquely defined: If the bottom-most path of $K$ is diagonal because it is unable to make any left turns, it can be removed and the remaining smaller set still satisfies the definition. Thus, we may assume without loss of generality that the path originating from the bottom-most vertex of $K$ is not diagonal. Suppose the bottom-most vertex of $K$ lies on strand $a$, and the path originating on strand $a$ in $C$ terminates at $b>a$. Then, $I'=(I\setminus b)\cup a$ is also an extremal index and $I=\Xi_{v,w}(I')$. We will proceed by cases on where the source vertex $(t+1)'$ lies with respect to $K$. 

If the vertices in $K$ all lie above the vertex labeled $(t+1)'$, then consider the non-intersecting path collection $C'$ with source set $[t+1]'$ such that the subcollection originating from $K$ is a left extreme path collection and the subcollection originating from the remaining vertices is diagonal. The sink set will be the extremal index $I\cup f$, where the vertex labeled $(t+1)'$ lies on strand $f$. In this case, $f<a<b$. Note that $I'\cup f$ is an extremal index as well. Thus, we have the three-term incidence \Pl relation $P_{I'}P_{I\cup f}=P_{I}P_{I'\cup f}+P_{(I'\setminus a)\cup f}P_{I'\cup b}$. Note that, $P_{(I'\setminus a)\cup f}=0$ since in the unique non-intersecting path collection with sink set $I'$, the path with sink $a$ must be diagonal and so it cannot be replaced by a path with sink $f<a$. Since $I'$ is lexicographically less than $I$, we have determined $P_I$ according to the desired condition and in fact, we have determined it by a monomial relation. This situation is illustrated in \Cref{examplealgindependent1}.

    \begin{figure}[H]
 \begin{tikzpicture}[node distance={10.5 mm}, thick, main/.style = {draw, circle,minimum size=2 mm}, 
blank/.style={circle, draw=green!0, fill=green!0, very thin, minimum size=3.5mm},]

\node[main] (1) {$3'$};
\node[main] (2) [above of=1] {$4'$};
\node[main] (3) [above of = 2] {$1'$}; 
\node[main] (4) [above of=3] {$2'$};
\node[main](5) [above of = 4] {$5'$};
\node (blank)[above of = 5]{};
\node[main](55) [right of = blank]{$5$};
\node(54) [below of = 55]{};
\node(53) [below of = 54]{};
\node(52) [below of = 53]{};
\node(51) [below of = 52]{};
\node[main] (44) [right  of=55] {$4$};
\node(43) [below of = 44]{};
\node(42) [below of = 43]{};
\node(41) [below of = 42]{};
\node[main] (33) [right of = 44] {$3$};
\node(32) [below of = 33]{};
\node(31) [below of = 32]{};
\node[main] (22) [right of = 33] {$2$};
\node(21) [below of = 22]{};
\node[main] (11) [right of = 22] {$1$};

\draw[] (5) -- ([xshift=-5mm,yshift=-5mm]55.center);
\draw[-{Latex[length=3mm]}, dashed] ([xshift=-5mm,yshift=-5mm]55.center) -- (55);
\draw[dashed] (4) -- ([xshift=-5mm,yshift=-5mm]54.center);
\draw[] ([xshift=-5mm,yshift=-5mm]54.center)--([xshift=-5mm,yshift=-5mm]44.center);
\draw[-{Latex[length=3mm]}, dashed] ([xshift=-5mm,yshift=-5mm]44.center) -- (44);
\draw[dashed] (3) -- ([xshift=-5mm,yshift=-5mm]43.center);
\draw[-{Latex[length=3mm]}] ([xshift=-5mm,yshift=-5mm]43.center)--(33);
\draw[-{Latex[length=3mm]}] (2)--(22);
\draw[-{Latex[length=3mm]},dotted] (1)--(11);

\draw[-{Latex[length=3mm]}, dashed] ([xshift=-5mm,yshift=-5mm]54.center) -- ([xshift=-5mm,yshift=-5mm]55.center);
\draw[-{Latex[length=3mm]}] ([xshift=-5mm,yshift=-5mm]53.center) -- ([xshift=-5mm,yshift=-5mm]54.center);
\draw[-{Latex[length=3mm]},dashed] ([xshift=-5mm,yshift=-5mm]43.center) -- ([xshift=-5mm,yshift=-5mm]44.center);
\draw[-{Latex[length=3mm]}] ([xshift=-5mm,yshift=-5mm]31.center) -- ([xshift=-5mm,yshift=-5mm]32.center);
\draw[-{Latex[length=3mm]}] ([xshift=-5mm,yshift=-5mm]32.center) -- ([xshift=-5mm,yshift=-5mm]33.center);
\draw[-{Latex[length=3mm]}] ([xshift=-5mm,yshift=-5mm]21.center) -- ([xshift=-5mm,yshift=-5mm]22.center);

\end{tikzpicture} 
\caption{Illustration of the case where $(t+1)'$ lies below $K$ in the case $v=34125$, $w=w_0$ and $t=2$. The dashed path collection is the unique non-intersecting path collection whose sink set is the extremal index $I=\{4,5\}$. In this case, the bottom-most vertex in $K$ is $1'$ lying on strand $a=3$ and the path originating at $1'$ terminates at $b=4$. Thus, $I'=\{3,5\}$. The dotted path shows the path originating at $3'$ and terminating at $f=1$. Note that $I\cup f$, and $I'\cup f$ are both easily seen to be extremal indices by our characterization of extremal path collections as unions of left extreme and diagonal path collections in \Cref{extremunique}.}
\label{examplealgindependent1}
 \end{figure}

Otherwise, $K$ contains vertices below the vertex labeled $(t+1)'$. Consider the non-intersecting path collection $C'$ of size $t+1$ where the paths originating at $K\cup \{(t+1)'\}$ are greedy and the rest are diagonal. This is precisely the construction we used in \Cref{extremalflag} to show that every extremal index is contained in a larger extremal index, and so the sink set of $C'$ must be of the form $I\cup d$ for some $d\in[n]$. Furthermore, since the lowest strand in $K$ is $a$, we have $d\geq  a$. 

We first consider the case where $d>b>a$. In this case, we have a three-term incidence \Pl relation which says that $P_IP_{I'\cup d}=P_{I'}P_{I\cup d}+P_{(I\setminus b)\cup d}P_{I\cup a}$. Observe that since $d>b$ and the paths originating from $K$ are already left extreme in $C$, $P_{(I\setminus b)\cup d}=0$. Also, since extremal path collections are unions of greedy and diagonal path collections and $I\cup d$ is extremal, $I'\cup d$ is an extremal index as well. Since $I'$ is lexicographically less than $I$, we have determined $P_I$ according to the desired condition and in fact, we have determined it by a monomial relation. This situation is illustrated in \Cref{examplealgindependent2}.

    \begin{figure}[H]
 \begin{tikzpicture}[node distance={10.5 mm}, thick, main/.style = {draw, circle,minimum size=2 mm}, 
blank/.style={circle, draw=green!0, fill=green!0, very thin, minimum size=3.5mm},]

\node[main] (1) {$1'$};
\node[main] (2) [above of=1] {$3'$};
\node[main] (3) [above of = 2] {$4'$}; 
\node[main] (4) [above of=3] {$2'$};
\node[main](5) [above of = 4] {$5'$};
\node (blank)[above of = 5]{};
\node[main](55) [right of = blank]{$5$};
\node(54) [below of = 55]{};
\node(53) [below of = 54]{};
\node(52) [below of = 53]{};
\node(51) [below of = 52]{};
\node[main] (44) [right  of=55] {$4$};
\node(43) [below of = 44]{};
\node(42) [below of = 43]{};
\node(41) [below of = 42]{};
\node[main] (33) [right of = 44] {$3$};
\node(32) [below of = 33]{};
\node(31) [below of = 32]{};
\node[main] (22) [right of = 33] {$2$};
\node(21) [below of = 22]{};
\node[main] (11) [right of = 22] {$1$};

\draw[] (5) -- ([xshift=-5mm,yshift=-5mm]55.center);
\draw[-{Latex[length=3mm]}, dashed] ([xshift=-5mm,yshift=-5mm]55.center) -- (55);
\draw[dashed] (4) -- ([xshift=-5mm,yshift=-5mm]54.center);
\draw[] ([xshift=-5mm,yshift=-5mm]54.center)--([xshift=-5mm,yshift=-5mm]44.center);
\draw[-{Latex[length=3mm]},dotted] ([xshift=-5mm,yshift=-5mm]44.center) -- (44);
\draw[] (3) -- ([xshift=-5mm,yshift=-5mm]43.center);
\draw[-{Latex[length=3mm]}] ([xshift=-5mm,yshift=-5mm]43.center)--(33);
\draw[dotted] (2) -- ([xshift=-5mm,yshift=-5mm]42.center);
\draw[-{Latex[length=3mm]}] ([xshift=-5mm,yshift=-5mm]42.center)--(22);
\draw[dashed](1)--([xshift=-5mm,yshift=-5mm]31.center);
\draw[-{Latex[length=3mm]}] ([xshift=-5mm,yshift=-5mm]31.center)--(11);

\draw[-{Latex[length=3mm]}, dashed] ([xshift=-5mm,yshift=-5mm]54.center) -- ([xshift=-5mm,yshift=-5mm]55.center);
\draw[-{Latex[length=3mm]}] ([xshift=-5mm,yshift=-5mm]51.center) to [out=-285,in=-78] ([xshift=-5mm,yshift=-5mm]54.center);
\draw[-{Latex[length=3mm]},dotted] ([xshift=-5mm,yshift=-5mm]43.center) -- ([xshift=-5mm,yshift=-5mm]44.center);
\draw[-{Latex[length=3mm]},dotted] ([xshift=-5mm,yshift=-5mm]42.center) -- ([xshift=-5mm,yshift=-5mm]43.center);
\draw[-{Latex[length=3mm]},dashed] ([xshift=-5mm,yshift=-5mm]31.center) -- ([xshift=-5mm,yshift=-5mm]32.center);
\draw[-{Latex[length=3mm]},dashed] ([xshift=-5mm,yshift=-5mm]32.center) -- ([xshift=-5mm,yshift=-5mm]33.center);
\draw[-{Latex[length=3mm]}] ([xshift=-5mm,yshift=-5mm]21.center) -- ([xshift=-5mm,yshift=-5mm]22.center);

\end{tikzpicture} 
\caption{Illustration of the case where $(t+1)'$ does not lie below $K$ in the case $v=13425$, $w=54231$ and $t=2$. The dashed path collection is the unique non-intersecting path collection whose sink set is the extremal index $I=\{3,5\}$. In this case, the bottom-most vertex in $K$ is $1'$ lying on strand $a=1$ and the path originating at $1'$ terminates at $b=3$. The dotted path shows the path originating at $3'$ and terminating at $d=4>b$. Note that $I\cup d$, $(I\setminus b)\cup a$ and $(I\setminus b)\cup ad$ are all easily seen to be extremal indices by our characterization of extremal path collections as unions of left extreme and diagonal path collections in \Cref{extremunique}.}
\label{examplealgindependent2}
 \end{figure}

Finally, we show by contradiction that we cannot have $b>d\geq a$. We use two subcases. For the first subcase, we suppose the path originating on strand $a$ in $C$ is identical to the path originating on strand $a$ in $C'$. If $d=a$, we must have a diagonal path on strand $a$ in $C'$ and thus also in $C$. This contradicts our definition of $K$. Therefore, in this subcase, $d> a$. We study the three-term relation $P_{(I\setminus b)\cup d}P_{I\cup a}=P_{I'\cup d}P_I+P_{I'}P_{I\cup d}$. We first show that $P_{I\cup a}=0$. To see this, observe that all the paths originating below strand $a$ are already diagonal in $C$, which has sink set $I$. Moreover, the vertex labeled $(t+1)'$ lies above strand $a$. Thus, the only way to have a non-intersecting path collection originating at $[t+1]'$ with sink set $I\cup a$ is to have a diagonal path on strand $a$. If we denote by $J$ the sinks of paths originating from $K$ in $C$, then the paths originating above strand $a$ in $C'$ would need to reach sink set $J$. However, in $C'$, the paths originating above strand $a$ are left extreme and achieve sink set $(J\setminus b)\cup d$, which is Gale less than $J$. Thus, it is impossible to have the desired non-intersecting path collection originating above strand $a$ and terminating at $J$, and so $P_{I\cup a}=0$. However, by definition of $I'$ and $d$, $P_{I'}P_{I\cup d}\neq 0$. Moreover, all \Pl are nonnegative. Thus, this case can never occur. 

Now, we consider the subcase where the path $p$ originating on strand $a$ in $C$ differs from the path $p'$ originating on strand $a$ in $C'$. By greediness, it must be the case that $p$ takes a left turn on a vertical edge $e$ which $p'$ is blocked from using. We claim that $e$ cannot be used by any graph extremal path collection consisting of at least $t+1$ paths. To show this, recall that $L_\tau$ denotes the greedy path collection originating from $[\tau]'$. Then, $e$ is not used in $L_{t+1}$, since whichever path blocks $p'$ from using edge $e$ in $C'$ will also block edge $e$ from being used in $L_{t+1}$. Consequently, by \Cref{addingedges}, $e$ is not used in $L_{\tau}$ for any $\tau>t$. Also, any vertical edge used by an extremal path collection consisting of $\tau$ paths is used by $L_\tau$, since that vertical edge must be used by the greedy part of the extremal path collection. Thus, $e$ is not used by any extremal path collection consisting of at least $t+1$ paths. Further, consider the unique non-intersecting path collection $\tilde{C}$ with sink set $\tilde{I}$ for any extremal $\tilde{I}$ which is of size $t$ and is Gale less than $I$. The path collection $\tilde{C}$ is obtained as the union of a diagonal path collection and the left extreme path collection originating from the top $\tilde{r}$ vertices of $K$ for some $\tilde{r}<|K|$. Since it is the path originating from the bottom-most strand $a$ of $K$ that uses $e$ in $C$, the edge $e$ is not used by $\tilde{C}$. This means that $C$ uses an edge that was not used by any non-intersecting path collections with sink set $J\prec I$ and thus the extremal non-zero \Pl coordinate corresponding to this path collection must be algebraically independent of those earlier than it in $\prec$ order. This contradicts the fact that $I\notin S_{v,w}$.

Thus, we can determine any extremal non-zero \Pl coordinate with index not in $S_{v,w}$ from those whose indices are in $S_{v,w}$ by three-term incidence \Pl relations.
\end{proof}

\begin{lemmabis}{Svwgenerate}\label{tropicalSvwgenerate}
For any point $P\in \textnormal{FlDr}_{v,w}^{\geq 0}$, any extremal non-infinite \Pl coordinate of $P$ whose index is not in $S_{v,w}$ is determined from the coordinates whose indices are in $S_{v,w}$ by positive tropical three-term incidence \Pl relations.
\end{lemmabis}

\begin{proof}
Argue exactly as in the proof of \Cref{Svwgenerate}, observing that any time we use incidence \Pl relations to determine a previously unknown coordinate, we do so in such a way that the corresponding tropical coordinate is also uniquely determined by the corresponding positive tropical incidence \Pl relation. 
\end{proof}

\begin{proposition}\label{prop:oneextraweightvw}
    Let $P=\Phi_{v,w}(\bm{a})$. Let $N$ be the number of extremal indices of a flag in $\mathcal{R}_{v,w}^{>0}$. Then, there are sets $A_i$ and a total order $\prec$ on the extremal indices $\{E_1\prec \cdots \prec E_N\}$ for $\mathcal{R}_{v,w}^{>0}$ such that $P_{E_{i}}=\prod_{k\in A_i}a_k$ and $\left|A_i\setminus \left(\bigcup_{j<i}A_j\right)\right|\leq 1$.
\end{proposition}

\begin{proof}
    The proof is identical to the proof of \Cref{prop:oneextraweight}, with $id,w_0$ replaced by $v,w$ and with references to \Cref{topcellcor} and \Cref{monomialtopcellcor} replaced by references to \Cref{extremunique} and \Cref{extremuniquecor}, respectively.
\end{proof}

\begin{proposition}\label{generalinverseLaurent}
The map $\Psi_{v,w}$, which is inverse to $\Phi_{v,w}$ on $\mathcal{R}_{v,w}^{>0}$, consists of Laurent monomials in the \Pl coordinates whose indices are in $S_{v,w}$. 
\end{proposition}

\begin{proof}

The proof is nearly identical to the proof of \Cref{inverseLaurent}, but with some subtle changes, so we reproduce it here.

We use the notation of \Cref{prop:oneextraweight}. We additionally denote by $S_1\prec\cdots \prec S_{N'}$ the subsequence of $E_1\prec\cdots \prec E_N$ consisting of extremal indices which are in $S_{v,w}$. For each $i\in [N]$, let ${k_i}$ be the unique element of $A_i\setminus \left(\bigcup_{j<i}A_j\right)$ if this set is nonempty and let $k_i$ be $0$ otherwise. Note that, by the definition of $S_{v,w}$, this set difference can only be nonzero if $E_i=S_{i'}$ for some $i'$. We let $a_0=1$ so that $a_{k_i}$ is defined for all $i\in [N]$. If $a_{k_i}=1$, there is nothing to show. Accordingly, we now show by induction that when $a_{k_i}\neq 1$, we can express $a_{k_i}$ as a Laurent monomial in $\{P_{S_1},\ldots ,P_{S_{i'}}\}$, where $i'$ is such that $S_{i'}=E_i$. 

As a base case, $E_1$ is Gale minimal. Thus, the corresponding extremal path collection is a diagonal path collection and $A_1=\emptyset$. It follows that $a_{k_1}=1$ and there is nothing to show. 

For the induction step, suppose that for $l<i$, $a_{k_l}$ is expressible as a Laurent monomial in $\{P_{S_{1}},\ldots, P_{S_{l'}}\}$. Suppose $a_{k_i}\neq 1$. Then, \Cref{prop:oneextraweightvw} gives a way to express $a_{k_i}$ as a Laurent monomial in $P_{S_{i'}}$ and $\{a_{k_j}\mid j<i\}$. By induction, this gives a way to express $a_{k_i}$ as a Laurent monomial in $\{P_{S_{1}},\ldots, P_{S_{i'}}\}$. 

We are just left to show that every weight $a_i$ gets determined in this way. By \Cref{3termgenerate}, the extremal \Pl coordinates determine all other \Pl coordinates. Moreover, by definition of $S_{v,w}$, the \Pl coordinates with indices in $S_{v,w}$ determine all the extremal \Pl coordinates. Thus, those \Pl coordinates with indices in $S_{v,w}$ determine all other \Pl coordinates. Thus, they must also determine the weights $a_i$ appearing in $G_{id,w_0}(\mathbf{a})$. Accordingly, every weight $a_i$ appears in the extremal path collection with sink set $S_{i'}$ for some value of $i'$. We conclude that all the weights are determined as Laurent monomials in the extremal non-zero \Pl coordinates by the induction in the previous paragraph.
\end{proof}

\begin{propbis}{generalinverseLaurent}\label{generaltropinverse}
 The map $\Trop\Psi_{v,w}$ is inverse to $\Trop\Phi_{v,w}$ and does not involve minimization.
\end{propbis}

\begin{proof}
By \Cref{generalinverseLaurent}, the map $\Psi_{v,w}$ expresses the weights $a_i$ as Laurent monomials in the extremal non-zero \Pl coordinates of $F$. By \Cref{extremunique}, the map $\Phi_{v,w}$ expresses the extremal \Pl coordinates as monomials in $\bm{a}$. We know that $\Phi_{v,w}$ and $\Psi_{v,w}$ are inverses. Since tropicalizing converts products and quotients to sums and differences, respectively, we will have that $\Trop \Psi_{v,w}$ can be written as sums and differences of extremal \Pl coordinates and, moreover, that $\Trop \Psi_{v,w}$ and $\Trop \Phi_{v,w}$ are inverses.
\end{proof}

We provide an example of \Cref{prop:oneextraweightvw}. 
\begin{example}\label{monomialsrunning}

We recall the graph $G_{v,w}=G_{1324,4213}$ from \Cref{toymodel2} and \Cref{toymodelweighted} in \Cref{graphrunning}. We also recall the extremal indices in the cell $\mathcal{R}_{1324,4213}^{\geq 0}$ are $\{1\}$, $\{3\}$, $\{1,3\}$, $\{2,3\}$, $\{1,2,3\}$, $\{1,3,4\}$ and $\{2,3,4\}$.

     \begin{figure}[H]
 \begin{tikzpicture}[node distance={10.5 mm}, thick, main/.style = {draw, circle,minimum size=2 mm}, 
blank/.style={circle, draw=green!0, fill=green!0, very thin, minimum size=3.5mm},]

\node[main] (1) {$1'$};
\node[main] (2) [above of=1] {$3'$};
\node[main] (3) [above of = 2] {$2'$}; 
\node[main] (4) [above of=3] {$4'$};
\node (blank)[above of = 4]{};
\node[main] (44) [right  of=blank] {$4$};
\node[main] (33) [right of = 44] {$3$};
\node[main] (22) [right of = 33] {$2$};
\node[main] (11) [right of = 22] {$1$};
\node (bot1)[right of =2]{};
\node (top1)[right of = 4]{};
\node (bot2)[right of = 4]{};
\node (bot5) [right of = 3]{};
\node (bot4)[right of =bot5]{};
\node (bot6)[right of = bot4]{};
\node (bot7)[above of = bot6]{};
\node (top4)[above of = bot4]{};
\node (top6)[right of = top4]{};
\node (top7)[above of = top6]{};
\node (colend) [below of = 33]{};
\draw[-{Latex[length=3mm]}] ([xshift=-5mm,yshift=-5mm]bot5.center) to [out=-285,in=-72] ([xshift=-5mm,yshift=-5mm]44.center);
\path (top4) -- (44) node [near start, left] {$a_2\;\;\;\;\;$};

\draw[-{Latex[length=3mm]}] (1) -- (11);
\draw[-{Latex[length=3mm]}] (2) -- (22);
\draw[-{Latex[length=3mm]}] (3) -- (33);
\draw[-{Latex[length=3mm]}] (4) -- (44);
\draw[-{Latex[length=3mm]}] ([xshift=-5mm,yshift=-5mm]bot1.center) to [out=-260,in=-102] ([xshift=-5mm,yshift=-5mm]top1.center);
\path ([xshift=-5mm,yshift=-5mm]bot1.center)--([xshift=-5mm,yshift=-5mm]top1.center) node [pos=0.35, right] {$a_1$};

\draw[-{Latex[length=3mm]}] ([xshift=-5mm,yshift=-5mm]bot7.center) -- ([xshift=-5mm,yshift=-5mm]top7.center) node [pos=0.6, right] {$a_4$};
\draw[line width=0.75mm, color=dark-gray] ([xshift=-5mm,yshift=-5mm]bot5.center) -- ([xshift=-5mm,yshift=-5mm]colend.center);
\end{tikzpicture} 
\caption{The graph $G_{1324,4213}$.}
 \label{graphrunning}
 \end{figure}

We start with the Gale minimal extremal index of size $3$, which is $\{1,2,3\}$. The unique non-intersecting path collection with source set $[3]'$ and sink set $\{1,2,3\}$ has weight $1$. The Gale-next extremal index of size $3$ is $\{1,3,4\}$. The unique non-intersecting path collection with source set $[3]'$ and sink set $\{1,3,4\}$ has weight $a_2$, so $a_2=P_{134}$. Next, we look at $\{2,3,4\}$. The unique non-intersecting path collection with source set $[3]'$ and sink set $\{2,3,4\}$ has weight $a_2a_4$, so $a_4=\frac{P_{234}}{P_{134}}$. Now we move on to the Gale minimal extremal index of size $2$, which is $\{1,3\}$. The corresponding non-intersecting path collection has weight $1$. The Gale-next extremal index of size $2$ is $\{2,3\}$. The unique non-intersecting path collection with source set $[2]'$ and sink set $\{2,3\}$ has weight $a_4$. We have already determined $a_4$ as a Laurent monomial in the extremal non-zero \Pl coordinates, so we have nothing further to do here. However, we know that the relation $P_{23}=a_4=\frac{P_{234}}{P_{134}}$ follows using only three-term incidence \Pl relations. Indeed, it follows from $P_{23}P_{134}=P_{13}P_{234}+P_{34}P_{123}$, with $P_{13}=1$ and $P_{34}=0$. Finally, we move to the Gale minimal extremal index of size $1$, which is just $\{1\}$. The corresponding non-intersecting path collection has weight $1$. The Gale-next extremal index of size $1$ is $\{3\}$. The unique non-intersecting path collection with source set $[1]'$ and sink set $\{3\}$ has weight $a_1$, so $a_1=P_3$. Thus, we have determined all the weights as Laurent monomials in the extremal non-zero \Pl coordinates. By definition, $S_{1324,4213}=\{\{1,3,4\},\{2,3,4\},\{3\}\}$ and we see that the weights are in fact Laurent monomials in the extremal \Pl coordinates with indices in $S_{1324,4213}$. 
\end{example}

We are now ready to prove our main results. The next theorem shows that the totally nonnegative complete flag variety exactly coincides with the part of the complete flag variety where all \Pl coordinates are nonnegative. The tropicalization of this result will say that the totally nonnegative tropical complete flag variety is identical to the totally nonnegative complete flag Dressian.

\begin{theorem}\label{nonnegPluckCoords}
The totally nonnegative complete flag variety $\textnormal{Fl}_n^{\geq 0}$ equals the set $\{F\in \textnormal{Fl}_n\mid\;P_I(F)\geq 0\; \forall \;\emptyset\neq I\subset [n]\}$.
\end{theorem}

\begin{proof}
We already established in \Cref{forwarddir} that for any $F$ in $\textnormal{Fl}_n^{\geq 0}$, we have $P_I(F)\geq 0$ for any $\emptyset\neq I\subset[n]$. We are left to prove the reverse direction. 

Let $F$ be any flag in $\textnormal{Fl}_n$ such that $P_I(F)\geq 0$ for any $\emptyset\neq I\subset [n]$. By \Cref{synthetic}, the flag matroid corresponding to the non-zero \Pl coordinates of $F$ is $\bm{\mathcal{M}^{v,w}}$ for some $v\leq w\in \mathfrak{S}_n$. We will show that $F$ has a representative in the corresponding $\mathcal{R}_{v,w}^{>0}$, completing the proof of the theorem.

Specifically, we construct weights $\bm{a}$ on $G_{v,w}$ such that $\Phi_{v,w}(\bm{a})=P(F)$. Using \Cref{generalinverseLaurent}, apply the Laurent monomial map $\Psi_{v,w}$ to the subset of the extremal non-zero \Pl coordinates of $F$ whose indices lie in $S_{v,w}$. This yields a collection of positive weights $\bm{a}$. Let $Q=\Phi_{v,w}(\bm{a})$. This means $Q=P(F')$ for some $F'\in \mathcal{R}_{v,w}^{>0}$. By construction, $P(F)$ and $Q$ have the same support. Since $\Phi_{v,w}$ and $\Psi_{v,w}$ are inverse maps, $Q$ and $P(F)$ also agree on coordinates with indices in $S_{v,w}$. By \Cref{Svwgenerate} and \Cref{3termgenerate}, the three-term incidence \Pl relations determine all the other \Pl coordinates in terms of the coordinates with indices in $S_{v,w}$. Thus, $P(F)$ and $Q$ agree on all \Pl coordinates and so, $F$ lies in $\mathcal{R}_{v,w}^{>0}$.
\end{proof}

\begin{thmbis}{nonnegPluckCoords}\label{tropnonnegPluckcoords}
The totally nonnegative tropical complete flag variety $\textnormal{TrFl}_n^{\geq 0}$ equals the totally nonnegative complete flag Dressian $\textnormal{FlDr}_n^{\geq 0}$.
\end{thmbis}

\begin{proof}
It is clear by definition that $\textnormal{TrFl}_n^{\geq 0}\subseteq \textnormal{FlDr}_n^{\geq 0}$. We are left to prove the reverse inclusion. 

Let $P\in \textnormal{FlDr}_n^{\geq 0}$. By \Cref{tropsynthetic}, $P\in \textnormal{FlDr}^{> 0}_{v,w}$ for some $v\leq w\in \mathfrak{S}_n$. Using \Cref{generaltropinverse}, Apply the map $\Trop\Psi_{v,w}$ to the subset of the extremal non-infinite \Pl coordinates of $P$ whose indices lie in $S_{v,w}$. This yields a collection of real weights $\bm{a}$. Consider $Q=\Trop\Phi_{v,w}(\bm{a})\in \textnormal{TrFl}_n^{\geq 0}$. By construction, $P$ and $Q$ have the same support. Since $\Trop \Phi_{v,w}$ and $\Trop\Psi_{v,w}$ are inverse maps, $P$ and $Q$ agree on those \Pl coordinates with indices in $S_{v,w}$. By \Cref{tropicalSvwgenerate} and \Cref{tropical3termgenerate}, the positive tropicalizations of the three-term incidence \Pl relations determine all the other \Pl coordinates in terms of the coordinates with indices in $S_{v,w}$. Since both $\textnormal{FlDr}_n^{\geq 0}$ and $\textnormal{TrFl}_n^{\geq 0}$ satisfy the positive tropicalizations of the three-term incidence \Pl relations, $P$ and $Q$ agree on all \Pl coordinates. Thus, $P$ lies in $\textnormal{TrFl}_n^{\geq 0}$.
\end{proof}

\section{Proof of \Cref{extremunique}}\label{sec:BigProof}

In this section, we present the proof of \Cref{extremunique}. The proof is by a careful induction. The base cases are involved in their own rights and are presented as their own results in this section. 

Throughout this section, fix an integer $k\in[n]$ and let $O_{v,w}$ be the set of strands in $G_{v,w}$ whose source labels are in $[k]'$. Note that, by \Cref{vertexinversion}, $O_{v,w}$ consists of strands $\{v^{-1}(i)\mid i\in[k]\}$. For $q\in[n]$, we will also denote by $O_{v,w}(q)$ the subset of $O_{v,w}$ of strands weakly above strand $q$. If $C$ is a path collection originating from $O_{v,w}$, we will denote by $C(q)$ the subcollection of $C$ originating from $O_{v,w}(q)$.

Our proof will induct on $\ell(v)$ and on the size of our path collection. We first establish the base case with respect to our induction on $\ell(v)$ by considering graphs $G_{id,w}$. 

\begin{definition}
    Let $v\leq w\in \mathfrak{S}_n$. We will call a non-intersecting path collection $C$ from strands $O$ to sink set $S$ in $G_{v,w}$ \textbf{unique} if it is the unique non-intersecting path collection from $O$ to $S$ in $G_{v,w}$.
\end{definition}

\begin{lemma} \label{idextremunique} 

  Let $w\in \mathfrak{S}_n$, and let $F\in \mathcal{R}_{id,w}^{> 0}$. Suppose $I$ is an extremal index of $F$. There is a unique extremal path collection with sink set $I$ in the graph $G_{id,w}$. It is the union of a diagonal path collection and a greedy path collection.

  Conversely, the sink set of any non-intersecting path collection in $G_{id,w}$ that is the union of a greedy path collection originating from the top $q$ vertices of $[|I|]'$ and a diagonal path collection originating from the remaining vertices of $[|I|]'$ for some $0\leq q\leq |I|$ is an extremal index.

 \end{lemma}

 \begin{proof}

Let $G=G_{id,w}$. A graph extremal index in $G$ is the sink set of a non-intersecting path collection consisting of a left extreme path collection originating from strands $O_{id,w}(q)$, union the unique diagonal path collection originating from strands $O_{id,w}\setminus O_{id,w}(q)$, for some $0\leq q\leq k+1$ (where $O_{id,w}(k+1)=\emptyset$). Let $C$ be the greedy path collection originating from $O_{id,w}(q)$. There are no vertical edges skipping strands in $G$, because strands are never permuted when constructing $G$. Thus, greedy paths are left extreme. We are just left to show that $C$ is unique. 

We will first prove an auxiliary fact: Suppose there is a path $p_{\kappa}$ which originates at strand $1\leq\kappa\leq k$, does not intersect $p_{\kappa'}^C$ for any $\kappa<\kappa'\leq k$ and which terminates at the same sink as $p_{\kappa}^C$. Then, $p_\kappa=p_\kappa^C$. We observe that since $p_\kappa^C$ is greedy, it will never lie below $p_\kappa$ and as a result, whenever $p_\kappa$ meets $p_\kappa^C$, it must be at a point where $p_\kappa^C$ is traveling along a diagonal strand and $p_\kappa$ is traveling vertically. This can only happen if $p_\kappa^C$ and $p_\kappa$ both reach a point $(\sigma,c)_{G}$, but $p_\kappa^C$ uses a diagonal strand from $(\sigma,c-1)_G$ to $(\sigma,c)_G$, whereas $p_\kappa$ uses a vertical edge from $(\sigma-1,c)_G$ to $(\sigma,c)_G$. 

Let $\iota$ be the largest number such that there are paths $q_1,q_2,\ldots, q_\iota$ of $C$, necessarily distinct, passing through $(\sigma+1,c)_G,(\sigma+2,c-1)_G,\ldots, (\sigma+\iota,c-\iota+1)_G$. Let $q_0=p^C_\kappa$. We argue by induction that for $0\leq i\leq \iota$, $q_i$ also passes through $(\sigma+i,c-i)_G$. This is true by construction for $q_0=p^C_\kappa$. For $0<i\leq \iota$, assume the claim holds for $q_{i-1}$. Since $q_i$ passes through $(\sigma+i,c-i+1)_G$ and, by induction, $q_{i-1}$ passes through $(\sigma+i-1,c-i+1)_G$, $q_i$ cannot reach $(\sigma+i,c-i+1)_G$ using a vertical edge. It must instead take a diagonal strand from $(\sigma+i,c-i)_G$, proving the claim.

By repeated application of \Cref{posdist}, since we have a vertical edge from $(\sigma,c)_{G}$ to $(\sigma+1,c)_{G}$, we also have a vertical edge $e_\iota$ from $(\sigma+\iota,c-\iota)_{G}$ to $(\sigma+\iota+1,c-\iota)_{G}$. By definition of $\iota$, there is no path of $C$ passing through $(\sigma+\iota+1,c-\iota)_{G}$. The fact that $q_\iota$ did not use $e_\iota$ then contradicts the greediness of $C$. The $\iota=1$ case is illustrated in \Cref{fig:casesforid}.

We now return to proving that $C$ is unique. Suppose that $C'$ has the same source and sink set as $C$. Note that, in any path collection in $G_{id,w}$, the path with the $i^\textnormal{th}$ highest source terminates at the sink the $i^\textnormal{th}$ furthest to the left since no vertical edges skip strands. Thus, for any strand $\kappa$ in $O_{id,w}$, we have that $p_\kappa^{C'}$ terminates at the same sink as $p_\kappa^C$. Thus, applying the claim inductively to $p_k^C, p_{k-1}^C,\ldots, p_1^C$ proves that each path of $C'$ is identical to the path with the same source vertex in $C$. 

The converse of the Lemma holds since all greedy path collections in $G_{id,w}$ are left extreme.  
 
\end{proof}

 \begin{figure}[H]
  
 \scalebox{0.8}{\begin{tikzpicture}[node distance={10.5 mm}, thick, main/.style = {draw, circle,minimum size=2 mm}, 
blank/.style={circle, draw=green!0, fill=green!0, very thin, minimum size=3.5mm},]

\node[main] (1) {$1'$};
\node[main] (2) [above of=1] {$2'$};
\node[main] (3) [above of = 2] {$3'$}; 
\node[main] (4) [above of=3] {$4'$};
\node[main](5) [above of = 4] {$5'$};
\node (blank)[above of = 5]{};
\node[main](55) [right of = blank]{$5$};
\node(54) [below of = 55]{};
\node(53) [below of = 54]{};
\node(52) [below of = 53]{};
\node(51) [below of = 52]{};
\node[main] (44) [right  of=55] {$4$};
\node(43) [below of = 44]{};
\node(42) [below of = 43]{};
\node(41) [below of = 42]{};
\node[main] (33) [right of = 44] {$3$};
\node(32) [below of = 33]{};
\node(31) [below of = 32]{};
\node[main] (22) [right of = 33] {$2$};
\node(21) [below of = 22]{};
\node[main] (11) [right of = 22] {$1$};

\draw[-{Latex[length=3mm]}] (5) -- (55);
\draw[-{Latex[length=3mm]}] (4) -- (44);
\draw[] (3) -- ([xshift=-5mm,yshift=-5mm]53.center);
\draw[dashed] ([xshift=-5mm,yshift=-5mm]53.center) -- ([xshift=-5mm,yshift=-5mm]43.center);
\draw[-{Latex[length=3mm]}] ([xshift=-5mm,yshift=-5mm]43.center)--(33);
\draw[dashed] (2) -- ([xshift=-5mm,yshift=-5mm]52.center);
\draw[] ([xshift=-5mm,yshift=-5mm]52.center) -- ([xshift=-5mm,yshift=-5mm]42.center);
\draw[dotted,blue] ([xshift=-5mm,yshift=-5mm]42.center) -- ([xshift=-5mm,yshift=-5mm]32.center);
\draw[-{Latex[length=3mm]}] ([xshift=-5mm,yshift=-5mm]32.center) -- (22);
\draw[dotted] (1) -- ([xshift=-5mm,yshift=-5mm]41.center);
\draw[dotted,red] ([xshift=-5mm,yshift=-5mm]41.center) -- ([xshift=-5mm,yshift=-5mm]31.center);
\draw[-{Latex[length=3mm]}] ([xshift=-5mm,yshift=-5mm]31.center)--(11);

\draw[-{Latex[length=3mm]}] ([xshift=-5mm,yshift=-5mm]54.center) -- ([xshift=-5mm,yshift=-5mm]55.center);
\draw[-{Latex[length=3mm]},opacity=0.35] ([xshift=-5mm,yshift=-5mm]53.center) -- ([xshift=-5mm,yshift=-5mm]54.center);
\draw[-{Latex[length=3mm]},dashed] ([xshift=-5mm,yshift=-5mm]52.center) -- ([xshift=-5mm,yshift=-5mm]53.center);
\draw[-{Latex[length=3mm]},dashed] ([xshift=-5mm,yshift=-5mm]43.center) -- ([xshift=-5mm,yshift=-5mm]44.center);
\draw[-{Latex[length=3mm]}] ([xshift=-5mm,yshift=-5mm]42.center) -- ([xshift=-5mm,yshift=-5mm]43.center);
\draw[-{Latex[length=3mm]},dotted,blue] ([xshift=-5mm,yshift=-5mm]41.center) -- ([xshift=-5mm,yshift=-5mm]42.center);
\draw[-{Latex[length=3mm]},dotted,color=red] ([xshift=-5mm,yshift=-5mm]31.center) -- ([xshift=-5mm,yshift=-5mm]32.center);
\draw[-{Latex[length=3mm]},dotted] ([xshift=-5mm,yshift=-5mm]32.center) -- ([xshift=-5mm,yshift=-5mm]33.center);

\end{tikzpicture} }
 \caption{Suppose the grey edge in column 1 were not part of this graph. Then the dashed path combined with the dotted path using either the red or the blue dotted edges, would give two different left extreme path collections originating at $\{1',2'\}$. However, the existence of the vertical red edge in column $3$ means, by \Cref{posdist}, we must have the grey edge in the graph. If we include the grey edge, one can verify that we do indeed have a unique left extreme path collection (with sink set $\{4,5\}$), and that path collection is greedy.}
\label{fig:casesforid}

\end{figure}

The next statement is a consequence of the previous Lemma, by considering each of the extremal path collections consisting of a single greedy path union diagonal paths in $G_{id,w}$.

\begin{cor}\label{coruniquepath}
 For any $w\in \mathfrak{S}_n$, there is a unique left extreme path originating from any single source vertex in $G_{id,w}$.

\end{cor}

Having established the base case for $\ell(v)$, we now establish the notation and machinery we will use for the induction from permutations of length $d-1$ to those of length $d$. Let $\bm{w}=s_{i_1}\cdots s_{i_\ell}$ be the positive distinguished subexpression for $w$ in $\bm{w_0}$. Suppose that $v$ has positive distinguished subexpression $\bm{v}=s_{i_{j_1}}\cdots s_{i_{j_m}}$ in $\bm{w}$. Let $\bm{v}^{(d)}=s_{i_{j_1}}\cdots s_{i_{j_d}}$ for $1\leq d\leq m$ be an expression for a permutation of length $d$ which we denote $v^{(d)}$. Note that each $\bm{v}^{(d)}$ is a positive distinguished subexpression in $\bm{w}$. 

\begin{remark} \label{rem:differentnotations} The permutation $v^{(d)}$ differs from $v_{(d)}$, introduced in \Cref{rmk:wkandvk}, where we truncated $v$ to the first $d$ simple reflections of $w$. The latter notation will not appear in the rest of this paper, which we hope will avoid confusion. 
\end{remark}

We fix $1\leq d\leq m$ and also fix $G_1=G_{{v^{(d)}},w}$ and $G_2=G_{v^{(d+1)},w}$. We will argue that if \Cref{extremunique} holds for $G_1$, it also holds for $G_2$. We define $O_1=O_{v^{(d)},w}$ and $O_2=O_{v^{(d+1)},w}$ to be the set of origin strands for extremal path collections in $G_1$ and $G_2$, respectively.  We also define $\mathbf{w}^{(d+1)}$ to be the truncation of $\mathbf{w}$ such that the last simple reflection of $\mathbf{w}^{(d+1)}$ is the final simple reflection, $s_{i_{j_{d+1}}}$ of $\mathbf{v}^{(d+1)}=s_{i_{j_1}}\cdots s_{i_{j_{d+1}}}$. 

Suppose $s_{i_{j_{d+1}}}=s_r$, so that $v^{(d+1)}=v^{(d)}s_r$. Then, \Cref{matrixdecompose} and \Cref{constructG} imply that $G_2$ differs from $G_1$ according to the following three operations, which we name for future reference:

 \begin{itemize}
 \item \typeop{(\textbf{Rem})}{removal} A single vertical edge is removed from $G_1$ between strands $r$ and $r+1$ in some column, say in column $c$.
 \item \typeop{(\textbf{Int})}{interchange} The source vertex labels on strands $r$ and $r+1$ are interchanged, with the larger label moving from strand $r+1$ to strand $r$. 
 \item \typeop{(\textbf{Swap})}{incidences} Any vertical edge of $G_1$ strictly left of column $c$, or in column $c$ and originating below strand $r+1$, which has an endpoint on either strand $r$ or strand $r+1$ has that endpoint swapped to strand $r+1$ or strand $r$, respectively. Note that it is impossible for such an edge to have endpoints on both strands $r$ and $r+1$ by \Cref{vertarrow}. 
  \end{itemize}

As a result of \Cref{incidences}, there are a number of different paths in $G_2$ and in $G_1$ which can be naturally associated with one another. Explicitly, we will define a map from paths of $G_2$ to paths of $G_1$ and say that paths $p_1$ and $p_2$ in $G_1$ and $G_2$, respectively, are associated if this map takes $p_2$ to $p_1$. These associations will help us to get a concrete grasp on how the extremal path collections in $G_2$ compare to those in $G_1$. Note that the map goes ``backwards"; $G_2$ is obtained from $G_1$ by applying \Cref{removal}, \Cref{interchange}, and \Cref{incidences}, but we map each path in $G_2$ to an associated path in $G_1$. We now describe the map explicitly on a case-by-case basis. By construction, each vertical edge of $G_2$ corresponds to a vertical edge of $G_1$; some are identical and some have their endpoints modified by \Cref{incidences}. In what follows, we will abuse notation and identify an edge in $G_2$ with its counterpart in $G_1$.

\begin{definition}\label{def:swapequivalencemap}
    Define the map $\Lambda$ from paths in $G_2$ to paths in $G_1$ as follows: 
\begin{enumerate}
    \item \label{case1} Let $p$ be a path in $G_2$ originating on strand $r$.
      \begin{enumerate}
          \item \label{case0} If $p$ is diagonal, $\Lambda(p)$ is the diagonal path on strand $r+1$ in $G_1$.
          \item \label{case1a}If $p$ leaves strand $r$ weakly left of column $c$, $\Lambda(p)$ is the path in $G_1$ originating on strand $r+1$ which uses precisely the same edges as $p$.  
          \item  \label{case1b} If $p$ leaves strand $r$ right of column $c$, $\Lambda(p)$ is the path in $G_1$ originating on strand $r+1$ which uses all the edges of $p$ except the edge which $p$ uses to get from strand $r$ to strand $r+1$. 
      \end{enumerate}
      See \Cref{figupto} for examples.
    \item \label{case2} Let $p$ be a path in $G_2$ originating on strand $r+1$.
    
    \begin{enumerate}
        \item \label{case2a}If $p$ leaves strand $r+1$ strictly left of column $c$, $\Lambda(p)$ is the path in $G_1$ originating on strand $r$ which uses the same vertical edges. 
        
        \item \label{case2b} If $p$ leaves strand $r+1$ weakly right of column $c$ or is diagonal, $\Lambda(p)$ is the path in $G_1$ originating on strand $r$ which uses the vertical edge removed by \Cref{removal} to get from strand $r$ to strand $r+1$, and then afterwards uses the same vertical edges as $p$. 
    \end{enumerate}
    
    \item \label{case3} Let $p$ be a path in $G_2$ originating on strand $\sigma$ for $\sigma\neq r,r+1$.
    
    \begin{enumerate}
        \item \label{case3a}If $p$ does not intersect strands $r$ or $r+1$, $\Lambda(p)$ is the path in $G_1$ originating on strand $\sigma$ which uses the same vertical edges.
        \item \label{case3b}
        If $p$ both enters and exits strand $r$ strictly left of column $c$ or both enters and exits strand $r+1$ weakly left of column $c$, $\Lambda(p)$ is the path in $G_1$ originating on strand $\sigma$ which uses the same vertical edges as $p$. 
         \item \label{case3c} If $p$ enters strand $r+1$ weakly left of column $c$ and leaves strand $r+1$ weakly right of column $c$ or terminates on strand $r+1$, $\Lambda(p)$ is the path in $G_1$ which originates on strand $\sigma$, uses all the same edges as $p$ and additionally uses the vertical edge removed by \Cref{removal} to get from strand $r$ to strand $r+1$. See \Cref{figupto2} for an example.
         
        \item \label{case3d}If $p$ enters strand $r$ weakly left of column $c$, uses a vertical edge to reach strand $r+1$ and then either exits strand $r+1$ weakly right of column $c$ or terminates on strand $r+1$, $\Lambda(p)$ is the path in $G_1$ originating on strand $\sigma$ which uses the same vertical edges except for the edge between strands $r$ and $r+1$. See \Cref{figupto3} for an example. 
        
        \item \label{case3d2}If $p$ enters strand $r$ strictly left of column $c$, and terminates on strand $r$, $\Lambda(p)$ is the path in $G_1$ originating on strand $\sigma$ which uses the same vertical edges and terminates on strand $r+1$. 
        
        \item \label{case3e}If $p$ enters strand $r$ strictly right of column $c$, $\Lambda(p)$ is the path in $G_1$ originating on strand $\sigma$ which uses the same vertical edges.

    \end{enumerate}

\end{enumerate}
\end{definition}
Observe that we need not account for when $p$ enters strand $r$ in column $c$. If this were possible, there would have to be two vertical edges terminating at $(r+1,c)_{G_1}$; The one removed by \Cref{removal} and the one modified by \Cref{incidences} to allow $p$ to enter strand $r$ in column $c$. The construction in \Cref{constructG} does not allow this to happen. Also, since there are no vertical edges skipping strands to the right of column $c$, any path entering strand $r+1$ strictly right of column $c$ must pass through strand $r$ and so is covered by cases \ref{case3d} and \ref{case3e}. Accordingly, $\Lambda$ is defined on all paths in $G_2$.

\begin{definition}\label{upto}
Let $p$ be a path in $G_2$ and let $p'=\Lambda(p)$ be a path in $G_1$. We will say that $p$ and $p'$ are \textbf{\Cref{incidences} equivalent}. Some instances of this relationship are illustrated in \Cref{figupto}, \Cref{figupto2}, and \Cref{figupto3}.
\end{definition}

\begin{lemma}\label{lem:laminverse}
Let $p'$ be a path in $G_1$ and assume that $p'$ does not leave strand $r+1$ to the right of column $c$. Then, $|\Lambda^{-1}(p')|\leq 1$. In words, $p'$ is \Cref{incidences} equivalent to at most one path.     
\end{lemma}

\begin{proof}
    This can be checked directly. Specifically, one can observe that if $p_1\neq p_2$ and $\Lambda(p_1)=\Lambda(p_2)$, then $p_1$ and $p_2$ fall under cases \ref{case2b}, \ref{case3d}, or \ref{case3d2} of \Cref{def:swapequivalencemap}.
\end{proof}

\begin{figure}[H]
 \begin{subfigure}{0.49\textwidth}
  \scalebox{0.79}{
   \begin{tikzpicture}[node distance={10.5 mm}, thick, main/.style = {draw, circle, minimum size=0.8cm}, 
blank/.style={circle, draw=green!0, fill=green!0, very thin, minimum size=3.5mm}]

\node[main] (1) {$i_1'$};
\node[main] (2) [above of=1] {$i_2'$};
\node[main] (3) [above of = 2] {$i_3'$}; 
\node[main] (4) [above of=3] {$i_4'$};
\node[main](5) [above of = 4] {$i_5'$};
\node[main](6) [above of = 5]{$i_6'$};
\node (blank)[above of = 6]{};
\node[main](66) [right of = blank]{\tiny $r+5$};
\node(65) [below of = 66]{};
\node(64) [below of = 65]{};
\node(63) [below of = 64]{};
\node(62) [below of = 63]{};
\node(61) [below of = 62]{};
\node[main] (55) [right of = 66]{\tiny $r+4$};
\node(54) [below of = 55]{};
\node(53) [below of = 54]{};
\node(52) [below of = 53]{};
\node(51) [below of = 52]{};
\node[main] (44) [right of = 55] {\tiny $r+3$};
\node(43) [below of = 44]{};
\node(42) [below of = 43]{};
\node(41) [below of = 42]{};
\node[main] (33) [right of = 44] {\tiny $r+2$};
\node(32) [below of = 33]{};
\node(31) [below of = 32]{};
\node[main] (22) [right of = 33] {\tiny $r+1$};
\node(21) [below of = 22]{};
\node[main] (11) [right of = 22] {$r$};

\draw[] (6) -- ([xshift=-5mm,yshift=-5mm]66.center);
\draw[-{Latex[length=3mm]},red] ([xshift=-5mm,yshift=-5mm]66.center) -- (66);
\draw[] (5) -- ([xshift=-5mm,yshift=-5mm]55.center);
\draw[-{Latex[length=3mm]}] ([xshift=-5mm,yshift=-5mm]55.center) -- (55);
\draw[-{Latex[length=3mm]}] (4) -- (44);
\draw[] (3) -- ([xshift=-5mm,yshift=-5mm]53.center);
\draw[] ([xshift=-5mm,yshift=-5mm]53.center) -- ([xshift=-5mm,yshift=-5mm]33.center);
\draw[-{Latex[length=3mm]},dotted] ([xshift=-5mm,yshift=-5mm]33.center)--(33);
\draw[red,dotted] (2) -- ([xshift=-5mm,yshift=-5mm]62.center);
\draw[dotted] ([xshift=-5mm,yshift=-5mm]62.center) -- ([xshift=-5mm,yshift=-5mm]42.center);
\draw[dotted] ([xshift=-5mm,yshift=-5mm]42.center) -- ([xshift=-5mm,yshift=-5mm]32.center);
\draw[-{Latex[length=3mm]}] ([xshift=-5mm,yshift=-5mm]32.center) -- (22);
\draw[] (1) -- ([xshift=-5mm,yshift=-5mm]61.center);
\draw[] ([xshift=-5mm,yshift=-5mm]61.center) -- ([xshift=-5mm,yshift=-5mm]41.center);
\draw[] ([xshift=-5mm,yshift=-5mm]41.center) -- ([xshift=-5mm,yshift=-5mm]31.center);
\draw[-{Latex[length=3mm]}] ([xshift=-5mm,yshift=-5mm]31.center)--(11);

\draw[-{Latex[length=3mm]},red] ([xshift=-5mm,yshift=-5mm]62.center) to ([xshift=-5mm,yshift=-5mm]63.center);
\draw[-{Latex[length=3mm]},red] ([xshift=-5mm,yshift=-5mm]63.center) -- ([xshift=-5mm,yshift=-5mm]64.center);
\draw[-{Latex[length=3mm]},red] ([xshift=-5mm,yshift=-5mm]64.center) -- ([xshift=-5mm,yshift=-5mm]65.center);
\draw[-{Latex[length=3mm]},red]
([xshift=-5mm,yshift=-5mm]65.center) -- ([xshift=-5mm,yshift=-5mm]66.center);
\draw[-{Latex[length=3mm]}] ([xshift=-5mm,yshift=-5mm]51.center) --
([xshift=-5mm,yshift=-5mm]52.center)node [pos=0.55, right] {$e$};
\draw[-{Latex[length=3mm]}] ([xshift=-5mm,yshift=-5mm]41.center) -- ([xshift=-5mm,yshift=-5mm]42.center);
\draw[-{Latex[length=3mm]},dotted] ([xshift=-5mm,yshift=-5mm]32.center) -- ([xshift=-5mm,yshift=-5mm]33.center);
\draw[-{Latex[length=3mm]}] ([xshift=-5mm,yshift=-5mm]31.center) -- ([xshift=-5mm,yshift=-5mm]32.center);

\end{tikzpicture}} 
\label{upto2}
\caption{}
\end{subfigure}
\begin{subfigure}{0.49\textwidth}
  \scalebox{0.79}{
  \begin{tikzpicture}[node distance={10.5 mm}, thick, main/.style = {draw, circle,minimum size= 0.8cm}, 
blank/.style={circle, draw=green!0, fill=green!0, very thin, minimum size=3.5mm}]

\node[main] (1) {$i_2'$};
\node[main] (2) [above of=1] {$i_1'$};
\node[main] (3) [above of = 2] {$i_3'$}; 
\node[main] (4) [above of=3] {$i_4'$};
\node[main](5) [above of = 4] {$i_5'$};
\node[main](6) [above of = 5]{$i_6'$};
\node (blank)[above of = 6]{};
\node[main](66) [right of = blank]{\tiny $r+5$};
\node(65) [below of = 66]{};
\node(64) [below of = 65]{};
\node(63) [below of = 64]{};
\node(62) [below of = 63]{};
\node(61) [below of = 62]{};
\node[main] (55) [right of = 66]{\tiny $r+4$};
\node(54) [below of = 55]{};
\node(53) [below of = 54]{};
\node(52) [below of = 53]{};
\node(51) [below of = 52]{};
\node[main] (44) [right of = 55] {\tiny $r+3$};
\node(43) [below of = 44]{};
\node(42) [below of = 43]{};
\node(41) [below of = 42]{};
\node[main] (33) [right of = 44] {\tiny $r+2$};
\node(32) [below of = 33]{};
\node(31) [below of = 32]{};
\node[main] (22) [right of = 33] {\tiny $r+1$};
\node(21) [below of = 22]{};
\node[main] (11) [right of = 22] {$r$};

\draw[] (6) -- ([xshift=-5mm,yshift=-5mm]66.center);
\draw[-{Latex[length=3mm]},red] ([xshift=-5mm,yshift=-5mm]66.center) -- (66);
\draw[] (5) -- ([xshift=-5mm,yshift=-5mm]55.center);
\draw[-{Latex[length=3mm]}] ([xshift=-5mm,yshift=-5mm]55.center) -- (55);
\draw[-{Latex[length=3mm]}] (4) -- (44);
\draw[] (3) -- ([xshift=-5mm,yshift=-5mm]53.center);
\draw[] ([xshift=-5mm,yshift=-5mm]53.center) -- ([xshift=-5mm,yshift=-5mm]33.center);
\draw[-{Latex[length=3mm]},dotted] ([xshift=-5mm,yshift=-5mm]33.center)--(33);
\draw[] (2) -- ([xshift=-5mm,yshift=-5mm]52.center);
\draw[] ([xshift=-5mm,yshift=-5mm]52.center) -- ([xshift=-5mm,yshift=-5mm]42.center);
\draw[dotted] ([xshift=-5mm,yshift=-5mm]42.center) -- ([xshift=-5mm,yshift=-5mm]32.center);
\draw[-{Latex[length=3mm]}] ([xshift=-5mm,yshift=-5mm]32.center) -- (22);
\draw[red,dotted] (1) -- ([xshift=-5mm,yshift=-5mm]61.center);
\draw[dotted] ([xshift=-5mm,yshift=-5mm]61.center) -- ([xshift=-5mm,yshift=-5mm]41.center);
\draw[dotted] ([xshift=-5mm,yshift=-5mm]41.center) -- ([xshift=-5mm,yshift=-5mm]31.center);
\draw[-{Latex[length=3mm]}] ([xshift=-5mm,yshift=-5mm]31.center)--(11);

\draw[-{Latex[length=3mm]},red] ([xshift=-5mm,yshift=-5mm]61.center) to [out=-280,in=-78] ([xshift=-5mm,yshift=-5mm]63.center);
\draw[-{Latex[length=3mm]},red] ([xshift=-5mm,yshift=-5mm]63.center) -- ([xshift=-5mm,yshift=-5mm]64.center);
\draw[-{Latex[length=3mm]},red] ([xshift=-5mm,yshift=-5mm]64.center) -- ([xshift=-5mm,yshift=-5mm]65.center);
\draw[-{Latex[length=3mm]},red]
([xshift=-5mm,yshift=-5mm]65.center) -- ([xshift=-5mm,yshift=-5mm]66.center);
\draw[-{Latex[length=3mm]},dotted] ([xshift=-5mm,yshift=-5mm]41.center) -- ([xshift=-5mm,yshift=-5mm]42.center);
\draw[-{Latex[length=3mm]},dotted] ([xshift=-5mm,yshift=-5mm]32.center) -- ([xshift=-5mm,yshift=-5mm]33.center);
\draw[-{Latex[length=3mm]},dotted] ([xshift=-5mm,yshift=-5mm]31.center) -- ([xshift=-5mm,yshift=-5mm]32.center);

\end{tikzpicture}} 
\label{upto1}
\caption{}
\end{subfigure}
\caption{The left graph is $G_1$ and the right graph is $G_2$. Edge $e$ is the edge which is removed by \Cref{removal} and lies in column $c$. $\Lambda$ maps the red path on the right to the red path on the left (case \ref{case2a}) and each of the dotted paths on the right to the dotted path on the left (case \ref{case2b}). Accordingly, the red paths are \Cref{incidences} equivalent, and the dotted path on the left is \Cref{incidences} equivalent to each of the dotted paths on the right.}
\label{figupto}
 \end{figure}

\begin{figure}[H]

\begin{subfigure}{0.49\textwidth}
  \scalebox{0.79}{
  \begin{tikzpicture}[node distance={10.5 mm}, thick, main/.style = {draw, circle,minimum size= 0.8cm}, 
blank/.style={circle, draw=green!0, fill=green!0, very thin, minimum size=3.5mm}]

\node[main] (1) {$i_1'$};
\node[main] (2) [above of=1] {$i_2'$};
\node[main] (3) [above of = 2] {$i_3'$}; 
\node[main] (4) [above of=3] {$i_4'$};
\node[main](5) [above of = 4] {$i_5'$};
\node[main](6) [above of = 5]{$i_6'$};
\node (blank)[above of = 6]{};
\node[main](66) [right of = blank]{\tiny $r+4$};
\node(65) [below of = 66]{};
\node(64) [below of = 65]{};
\node(63) [below of = 64]{};
\node(62) [below of = 63]{};
\node(61) [below of = 62]{};
\node[main] (55) [right of = 66]{\tiny $r+3$};
\node(54) [below of = 55]{};
\node(53) [below of = 54]{};
\node(52) [below of = 53]{};
\node(51) [below of = 52]{};
\node[main] (44) [right of = 55] {\tiny $r+2$};
\node(43) [below of = 44]{};
\node(42) [below of = 43]{};
\node(41) [below of = 42]{};
\node[main] (33) [right of = 44] {\tiny $r+1$};
\node(32) [below of = 33]{};
\node(31) [below of = 32]{};
\node[main] (22) [right of = 33] { $r$};
\node(21) [below of = 22]{};
\node[main] (11) [right of = 22] {\tiny $r-1$};

\draw[] (6) -- ([xshift=-5mm,yshift=-5mm]66.center);
\draw[-{Latex[length=3mm]}] ([xshift=-5mm,yshift=-5mm]66.center) -- (66);
\draw[] (5) -- ([xshift=-5mm,yshift=-5mm]55.center);
\draw[-{Latex[length=3mm]}] ([xshift=-5mm,yshift=-5mm]55.center) -- (55);
\draw[] (4) -- ([xshift=-5mm,yshift=-5mm]44.center);
\draw[-{Latex[length=3mm]},dotted] ([xshift=-5mm,yshift=-5mm]44.center) -- (44);
\draw[] (3) -- ([xshift=-5mm,yshift=-5mm]53.center);
\draw[dotted] ([xshift=-5mm,yshift=-5mm]53.center) -- ([xshift=-5mm,yshift=-5mm]43.center);
\draw[-{Latex[length=3mm]}] ([xshift=-5mm,yshift=-5mm]43.center)--(33);
\draw[] (2) -- ([xshift=-5mm,yshift=-5mm]62.center);
\draw[dotted] ([xshift=-5mm,yshift=-5mm]62.center) -- ([xshift=-5mm,yshift=-5mm]52.center);
\draw[-{Latex[length=3mm]}] ([xshift=-5mm,yshift=-5mm]52.center) -- (22);
\draw[dotted] (1) -- ([xshift=-5mm,yshift=-5mm]61.center);
\draw[] ([xshift=-5mm,yshift=-5mm]61.center) -- ([xshift=-5mm,yshift=-5mm]41.center);
\draw[] ([xshift=-5mm,yshift=-5mm]41.center) -- ([xshift=-5mm,yshift=-5mm]31.center);
\draw[-{Latex[length=3mm]}] ([xshift=-5mm,yshift=-5mm]31.center)--(11);

\draw[-{Latex[length=3mm]},dotted] ([xshift=-5mm,yshift=-5mm]61.center) to ([xshift=-5mm,yshift=-5mm]62.center);
\draw[-{Latex[length=3mm]},dotted] ([xshift=-5mm,yshift=-5mm]52.center) --
([xshift=-5mm,yshift=-5mm]53.center)node [pos=0.55, right] {$e$};
\draw[-{Latex[length=3mm]},dotted] ([xshift=-5mm,yshift=-5mm]43.center) -- ([xshift=-5mm,yshift=-5mm]44.center);
\draw[-{Latex[length=3mm]}] ([xshift=-5mm,yshift=-5mm]42.center) -- ([xshift=-5mm,yshift=-5mm]43.center);
\draw[-{Latex[length=3mm]}] ([xshift=-5mm,yshift=-5mm]41.center) -- ([xshift=-5mm,yshift=-5mm]42.center);

\end{tikzpicture}}
\caption{}
\end{subfigure}
 \begin{subfigure}{0.49\textwidth}
  \scalebox{0.79}{
  \begin{tikzpicture}[node distance={10.5 mm}, thick, main/.style = {draw, circle,minimum size= 0.8cm}, 
blank/.style={circle, draw=green!0, fill=green!0, very thin, minimum size=3.5mm}]

\node[main] (1) {$i_1'$};
\node[main] (2) [above of=1] {$i_3'$};
\node[main] (3) [above of = 2] {$i_2'$}; 
\node[main] (4) [above of=3] {$i_4'$};
\node[main](5) [above of = 4] {$i_5'$};
\node[main](6) [above of = 5]{$i_6'$};
\node (blank)[above of = 6]{};
\node[main](66) [right of = blank]{\tiny $r+4$};
\node(65) [below of = 66]{};
\node(64) [below of = 65]{};
\node(63) [below of = 64]{};
\node(62) [below of = 63]{};
\node(61) [below of = 62]{};
\node[main] (55) [right of = 66]{\tiny $r+3$};
\node(54) [below of = 55]{};
\node(53) [below of = 54]{};
\node(52) [below of = 53]{};
\node(51) [below of = 52]{};
\node[main] (44) [right of = 55] {\tiny $r+2$};
\node(43) [below of = 44]{};
\node(42) [below of = 43]{};
\node(41) [below of = 42]{};
\node[main] (33) [right of = 44] {\tiny $r+1$};
\node(32) [below of = 33]{};
\node(31) [below of = 32]{};
\node[main] (22) [right of = 33] { $r$};
\node(21) [below of = 22]{};
\node[main] (11) [right of = 22] {\tiny $r-1$};

\draw[] (6) -- ([xshift=-5mm,yshift=-5mm]66.center);
\draw[-{Latex[length=3mm]}] ([xshift=-5mm,yshift=-5mm]66.center) -- (66);
\draw[] (5) -- ([xshift=-5mm,yshift=-5mm]55.center);
\draw[-{Latex[length=3mm]}] ([xshift=-5mm,yshift=-5mm]55.center) -- (55);
\draw[] (4) -- ([xshift=-5mm,yshift=-5mm]44.center);
\draw[-{Latex[length=3mm]},dotted] ([xshift=-5mm,yshift=-5mm]44.center) -- (44);
\draw[] (3) -- ([xshift=-5mm,yshift=-5mm]63.center);
\draw[dotted] ([xshift=-5mm,yshift=-5mm]63.center) -- ([xshift=-5mm,yshift=-5mm]43.center);
\draw[-{Latex[length=3mm]}] ([xshift=-5mm,yshift=-5mm]43.center)--(33);
\draw[] (2) -- ([xshift=-5mm,yshift=-5mm]62.center);
\draw[] ([xshift=-5mm,yshift=-5mm]62.center) -- ([xshift=-5mm,yshift=-5mm]32.center);
\draw[-{Latex[length=3mm]}] ([xshift=-5mm,yshift=-5mm]32.center) -- (22);
\draw[dotted] (1) -- ([xshift=-5mm,yshift=-5mm]61.center);
\draw[] ([xshift=-5mm,yshift=-5mm]61.center) -- ([xshift=-5mm,yshift=-5mm]41.center);
\draw[] ([xshift=-5mm,yshift=-5mm]41.center) -- ([xshift=-5mm,yshift=-5mm]31.center);
\draw[-{Latex[length=3mm]}] ([xshift=-5mm,yshift=-5mm]31.center)--(11);

\draw[-{Latex[length=3mm]},dotted] ([xshift=-5mm,yshift=-5mm]61.center) to [out=-280,in=-78] ([xshift=-5mm,yshift=-5mm]63.center);
\draw[-{Latex[length=3mm]},dotted] ([xshift=-5mm,yshift=-5mm]43.center) -- ([xshift=-5mm,yshift=-5mm]44.center);
\draw[-{Latex[length=3mm]}] ([xshift=-5mm,yshift=-5mm]42.center) -- ([xshift=-5mm,yshift=-5mm]43.center);
\draw[-{Latex[length=3mm]}] ([xshift=-5mm,yshift=-5mm]41.center) -- ([xshift=-5mm,yshift=-5mm]42.center);

\end{tikzpicture}} 
\caption{}
\end{subfigure}
\caption{The left graph is $G_1$ and the right graph is $G_2$. Edge $e$ is the edge which is removed by \Cref{removal} and lies in column $c$. $\Lambda$ maps the dotted path on the right to the dotted path on the left (case \ref{case3c}). Accordingly, the dotted paths are \Cref{incidences} equivalent. }

\label{figupto2}
 \end{figure}

\begin{figure}[H]

\begin{subfigure}{0.49\textwidth}
  \scalebox{0.79}{
  \begin{tikzpicture}[node distance={10.5 mm}, thick, main/.style = {draw, circle,minimum size= 0.8cm}, 
blank/.style={circle, draw=green!0, fill=green!0, very thin, minimum size=3.5mm}]

\node[main] (1) {$i_1'$};
\node[main] (2) [above of=1] {$i_2'$};
\node[main] (3) [above of = 2] {$i_3'$}; 
\node[main] (4) [above of=3] {$i_4'$};
\node[main](5) [above of = 4] {$i_5'$};
\node[main](6) [above of = 5]{$i_6'$};
\node (blank)[above of = 6]{};
\node[main](66) [right of = blank]{\tiny $r+4$};
\node(65) [below of = 66]{};
\node(64) [below of = 65]{};
\node(63) [below of = 64]{};
\node(62) [below of = 63]{};
\node(61) [below of = 62]{};
\node[main] (55) [right of = 66]{\tiny $r+3$};
\node(54) [below of = 55]{};
\node(53) [below of = 54]{};
\node(52) [below of = 53]{};
\node(51) [below of = 52]{};
\node[main] (44) [right of = 55] {\tiny $r+2$};
\node(43) [below of = 44]{};
\node(42) [below of = 43]{};
\node(41) [below of = 42]{};
\node[main] (33) [right of = 44] {\tiny $r+1$};
\node(32) [below of = 33]{};
\node(31) [below of = 32]{};
\node[main] (22) [right of = 33] { $r$};
\node(21) [below of = 22]{};
\node[main] (11) [right of = 22] {\tiny $r-1$};

\draw[] (6) -- ([xshift=-5mm,yshift=-5mm]66.center);
\draw[-{Latex[length=3mm]}] ([xshift=-5mm,yshift=-5mm]66.center) -- (66);
\draw[] (5) -- ([xshift=-5mm,yshift=-5mm]55.center);
\draw[-{Latex[length=3mm]}] ([xshift=-5mm,yshift=-5mm]55.center) -- (55);
\draw[-{Latex[length=3mm]}] (4) -- (44);
\draw[] (3) -- ([xshift=-5mm,yshift=-5mm]63.center);
\draw[dotted] ([xshift=-5mm,yshift=-5mm]63.center) -- ([xshift=-5mm,yshift=-5mm]43.center);
\draw[-{Latex[length=3mm]},dotted] ([xshift=-5mm,yshift=-5mm]43.center)--(33);
\draw[] (2) -- ([xshift=-5mm,yshift=-5mm]62.center);
\draw[] ([xshift=-5mm,yshift=-5mm]62.center) -- ([xshift=-5mm,yshift=-5mm]32.center);
\draw[-{Latex[length=3mm]}] ([xshift=-5mm,yshift=-5mm]32.center) -- (22);
\draw[dotted] (1) -- ([xshift=-5mm,yshift=-5mm]61.center);
\draw[] ([xshift=-5mm,yshift=-5mm]61.center) -- ([xshift=-5mm,yshift=-5mm]41.center);
\draw[] ([xshift=-5mm,yshift=-5mm]41.center) -- ([xshift=-5mm,yshift=-5mm]31.center);
\draw[-{Latex[length=3mm]}] ([xshift=-5mm,yshift=-5mm]31.center)--(11);

\draw[-{Latex[length=3mm]},dotted] ([xshift=-5mm,yshift=-5mm]61.center) to [out=-280,in=-78] ([xshift=-5mm,yshift=-5mm]63.center);
\draw[-{Latex[length=3mm]}] ([xshift=-5mm,yshift=-5mm]52.center) --
([xshift=-5mm,yshift=-5mm]53.center)node [pos=0.55, right] {$e$};
\draw[-{Latex[length=3mm]}] ([xshift=-5mm,yshift=-5mm]41.center) -- ([xshift=-5mm,yshift=-5mm]42.center);
\draw[-{Latex[length=3mm]}] ([xshift=-5mm,yshift=-5mm]42.center) -- ([xshift=-5mm,yshift=-5mm]43.center);
\draw[-{Latex[length=3mm]}] ([xshift=-5mm,yshift=-5mm]32.center) -- ([xshift=-5mm,yshift=-5mm]33.center);

\end{tikzpicture}}
\caption{}
\end{subfigure}
 \begin{subfigure}{0.49\textwidth}
  \scalebox{0.79}{
  \begin{tikzpicture}[node distance={10.5 mm}, thick, main/.style = {draw, circle,minimum size= 0.8cm}, 
blank/.style={circle, draw=green!0, fill=green!0, very thin, minimum size=3.5mm}]

\node[main] (1) {$i_1'$};
\node[main] (2) [above of=1] {$i_3'$};
\node[main] (3) [above of = 2] {$i_2'$}; 
\node[main] (4) [above of=3] {$i_4'$};
\node[main](5) [above of = 4] {$i_5'$};
\node[main](6) [above of = 5]{$i_6'$};
\node (blank)[above of = 6]{};
\node[main](66) [right of = blank]{\tiny $r+4$};
\node(65) [below of = 66]{};
\node(64) [below of = 65]{};
\node(63) [below of = 64]{};
\node(62) [below of = 63]{};
\node(61) [below of = 62]{};
\node[main] (55) [right of = 66]{\tiny $r+3$};
\node(54) [below of = 55]{};
\node(53) [below of = 54]{};
\node(52) [below of = 53]{};
\node(51) [below of = 52]{};
\node[main] (44) [right of = 55] {\tiny $r+2$};
\node(43) [below of = 44]{};
\node(42) [below of = 43]{};
\node(41) [below of = 42]{};
\node[main] (33) [right of = 44] {\tiny $r+1$};
\node(32) [below of = 33]{};
\node(31) [below of = 32]{};
\node[main] (22) [right of = 33] { $r$};
\node(21) [below of = 22]{};
\node[main] (11) [right of = 22] {\tiny $r-1$};

\draw[] (6) -- ([xshift=-5mm,yshift=-5mm]66.center);
\draw[-{Latex[length=3mm]}] ([xshift=-5mm,yshift=-5mm]66.center) -- (66);
\draw[] (5) -- ([xshift=-5mm,yshift=-5mm]55.center);
\draw[-{Latex[length=3mm]}] ([xshift=-5mm,yshift=-5mm]55.center) -- (55);
\draw[-{Latex[length=3mm]}] (4) -- (44);
\draw[] (3) -- ([xshift=-5mm,yshift=-5mm]53.center);
\draw[] ([xshift=-5mm,yshift=-5mm]53.center) -- ([xshift=-5mm,yshift=-5mm]43.center);
\draw[-{Latex[length=3mm]},dotted] ([xshift=-5mm,yshift=-5mm]43.center)--(33);
\draw[] (2) -- ([xshift=-5mm,yshift=-5mm]62.center);
\draw[dotted] ([xshift=-5mm,yshift=-5mm]62.center) -- ([xshift=-5mm,yshift=-5mm]32.center);
\draw[-{Latex[length=3mm]}] ([xshift=-5mm,yshift=-5mm]32.center) -- (22);
\draw[dotted] (1) -- ([xshift=-5mm,yshift=-5mm]61.center);
\draw[] ([xshift=-5mm,yshift=-5mm]61.center) -- ([xshift=-5mm,yshift=-5mm]41.center);
\draw[] ([xshift=-5mm,yshift=-5mm]41.center) -- ([xshift=-5mm,yshift=-5mm]31.center);
\draw[-{Latex[length=3mm]}] ([xshift=-5mm,yshift=-5mm]31.center)--(11);

\draw[-{Latex[length=3mm]},dotted] ([xshift=-5mm,yshift=-5mm]61.center) to ([xshift=-5mm,yshift=-5mm]62.center);
\draw[-{Latex[length=3mm]}] ([xshift=-5mm,yshift=-5mm]41.center) -- ([xshift=-5mm,yshift=-5mm]42.center);
\draw[-{Latex[length=3mm]},dotted] ([xshift=-5mm,yshift=-5mm]42.center) -- ([xshift=-5mm,yshift=-5mm]43.center);
\draw[-{Latex[length=3mm]},dotted] ([xshift=-5mm,yshift=-5mm]32.center) -- ([xshift=-5mm,yshift=-5mm]33.center);

\end{tikzpicture}} 
\caption{}
\end{subfigure}

\caption{The left graph is $G_1$ and the right graph is $G_2$. Edge $e$ is the edge which is removed by \Cref{removal} and lies in column $c$. $\Lambda$ maps each of the dotted paths on the right to the dotted path on the left (case \ref{case3d}). Accordingly, the dotted path on the left is \Cref{incidences} equivalent to each of the dotted paths on the right.}
\label{figupto3}
 \end{figure}

We return to the induction structure for our proof of \Cref{extremunique}. In the final stage of that proof, we will need to induct on the size of our \Pl coordinates. The base case is proved in the following result.

\begin{lemma}\label{kequalsone}
Let $b_1$ be the maximal sink attainable by a path originating at source vertex $1'$ in $G=G_{v,w}$. Then, there exists a unique path from $1'$ to $b_1$ and this path is greedy.
\end{lemma}

\begin{proof}
We begin by observing that by \Cref{bruhatinterval}, $b_1$ is the largest number such that there exists $v^{-1}\leq u\leq w^{-1}$ with $u(1)=b_1$. Thus, by choosing $u=w^{-1}$, $b_1\geq w^{-1}(1)$. However, if $u(1)>w^{-1}(1)$, then it is not possible to have $u<w^{-1}$ by \Cref{tableaucrit}. Thus, $b_1=w^{-1}(1)$. 

This lemma is true if $v=id$ by \Cref{idextremunique}. We continue to use the notation introduced just before \Cref{rem:differentnotations}. Recall that $\mathbf{v}^{(d+1)}=\mathbf{v}^{(d)}s_r$, and accordingly, $G_2=G_{v^{(d)}s_r,w}$. We proceed by induction, assuming that the result holds for $G_1=G_{v^{(d)},w}$ and proving it for $G_2=G_{v^{(d+1)},w}$. Note that since $w$ is fixed throughout this induction, so is $b_1$. 

We consider a few cases.

\begin{enumerate}
    \item If $1'$ lies above strand $r+1$ in $G_2$, then $1'$ also lies above strand $r+1$ in $G_1$, since the source labeling in these graphs differ only by \Cref{interchange}. Since $G_2$ and $G_1$ have precisely the same set of edges with both endpoints lying above strand $r+1$, and there are no vertical edges oriented downwards by \Cref{vertarrow}, the paths originating at $1'$ in $G_2$ and $G_1$ are identical. Thus, by induction, the Lemma  also holds in $G_2$. Similarly, if sink vertex $b_1$ lies to the right of strand $r$ in both $G_2$ and $G_1$, all vertical edges used by a path from $1'$ to $b_1$ lie strictly below strand $r$ and so the paths terminating at $b_1$ in $G_2$ and $G_1$ are identical, proving the lemma for $G_2$.

    \item Suppose $1'$ lies below strand $r$ and strand $b_1$ lies left of strand $r$ in $G_2$. Then, consider $p_1$ and $p_2$ from $1'$ to $b_1$ in $G_2$. We will show they coincide and are greedy. Consider the paths $p_1'=\Lambda(p_1)$ and $p_2'=\Lambda(p_2)$ in $G_1$. The paths $p_1$ and $p_2$ fall under one of subcases \ref{case3b}, \ref{case3c}, \ref{case3d}, or \ref{case3e} of \Cref{def:swapequivalencemap}. A close inspection of these subcases shows that $p_1'$ and $p_2'$ both terminate at $b_1$. Thus, by induction, $p_1'=p_2'$ is a greedy path. 
    
    If at least one of $p_1$ and $p_2$ does not fall under subcase \ref{case3d}, then \Cref{lem:laminverse} implies that $p_1=p_2$. If $p_1$ were not greedy, then by examining the relevant subcases of case \ref{case3}, we observe that $\Lambda(p)$ is not greedy either. This means the unique path in $G_1$ from $1'$ to $b_1$ is not greedy, contradicting our induction hypothesis.

    We are left to consider the case where both $p_1$ and $p_2$ fall under subcase \ref{case3d}. Suppose towards a contradiction that both $p_1$ and $p_2$ enter strand $r$ strictly left of column $c$, use a vertical edge to reach strand $r+1$ and then either exit strand $r+1$ weakly right of column $c$ or terminate on strand $r+1$ in $G_2$. Note that $p_1$ and $p_2$ must use strand $r$ to get from strictly left of column $c$ to strictly right of column $c$, since all edges from strand $r$ to strand $r+1$ in $G_2$ are to the right of column $c$. If $p_1$ and $p_2$ differ before reaching $(r,c)_{G_2}$, then $\Lambda(p_1)\neq \Lambda(p_2)$ in $G_1$, a contradiction. Thus, $p_1$ and $p_2$ are identical before reaching $(r,c)_{G_2}$. Moreover, $p_1$ and $p_2$ are greedy before reaching $(r,c)_{G_2}$ since the path they are \Cref{incidences} equivalent to in $G_1$ is greedy by induction and subcase \ref{case3d} preserves greediness.
    
    We now look at the part of the graph to the right of column $c$. For a subexpression $\mathbf{u}$ of $\mathbf{w_0}$, define $\mathbf{\hat{u}}$ to be the expression obtained from $\mathbf{u}$ by removing the simple reflections in $\mathbf{u}$ coming from the first $c$ runs of $\bm{w_0}$. Let $\hat{u}$ be the resulting permutation. Note that, in particular, $\mathbf{\widehat{w_0}}$ is precisely our expression for the longest permutation in $\mathfrak{S}_{n-c}$. Moreover, $\mathbf{\hat{u}}$ is a positive distinguished subexpression of $\mathbf{\widehat{w_0}}$, as can be checked directly from \Cref{posdist}. By construction, $\widehat{v^{(d+1)}}=id$. Thus, if we ignore the labeling of the source vertices, the part of $G_2$ which lies strictly to the right of column $c$ is identical to $G_{id,\hat{w}}$, with $id$ and $\hat{w}$ viewed as permutations in $\mathfrak{S}_{n-c}$. We give a brief example of this construction in \Cref{ex:truncatedgraph}. By \Cref{coruniquepath}, there is a unique left extreme path in $G_{id,\hat{w}}$ which originates on any given strand, and this path is greedy. Thus, $p_1$ and $p_2$ must continue on from $(r,c)_{G_2}$ following the unique greedy path from strand $r$ to $b_1$ in $G_{id,\hat{w}}$. To summarize, we have shown that $p_1=p_2$ is greedy by showing that any path from $1'$ to $b_1$ in $G_2$ must reach $(r,c)_{G_2}$ in the same way, which is greedy, and then terminate as the unique left extreme path originating on strand $r$ in $G_{id,\hat{w}}$, which is also greedy.   
        \begin{example}\label{ex:truncatedgraph}
            Consider the graph $G_{v,w}=G_{s_2,s_1s_3s_2s_1}$ on the right of \Cref{weightedgraph}. Let $d=0$, so $v^{(d+1)}=v^{(1)}=v$. Let us now identify the positive distinguished subexpression for $w$ in $\mathbf{w_0}=(s_1s_2s_3)(s_1s_2)(s_1)$. It is $(s_1s_3)(s_2)(s_1)$, where simple reflections in the $i^{\textnormal{th}}$ pair of parentheses come from the $i^{\textnormal{th}}$ run of $\mathbf{w_0}$. Similarly, the positive distinguished subexpression for $v$ in $\mathbf{w}$ is $()(s_2)()$. Since the unique simple reflection of $\mathbf{v}$ comes from run $2$ of $\mathbf{w_0}$, $c=2$. Removing the first $2$ runs of $\bm{w}$, we are left with $\bm{\hat{w}}=s_1$, viewed as a positive distinguished subexpression of the longest element $\mathbf{\widehat{w_0}}=s_1$ in $\mathfrak{S}_{4-2}=\mathfrak{S}_2$. Ignoring source labels, the part of $G_{v,w}$ to the right of column $2$ is identical to $G_{id,\hat{w}}=G_{id,s_1}$ for $n=2$.
        \end{example}
       
        \item If $1'$ lies below strand $r$ and  $b_1=r$, then a path $p$ in $G_2$ from $1'$ to $b_1$ falls under subcase \ref{case3d2} or \ref{case3e} of the definition of \Cref{incidences} equivalence. If $p$ falls under subcase \ref{case3d2}, then $b_1=r$ and the \Cref{incidences} equivalent path in $G_1$ terminates at $r+1$, to the left of $b_1$. We noted earlier that $b_1$ is fixed throughout the induction, so this is impossible. Thus, we are in subcase \ref{case3e}. The rest of this argument is similar to the argument employed in the previous numbered bullet. If $p_1$ and $p_2$ are both paths from $1'$ to $b_1$, then $\Lambda(p_1)=\Lambda(p_2)$ by induction. By \Cref{lem:laminverse}, this implies $p_1=p_2$. If the unique path $p$ from $1'$ to $b_1$ in $G_2$ were not greedy, then $\Lambda(p)$ is not greedy, again contradicting our induction hypothesis. 
    
    \item Source $1'$ is not on strand $r$ in $G_2$, since \Cref{interchange} implies $1'$ cannot move from strand $r+1$ to strand $r$.
    
    \item If $1'$ lies on strand $r+1$, we are in case \ref{case2} of the definition of \Cref{incidences} equivalence. We repeat the same argument we have used previously: If $p_1$ and $p_2$ are both paths from $1'$ to $b_1$, then $\Lambda(p_1)=\Lambda(p_2)$ by induction. By \Cref{lem:laminverse}, this implies $p_1=p_2$. If the unique path $p$ from $1'$ to $b_1$ in $G_2$ were not greedy, then $\Lambda(p)$ is not greedy, again contradicting our induction hypothesis. 
\end{enumerate}  

This proves the result for $G_2$ and we are done by induction.
\end{proof}

Our goal in the proof of \Cref{extremunique} will be to show that the union of a greedy path collection with diagonal paths in $G_2$ is, first, extremal and second, unique. The next two Lemmas will help with these goals. Recall that $O_1(q)$ and $O_2(q)$ are the sets of strands weakly above strand $q$ with source label in $[k]'$ in $G_1$ and $G_2$, respectively.

 \begin{lemma}\label{maximalbound} 
 Let $q\in[n]$. Assume $S$ is Gale maximal among the sink sets of all non-intersecting path collections originating at $O_1(q)$ in $G_1$. Further assume that if $q=r+1$, then either both or neither of strands $r$ and $r+1$ are in $O_1$. Then there is no non-intersecting path collection in $G_2$ from $O_2(q)$ to a sink set which is Gale larger than $S$. 
 \end{lemma}
 \begin{proof}

We begin by clarifying that the further assumption when $q=r+1$ simply ensures that $|O_1(q)|=|O_2(q)|$. 

Suppose there exists a non-intersecting path collection $C^2_0$ originating at $O_2(q)$ in $G_2$ which violates the lemma by having a sink set $S'$ Gale larger than $S$. We extend $C^2_0$ to a non-intersecting path collection $C^2$ originating from $O_2$ by adding diagonal paths on each strand in $O_2\setminus  O_2(q)$. By \Cref{bruhatinterval}, the sink set of this path collection is of the form $u[k]$ for some $u$ such that $s_r(v^{(d)})^{-1}\leq u\leq w^{-1}$. However, since $(v^{(d)})^{-1}<s_r(v^{(d)})^{-1}$, this means that $(v^{(d)})^{-1}\leq u\leq w^{-1}$. Thus, there exists a non-intersecting path collection $C^1$ in $G_1$ originating from $O_1$ with this same sink set.

  By \Cref{vertarrow}, $C^1$ must have diagonal paths on strands $O_1\setminus O_1(q)$ in order to reach the sink vertices of the diagonal paths in $C^2$. Removing these diagonal paths, we are left with a non-intersecting path collection from $O_1(q)$ to $S'$, which is Gale larger than $S$, a contradiction.

 \end{proof}

 \begin{lemma}\label{uniqueness}
     
     Suppose that $w=w^{(d+1)}$, so that $G_2=G_{v^{(d+1)},w^{(d+1)}}$. Suppose $C^2$ and $D^2$ are two distinct non-intersecting path collections from strands $O_2$ to $S$ in $G_2$. Then there exist two distinct non-intersecting path collections from strands $O_1$ to sink sets which are at least as large as $S$ in Gale order in $G_1$.
 \end{lemma}

 \begin{proof}
    Recall that the final simple reflection in both $\bm{v}^{(d+1)}$ and $\bm{w}^{(d+1)}$ is an $s_r$. It follows that  $G_2$ has no vertical edges which leave strands $r$ or $r+1$ in column $c$, no vertical edges strictly to the right of column $c$, and no vertical edges between strands $r$ and $r+1$ (since all such edges must lie to the right of column $c$). Accordingly, \Cref{lem:laminverse} holds for all paths in $G_2$.

     Construct the path collections $\tilde{C}^1$ and $\tilde{D}^1$ in $G_1$ which consist of the paths \Cref{incidences} equivalent to each of the paths in $C^2$ and in $D^2$, respectively. These are not necessarily non-intersecting. Modify them as follows: If ever a path $p$ in $\tilde{C}^1$ uses the edge removed by \Cref{removal} to get from strand $r$ to strand $r+1$ and intersects with another path in $\tilde{C}^1$ at the top of that edge, namely the point $(r+1,c)_{G_1}$, then replace $p$ by a path that starts identically to $p$, but terminates on strand $r$ without using the edge removed by \Cref{removal}. Call this path collection $C^1$. Similarly, construct $D^1$. We claim these satisfy the conditions of the theorem. This construction is illustrated in \Cref{fig:Ctilde}. 
     
     First we observe that $C^1$ and $D^1$ are non-intersecting path collections. This is a tedious but straightforward check which can be done by considering paths in different subcases of \Cref{def:swapequivalencemap}. Specifically, paths which do not intersect strands $r$ or $r+1$ are unchanged by \Cref{incidences} equivalence, so all one needs to do is to consider, pairwise, paths intersecting strands $r$ and $r+1$ which lie in different subcases of \Cref{def:swapequivalencemap} and which do not enter or leave strands $r$ or $r+1$ to the right of column $c$. 

Next we observe that $C^1$ and $D^1$ are distinct. First, $\tilde{C}^1$ and $\tilde{D}^1$ are distinct since \Cref{lem:laminverse} holds for all paths in $G_2$. For $p$ a path in $G_2$, note that if $\Lambda(p)$ reaches the point $(r,c)_{G_1}$ at the bottom of the edge $e$ removed by \Cref{removal}, it actually uses $e$. It follows that the modification we make to construct $C^1$ and $D^1$ from $\tilde{C}^1$ and $\tilde{D}^1$ is invertible: If ever a path in $C^1$ or $D^1$ takes a diagonal strand past the bottom of the edge removed by \Cref{removal}, we know that the corresponding path in $\tilde{C}^1$ or $\tilde{D}^1$ should use that edge. Thus, the entire construction of $C^1$ and $D^1$ from $C^2$ and $D^2$ is invertible. Since $C^2$ and $D^2$ are distinct, so are $C^1$ and $D^1$. 

Finally, we show that the sink sets of $C^1$ and $D^1$ are at least Gale as large as $S$. Without loss of generality, we focus on $C^1$. Note that a path $p$ in $G_2$ terminates at a sink at most as large as the sink of $\Lambda(p)$. Thus, each path of $\tilde{C}^1$ terminates at a sink at least as large as the sink of the corresponding path in $C^2$. If $C^1=\tilde{C}^1$, we are done. Otherwise, $C^1$ is obtained from $\tilde{C}^1$ by replacing a path $p$ which uses the edge removed by \Cref{removal} with a path terminating at sink $r$. Specifically, this occurs if there is a second path $p'$ which passes the point $(r+1,c)_{G_1}$ at the terminus of the edge removed by \Cref{removal}. Observe that, since there are no vertical edges leaving strand $r+1$ to the right of column $c$, this means that both $p$ and $p'$ terminate at $r+1$ in $\tilde{C}^1$. A path which is \Cref{incidences} equivalent to a path terminating at $r$ or $r+1$ must itself terminate at $r$ or $r+1$. Since $C^2$ is non-intersecting, of the paths which $\Lambda$ maps to $p$ and $p'$, exactly one terminates at $r$ and one at $r+1$. Therefore, $C^1$ only differs from $\tilde{C}^1$ if the sink set $S$ of $C^2$ contains both $r$ and $r+1$, and $\tilde{C}^1$ has two paths terminating at $r+1$. Then, by our modification, $C^1$ contains both $r$ and $r+1$ in its sink set. Together with the fact that each other path in $C^1$ terminates at a sink at least as large as the corresponding path in $C^2$, we conclude that the sink set of $C^1$ is at least as large as $S$ in Gale order.  

\end{proof}

    \begin{figure}[H]
 \begin{tikzpicture}[node distance={10.5 mm}, thick, main/.style = {draw, circle,minimum size=2 mm}, 
blank/.style={circle, draw=green!0, fill=green!0, very thin, minimum size=3.5mm},]

\node[main] (1) {$2'$};
\node[main] (2) [above of=1] {$1'$};
\node[main] (3) [above of = 2] {$3'$}; 
\node (blank)[above of = 3]{};
\node[main] (33) [right of = blank] {$3$};
\node(32) [below of = 33]{};
\node(31) [below of = 32]{};
\node[main] (22) [right of = 33] {$2$};
\node(21) [below of = 22]{};
\node[main] (11) [right of = 22] {$1$};

\draw[-{Latex[length=3mm]}] (3) -- (33);
\draw[-{Latex[length=3mm]},dashed] (2)--(22);
\draw[-{Latex[length=3mm]},red] (1)--(11);

\end{tikzpicture} 
 \begin{tikzpicture}[node distance={10.5 mm}, thick, main/.style = {draw, circle,minimum size=2 mm}, 
blank/.style={circle, draw=green!0, fill=green!0, very thin, minimum size=3.5mm},]
\node[main] (1) {$1'$};
\node[main] (2) [above of=1] {$2'$};
\node[main] (3) [above of = 2] {$3'$}; 
\node (blank)[above of = 3]{};
\node[main] (33) [right of = blank] {$3$};
\node(32) [below of = 33]{};
\node(31) [below of = 32]{};
\node[main] (22) [right of = 33] {$2$};
\node(21) [below of = 22]{};
\node[main] (11) [right of = 22] {$1$};

\draw[-{Latex[length=3mm]}] (3) -- (33);
\draw[red] (2)--([xshift=-5mm,yshift=-5mm]32.center);
\draw[-{Latex[length=3mm]},dashed,red] ([xshift=-5mm,yshift=-5mm]32.center)--(22);

\draw[dashed] (1)--([xshift=-5mm,yshift=-5mm]31.center);
\draw[-{Latex[length=3mm]}] ([xshift=-5mm,yshift=-5mm]31.center) -- (11);

\draw[-{Latex[length=3mm]},dashed] ([xshift=-5mm,yshift=-5mm]31.center) -- ([xshift=-5mm,yshift=-5mm]32.center);

\end{tikzpicture} 
\begin{tikzpicture}[node distance={10.5 mm}, thick, main/.style = {draw, circle,minimum size=2 mm}, 
blank/.style={circle, draw=green!0, fill=green!0, very thin, minimum size=3.5mm},]
\node[main] (1) {$1'$};
\node[main] (2) [above of=1] {$2'$};
\node[main] (3) [above of = 2] {$3'$}; 
\node (blank)[above of = 3]{};
\node[main] (33) [right of = blank] {$3$};
\node(32) [below of = 33]{};
\node(31) [below of = 32]{};
\node[main] (22) [right of = 33] {$2$};
\node(21) [below of = 22]{};
\node[main] (11) [right of = 22] {$1$};

\draw[-{Latex[length=3mm]}] (3) -- (33);
\draw[red] (2)--([xshift=-5mm,yshift=-5mm]32.center);
\draw[-{Latex[length=3mm]},red] ([xshift=-5mm,yshift=-5mm]32.center)--(22);

\draw[dashed] (1)--([xshift=-5mm,yshift=-5mm]31.center);
\draw[-{Latex[length=3mm]},dashed] ([xshift=-5mm,yshift=-5mm]31.center) -- (11);

\draw[-{Latex[length=3mm]}] ([xshift=-5mm,yshift=-5mm]31.center) -- ([xshift=-5mm,yshift=-5mm]32.center);

\end{tikzpicture} 
\caption{The left image shows a non-intersecting path collection $C^2$ in the graph $G_2=G_{s_1,s_1}$, consisting of a dashed path originating on strand $2$ and a red path originating on strand $1$. The middle image shows the path collection $\tilde{C}^1$ consisting of the paths which are \Cref{incidences} equivalent to $C^2$ in $G_1=G_{id,s_1}$. Observe that it is not a non-intersecting path collection. The right image shows the non-intersecting path collection $C^1$ obtained by replacing the dashed path in $\tilde{C}^1$ by a horizontal path.}
\label{fig:Ctilde}
 \end{figure}

We now move on to the proof of \Cref{extremunique}. The proof first addresses the special case $w=w^{(d+1)}$ using a careful induction, before finishing with the general case.

\begin{proof}[Proof of \Cref{extremunique}]

 We begin by showing that extremal path collections are unique and consist of a union of a diagonal path collection and a greedy path collection. The converse will be addressed later. 
 
 Our proof strategy is as follows: We will work by induction, assuming the result holds for $G_1=G_{{v^{(d)}},w}$ and proving it for $G_2=G_{{v^{(d+1)}},w}$. For each $q\in[n]$ and $i\in \{1,2\}$, we will consider the path collection $C^i_q$ in $G_i$ which originates from strands $O_i$ and consists of diagonal paths originating below strand $q$, and a greedy path collection originating weakly above strand $q$. We denote by $S^i_q$ the sink set of $C^i_q$. We will show that the greedy part of $C^2_q$ is left extreme and that $C^2_q$ is unique, proving that for each extremal index, there is a unique extremal path collection consisting of a union of greedy and diagonal paths. We will do all this first in the special case where $w=w^{(d+1)}$, so that the last simple reflection in $\mathbf{w}$ is also the last simple reflection in $\mathbf{v}^{(d+1)}$, and then consider the general case later.

 We proceed by induction on $d$, with base case $v^{(0)}=id$ considered in \Cref{idextremunique}. Let $k=|O_2|$ be the size of the indices being considered, which will be fixed throughout the proof. Recall that $O_2(q)$ is the part of $O_2$ weakly above strand $q$.

For now, let $w=w^{(d+1)}$, which guarantees that $G_2$ has no vertical edges which leave strands $r$ or $r+1$ weakly to the right of column $c$ or enter it strictly to the right of column $c$. We will compare $C^2_q$ with $C^1_q$. Fix $e$ to be the edge which gets removed by \Cref{removal}.
We establish three useful observations: 

\begin{enumerate}[label=\roman*.]

\item 
\label{firstrestriction}

Let $s_r$ be the simple reflection swapping $r$ and $r+1$. Suppose a greedy path $p_\sigma^{C_q^1}$ originating on strand $\sigma$ in the greedy path collection $C^1_q$ uses vertical edges $E$. Then, the path $p_{s_r(\sigma)}^{C^2_q}$ originating on strand $s_r(\sigma)$ in $C^2_q$ uses the vertical edges $E\setminus \{e\}$, that is, the same vertical edges except possibly for $e$. This is straightforward to verify using the facts that (1) other than edge $e$, if there is an edge between strands $\sigma_1$ and $\sigma_2$ in $G_1$, then the corresponding edge in $G_2$ goes from strand $s_r\sigma_1$ to strand $s_r\sigma_2$ and (2) if a path in $G_1$ uses edge $e$, it is necessarily the final vertical edge used by that path.

\item \label{secondrestriction} Suppose $p$ is a path in ${C^1_q}$ which is non-diagonal and terminates on strand $r+1$. Then, it must use edge $e$: Let $\sigma$ be the strand on which $p$ originates. The claim is clear if $\sigma=r$, since $e$ is the unique edge from strand $r$ to strand $r+1$ when $w=w^{(d+1)}$. For the rest of this argument, we suppose $\sigma\neq r$. By the inductive definition of greedy path collections, it suffices to show the result for $q=\sigma$. Suppose, towards a contradiction, that in ${C^1_{\sigma}}$, $p$ originates from strand $\sigma$, terminates on strand $r+1$, and does not use $e$. 

We first show that no path of $C^1_\sigma$ terminates on strand $r$. Suppose otherwise. A path $p_1$ terminating on strand $r$ must pass by the origin of $e$. By greediness, there must be a second path $p_2$ blocking $p_1$ from using $e$. Since $p$ is the greedy path with the lowest origin in $C_\sigma^1$, we know that $p$ is different from $p_2$. However, since $e$ is the rightmost vertical edge incident to strand $r+1$, $p_2$ would block $p$ from terminating on strand $r+1$, a contradiction. 

We next observe that $S_\sigma^2=(S_\sigma^1\setminus (r+1))\cup r$. For $\sigma'>\sigma$, the path path $p_{\sigma'}^{C_{\sigma}^1}$ originates above strand $\sigma$ and thus terminates on neither strand $r$ nor $r+1$. It follows from \hyperref[firstrestriction]{(i)}, that $p_{\sigma'}^{C_{\sigma}^2}$ terminates at the same sink as $p_{\sigma'}^{C_{\sigma}^1}$. Since $p$ terminates on strand $r+1$ and is non-diagonal, and also since we have already considered $\sigma=r$, we may assume $\sigma<r$. Moreover, by assumption (still towards a contradiction), $p$ does not use edge $e$. By \hyperref[firstrestriction]{(i)},  $p_\sigma^{C^2_\sigma}$ uses precisely the same vertical edges as $p$. Let $e'\neq e$ be the vertical edge which $p$ uses to reach strand $r+1$. This is the last vertical edge used by $p$ along its path. After applying \Cref{incidences}, the edge $e'$ terminates on strand $r$ in $G_2$, and it is the last edge used by $p_\sigma^{C^2_\sigma}$ along its path. Thus, $p_\sigma^{C^2_\sigma}$ terminates on strand $r$. Paths originating below strand $\sigma$ in $C_\sigma^1$ are diagonal paths originating below strand $r$ and are thus identical to the corresponding paths in $G_2$.

We now finally reach our contradiction. By \Cref{bruhatinterval}, $S^2_\sigma=(S^1_\sigma\setminus (r+1))\cup r=u([k])$ for some $u$ such that $s_r(v^{(d)})^{-1}\leq u\leq w^{-1}$. This implies that $(v^{(d)})^{-1}\leq u\leq w^{-1}$. Thus, there is some non-intersecting path collection $D$ in $G_1$ with sink set $(S\setminus (r+1))\cup r$. Since there are no edges leaving strand $r+1$ to the right of column $c$ and there is no path in $D$ terminating on strand $r+1$, we can modify $D$ by replacing the path terminating on strand $r$ with a path which begins identically but uses the edge $e$ and terminates on strand $r+1$ instead of $r$. This gives a non-intersecting path collection with source set $[k]'$ and sink set $S^1_\sigma$, which differs from $C^1_\sigma$ in its use of the edge $e$. This contradicts the uniqueness part of our induction hypothesis. Thus, we conclude that if $p$ terminates on strand $r+1$ and is non-diagonal, it must use edge $e$, as desired. 

\item \label{thirdrestriction} It is impossible for $C^1_q$ to have a greedy path terminating on strand $r$ unless it also has a greedy path terminating at $r+1$ which is diagonal: If a greedy path $p$ in $C^1_q$ terminated on strand $r$, then some other path $p'$ must have blocked it from using $e$. Since there are no vertical edges above and to the right of edge $e$, $p'$ must terminate on strand $r+1$. However, by \hyperref[secondrestriction]{(ii)}, if $p'$ were not diagonal, it uses edge $e$. No vertical edge terminates on strand $r$ to the right of $e$, so $p'$ would block any path from terminating on strand $r$.

\end{enumerate}

For now, we assume that $q\neq r+1$; we will handle the $q=r+1$ case separately in \Cref{lem:q=r+1}.

Recall that if a path collection $C$ with sink set $S$ originates at $O$, we denote by $C(q)$ the subcollection originating at $O(q)$ and by $S(q)$ the sink set of $O(q)$. We next show that $S^1_q(q)=S^2_q(q)$, that is, the greedy parts of $C^1_q$ and $C^2_q$ have identical sink sets. First suppose $r+1\notin S^1_q(q)$. By \hyperref[thirdrestriction]{(iii)}, we also have $r\notin S^1_q(q)$. Fix a path $p$ in $C^1_q(q)$ and let $e'$ be the last vertical edge it uses. Then $e'$ terminates on a strand $\sigma\notin\{r,r+1\}$, and so the corresponding edge in $G_2$ also terminates at $\sigma$. The claim is then immediate from \hyperref[firstrestriction]{(i)}. Next, suppose $r+1\in S^1_q(q)$ and the path $p$ terminating there is non-diagonal. By \hyperref[thirdrestriction]{(iii)}, $r\notin S^1_q(q)$. By \hyperref[secondrestriction]{(ii)}, $p$ must use the edge $e$ removed by \Cref{removal}. If $p$ originates on strand $r$, then \hyperref[firstrestriction]{(i)} guarantees a diagonal path on strand $r+1$ in $C^2_q(q)$. If $p$ originates below strand $r$, let $e'$ be the edge it uses to reach strand $r$. The corresponding edge in $G_2$ then terminates at $s_r r=r+1$. Thus, again by \hyperref[firstrestriction]{(i)}, we will have a path terminating at $r+1$ in $C^2_q(q)$ and so $S_q^1(q)=S_q^2(q)$. Finally, suppose $r+1\in S^1_q(q)$ and the path $p$ terminating there is diagonal. Note that there must be a path originating on strand $r$ in $C^1_q$, since the vertex label on strand $r$ is smaller than the vertex label on strand $r+1$ by \Cref{interchange}. Using \hyperref[firstrestriction]{(i)}, it suffices to show that $p_{r}^{C_q^1}$, $p_{r+1}^{C_q^2}$ and $p_{r}^{C_q^2}$ are diagonal too. Note that, by \hyperref[firstrestriction]{(i)}, any path originating above strand $r+1$ is identical in both $C^1_q$ and $C^2_q$. Thus, $p_{r+1}^{C_q^2}$ must be diagonal to satisfy \Cref{maximalbound}. Again by \hyperref[firstrestriction]{(i)}, the greedy path $p_r^{C_q^1}$ can use no vertical edge other than $e$, which it is blocked from doing by the diagonal path $p_{r+1}^{C_q^1}$, so $p_r^{C_q^1}$ is diagonal. Once again by \hyperref[firstrestriction]{(i)}, $p_{r}^{C_q^2}$ must use a subset of the edges used by the diagonal path $p_{r+1}^{C_q^1}$ and so is itself diagonal.

Assuming \Cref{lem:q=r+1}, which we will prove soon, we can now conclude our proof for the $w=w^{(d+1)}$ case. The greedy part of $C_q^2$ is left extreme by \Cref{maximalbound} and \Cref{galemax}. Uniqueness follows from \Cref{uniqueness}.

We consider the general case, where we do not necessarily have $w=w^{(d+1)}$. We continue to assume $q\neq r+1$.
 
  Let $G'_2=G_{v^{(d+1)},w^{(d+1)}}$. Observe that $G'_2$ is obtained from $G_2=G_{v^{(d+1)},w}$ by removing some vertical edges to the right of column $c$, or in column $c$ and above strand $r$. Given a path collection $C$ in $G_2$, define the truncation $T(C)$ in $G'_2$ to be the non-intersecting path collection which is identical to $C$ to the left of column $c$ or in column $c$ and weakly below strand $r+1$, and then terminates diagonally. Similarly, one constructs $G'_1=G_{v^{(d)},w^{(d+1)}}$ and, for a non-intersecting path collection $C$ in $G_1$, one obtains the truncation $T(C)$ in $G'_1$. Note that the set of vertical edges in $G_1$ but not $G'_1$ is precisely the same as the set of vertical edges in $G_2$ but not $G'_2$ (and in bijection with the simple reflections in $\bm{w}$ but not $\bm{w^{(d+1)}}$). We denote this common set by $E$.

Let $C^1_q$ and $C^2_q$ be defined as in the special case above. By induction, $C^1_q$ is extremal in $G_1$ and is unique. We show the same is true for $C^2_q$. 

Since $T(C^2_q)$ is the union of a greedy path collection and a diagonal path collection in $G'_2$, we know it is extremal and unique by the special case $w=w^{(d+1)}$ considered earlier. Moreover, it follows from the proof in that special case that $T(C^2_q)$ has the same sink set as $T(C^1_q)$. This also implies the sink set $S^1_q$ of $C^1_q$ equals the sink set $S^2_q$ of $C^2_q$, since $C^1_q$ and $C^2_q$ are obtained from $T(C^1_q)$ and $T(C^2_q)$, respectively, by greedily using the same set of edges $E$.

Suppose, towards a contradiction, there were a different path collection $C\neq C_q^2$ in $G_2$ with a sink set $S$ such that $S(q)$ is at least as large as $S^2_q(q)=S^1_q(q)$ in Gale order. We will use this to construct a path collection $\tilde{C}$ in $G_1$ which contradicts the induction hypothesis that $C^1_q$ is the unique extremal path collection.

Let $T(S)$ be the sink set of $T(C)$. Then $C$ begins identically to $T(C)$ before using a subset of the edges of $E$ to get from strands $T(S)$ to the sink set $S$. Note that since $T(C)$ terminates at $T(S)$, by \Cref{bruhatinterval}, we have $T(S)=u[k]$ for some $s_r(v^{(d)})^{-1}\leq u\leq (w^{(d+1)})^{-1}$. This implies that $(v^{(d)})^{-1}\leq u\leq (w^{(d+1)})^{-1}$ and so there exists a non-intersecting path collection $\hat{C}$ in $G'_1$ from strands $O_1$ to $T(S)$. Define a non-intersecting path collection $\tilde{C}$ in $G_1$ from $O_1$ to $S$ which extends $\hat{C}$ as follows: it begins identically to $\hat{C}$, and then uses the same subset of $E$ that is used by $C$ to get from strands $T(S)$ to $S$. If $S(q)$ is strictly larger than $S^1_q(q)$ in Gale order, we contradict the fact that $C^1_q(q)$ is left extreme. Otherwise, $S(q)=S^1_q(q)$ and we argue that $\tilde{C}$ differs from $C^1_q$, contradicting the uniqueness of $C^1_q$ and completing our proof. By the uniqueness of extremal path collections in $G_2'$, if $T(C)\neq T(C^2_q)$, then $T(S)$ differs from the sink set of $T(C^2_q)$, which is also the sink set of $T(C^1_q)$. By construction, the sink set of $T(\tilde{C})$ is $T(S)$, so $\tilde{C}\neq C^1_q$. Otherwise, $T(C)=T(C^2_q)$ and so, since $C\neq C^2_q$, the non-diagonal paths of $C$ must use a subset of $E$ in a way that is not greedy. However, by construction, the same would be true of $\tilde{C}$ and so $\tilde{C}\neq C^1_q$.

For the converse, observe that every path collection in $G_2$ that is the union of a greedy path collection originating from the top $q$ vertices of $[|I|]'$ and a diagonal path collection originating from the remaining vertices of $[|I|]'$ is of the form $C^2_q$. 
\end{proof}

\begin{lemma}\label{lem:q=r+1}
    Let $v\leq w\in \mathfrak{S}_n$. Let $\mathbf{w}$ be the positive distinguished subexpression for w  in $\mathbf{w_0}$. Assume the final transposition in the positive distinguished subexpression $\mathbf{v}$ for v in $\mathbf{w}$ is $s_r$. Then, for any $1\leq k\leq n-1$, there is a unique extremal path collection of size $k$ consisting of diagonal paths below strand $r+1$ and a left extreme path collection weakly above strand $r+1$. Moreover, this left extreme path collection is greedy. 
\end{lemma}

\begin{proof}
    We continue using the same notations as in the proof of \Cref{extremunique}. We also adopt the induction structure on the value of $d$, assuming the result holds for $G_1$ and proving it for $G_2$. The base case was proven before we assumed $q\neq r+1$. We may also assume that \Cref{extremunique} holds for any extremal path collection with $q\neq r+1$. In this proof, we do not condition on whether $w=w^{(d+1)}$.

For a fixed value of $d$, we induct on the size of indices being considered. Let $k=|O_2|$. We assume the result holds for extremal path collections consisting of fewer than $|O_2|$ many paths, with the base case of single paths considered in \Cref{kequalsone}.

If strand $r+1$ is not in $O_2$, then $C^2_{r+1}=C^2_{r+2}$ and we conclude the desired result from \Cref{extremunique}. Similarly, if strand $r$ were not in $O_2$, we would have $C^2_{r+1}=C^2_{r}$ and could conclude the desired result from \Cref{extremunique}. Thus, we assume that both strands $r$ and $r+1$ are in $O_2$ and, equivalently, in $O_1$. If $S_{r+1}^2(r+1)> S_{r+1}^1(r+1)$, we have a contradiction to \Cref{maximalbound} combined with \Cref{galemax}. If $S_{r+1}^2(r+1)= S_{r+1}^1(r+1)$, we have that $C^2_{r+1}$ is extremal by \Cref{maximalbound} combined with \Cref{galemax}, and uniqueness follows from \Cref{uniqueness}. Thus, we may assume $S_{r+1}^2(r+1)< S_{r+1}^1(r+1)$. 

    By \Cref{addins} and our graphical description of extremal indices, we have that the sink sets of the left extreme parts of extremal path collections form a flag. Explicitly, for any $1\leq i\leq n-1$,  $S_{i+2}^2(i+2)\subset S_{i+1}^2(i+1)\subset S_{i}^2(i)$ (where we define $S_{n+1}^2(n+1)=\emptyset$). Suppose that $S^2_r(r)=S^2_{r+2}(r+2)\cup\{a,b\}$ for $a<b$. From the proof of \Cref{extremunique}, $S^2_r=S^1_r$ and $S^2_{r+2}=S^1_{r+2}$. To have $S_{r+1}^2(r+1)<S_{r+1}^1(r+1)$, we must have that $p_{r+1}^{C^2_{r+1}}$ terminates at $a$ and $p_{r+1}^{C^1_{r+1}}$ terminates at $b$. Accordingly, $p_{r}^{C^2_{r}}$ terminates at $b$ and $p_{r}^{C^1_{r}}$ terminates at $a$. In order to show $C^2_{r+1}$ is extremal, it suffices to prove there is no non-intersecting path collection in $G_2$ from $O_2$ to $S^1_{r+1}=(S^2_{r+1}\setminus a)\cup b$, as $S_{r+1}^1(r+1)$ is the only set other than $S_{r+1}^2(r+1)$ forming a flag with $S^2_{r+2}(r+2)$ and $S^2_r(r)$.

     We will now consider path collections of size $k-1$. Thus, it will be useful to define $\hat{O}_1$ and $\hat{O}_2$ to be the sets of strands with source labels $[k-1]'$ in $G_1$ and $G_2$, respectively. If source $k'$ lies below strand $r+1$, then the greedy part of the path collection $D$ obtained from $C^2_{r+1}$ by deleting the diagonal path originating from $k'$ is identical to the greedy part of $C^2_{r+1}$. In this case, it is straightforward to see by induction on $k$ that $C^2_{r+1}$ is extremal and unique. Note that the source label on strand $r$ in $G_2$ is larger than the source label on strand $r+1$ by \Cref{interchange}. Since we have both $r$ and $r+1$ in $O_2$, we cannot have $k'$ on strand $r+1$. Thus, we assume that $k'$ lies above strand $r+1$. 
    
    Suppose towards a contradiction there were a non-intersecting path collection $\Gamma$ in $G_2$ from $O_2$ to $S^1_{r+1}$. Apply \Cref{lem:constructionstar} to $C^2_{r+1}(r+1)$ to obtain a greedy path collection $\tilde{D}^2_{r+1}$ originating from $\hat{O}_2(r+1)$. Add the diagonal part of $C^2_{r+1}$ to $\tilde{D}^2_{r+1}$ to obtain a path collection $D^2_{r+1}$ in $G_2$ from sources $\hat{O}_2$ to sink set $S^2_{r+1}\setminus \kappa$ for some $\kappa$. We note that $D^2_{r+1}$ is the union of the greedy path collection $\tilde{D}^2_{r+1}$ and some diagonal paths so, by induction on $k$, $D^2_{r+1}$ is extremal and unique. By deleting the path originating from $k'$ in $\Gamma$, there is also a non-intersecting path collection $\Delta$ in $G_2$ from strands $\hat{O}_2$ to a subset of $S^1_{r+1}$ of size $|S^1_{r+1}|-1$. By \Cref{galemax} and \Cref{maximalbound}, the sink set of the greedy part of $\Delta$ must be no larger than the sink set of the greedy part of $D^2_{r+1}$ in Gale order. This is only possible if $\kappa\leq a$, since $S^2_{r+1}$ has more elements less than or equal to $a$ than does $S^1_{r+1}=(S^2_{r+1}\setminus a)\cup b$. Since $k'$ lies above strand $r+1$, we have $\kappa>r+1$. Thus,  $r+1<\kappa\leq a$. We separate into cases depending on whether $\kappa<a$ or $\kappa=a$, still looking for a contradiction in both cases.

     Recall that \Cref{cor:LGVcor} says that $P_{S_1}P_{S_2}$, expressed as a polynomial in the weights $a_i$, has a monomial for each pair of non-intersecting path collections, one from $[|S_1|]'$ to $S_1$ and the other from $[|S_2|]'$ to $S_2$.

    Suppose $\kappa<a$. Then, let $S=S^2_{r+2}$. Since $r+1<\kappa<b$, we can consider the \Pl relation $P_{(S\setminus\kappa)\cup b}P_{S\setminus (r+1)}=P_{S}P_{(S\setminus (r+1)\kappa)\cup b}+P_{(S\setminus (r+1))\cup b}P_{S\setminus \kappa}$ in the cell corresponding to $G_2$. Note that since $C^2_{r+2}$ is extremal by \Cref{extremunique}, there is no path collection with sink set $(S\setminus \kappa)\cup b$, as this would involve increasing the sink set of the greedy part of $C^2_{r+2}$. Thus, the left side of this \Pl relation is zero. The right hand side must be $0$ as well. Since all \Pl coordinates are a subtraction free combination of the $a_i$ by \Cref{sumonly}, we must in particular have $P_{(S\setminus (r+1))\cup b}P_{S\setminus \kappa}=0$. The greedy path with the bottom-most source in $C^2_{r+1}$ is the path from strand $r+1$ to $a$. Since $\kappa\neq a$, it follows from \Cref{lem:constructionstar} that the path originating on strand $r+1$ in $D^2_{r+1}$ also terminates at $a$. Replacing this path by a diagonal path, we see that there is a path collection originating from sources $[k-1]'$ with sink set $S\setminus \kappa$. This implies that $P_{(S\setminus (r+1))\cup b}=0$, contradicting the existence of $\Gamma$. 

    The only other option is $\kappa=a$. We will again derive a contradiction. Now, let $S=S^2_{r+1}$. Consider the \Pl relation $P_{(S\setminus a \cup b)}P_{S\setminus r}=P_SP_{(S\setminus ar)\cup b}+P_{S\setminus r\cup b}P_{S\setminus a}$ in the cell corresponding to $G_2$. Note that $(S\setminus r)\cup b=S_r^2$, and $S\setminus a$ is extremal since $\kappa=a$. By \Cref{extremunique}, there is a unique extremal path collection with sink set $S_r^2$ and it is greedy weakly above strand $r$. There is a unique extremal path collection with sink set $S\setminus a$ and it is greedy path weakly above strand $r+1$, by induction on the size of extremal indices. Thus, $P_{S\setminus r\cup b}P_{S\setminus a}$ is a monomial. We derive a contradiction by showing that there is no pair of path collections contributing to $P_{(S\setminus a \cup b)}P_{S\setminus r}$ with the same weight.

Let $\mathcal{V}^1=(\mathcal{V}_1=C_{r}^2,\hat{\mathcal{V}}_1)$ be the unique pair of non-intersecting path collections contributing to $P_{S\setminus r\cup b}P_{S\setminus a}$. Let $\mathcal{V}^2=(\mathcal{V}_2,\hat{\mathcal{V}}_2)$ be any pair of non-intersecting path collections contributing to the term $P_{(S\setminus a \cup b)}P_{S\setminus r}$. We will show its total weight cannot be equal to the monomial $P_{S\setminus r\cup b}P_{S\setminus a}$, that is, the total weight of $\mathcal{V}^1$. Recall the notation $\epsilon$ introduced in \Cref{rem:lexmax}. 

In order for the weight of $\mathcal{V}^2$ to be the same as the weight of $\mathcal{V}^1$, we must (at least) have $\epsilon(\mathcal{V}_2)+\epsilon(\hat{\mathcal{V}}_2)=\epsilon(\mathcal{V}_1)+\epsilon(\hat{\mathcal{V}}_1)$. In fact, since diagonal paths do not contribute to $\epsilon$, it suffices to focus on the non-diagonal parts of these path collections, namely, $\mathcal{V}_1(r)$, $\hat{\mathcal{V}}_1(r+1)$, $\mathcal{V}_2(r+1)$ and $\hat{\mathcal{V}}_2(r)$. We will use the fact that a greedy path collection originating from source vertices $O$ uniquely maximizes $\epsilon$ amongst all non-intersecting path collections originating from $O$ in lexicographic order to show this equality cannot hold.

 We will also need a few helpful observations. First, $\hat{\mathcal{V}}_2(r+1)$ is greedy. If it were not, then at some point, a path $q$ of $\hat{\mathcal{V}}_2(r+1)$ fails to use a vertical edge $e$ which it could have used, and which was used by the greedy $\hat{\mathcal{V}}_1(r+1)$. In particular, $q$ passes along the bottom edge of $e$. Thus, no path of $\hat{\mathcal{V}}_2$ can use the edge $e$. If the total weight of $\mathcal{V}^2$ is to equal the weight of $\mathcal{V}^1$, then $e$ must be used by $\mathcal{V}_2$, but not by $\mathcal{V}_1$. We show that this is impossible. In fact, every edge used by $\hat{\mathcal{V}}_1(r+1)$ is used by $\mathcal{V}_1$: We obtain the greedy path collection $\hat{\mathcal{V}}_1(r+1)$ from the greedy path collection $\mathcal{V}_1(r+1)$ by \Cref{lem:constructionstar}. A close reading of the proof of this Lemma proves the claim. In particular, we have shown that $\hat{\mathcal{V}}_2(r+1)=\hat{\mathcal{V}}_1(r+1)$ are both the greedy path collection originating from $\hat{O}_2(r+1)$.

Next, let $p$ be the path originating on strand $r$ in $\mathcal{V}_1$ and let $\hat{p}$ be the path originating on strand $r$ in $\hat{\mathcal{V}}_2$. Thus, $p=p_{r}^{C^2_{r}}$, which terminates at $b$ by assumption. Let $\hat{\mathcal{U}}$ be the greedy path collection originating from $\hat{O}_2(r)$. We now show that $p$ appears in $\hat{\mathcal{U}}$ and use this to conclude that $\epsilon(\hat{p})\leq \epsilon(p)$. Note that since $\hat{\mathcal{U}}$ is greedy and originates from strands $\hat{O}_2(r)$, it must be left extreme by induction on $k$. By \Cref{galemax}, $\hat{\mathcal{U}}$ has the Gale maximal sink set amongst non-intersecting path collections originating from strands $\hat{O}_2(r)$. Observe that $\hat{\mathcal{U}}$ is obtained from the greedy path collection $\mathcal{V}_1(r)$ by using \Cref{lem:constructionstar}. By the inductive nature of the proof of \Cref{lem:constructionstar}, $\hat{\mathcal{U}}$ will contain $\hat{\mathcal{V}}_1(r+1)$ (which is obtained via \Cref{lem:constructionstar} from $\mathcal{V}_1(r+1)$) and either $p$, if it is greedy, or, if not, some other path $\tilde{p}$ originating from strand $r$ which is greedy. Again using the proof of \Cref{lem:constructionstar}, in the latter case, $\tilde{p}$ must terminate at the unique element in the sink of $\mathcal{V}_1(r+1)=C^2_{r}(r+1)$ but not in the sink of $\hat{\mathcal{V}}_1(r+1)$. By definition, this is $\kappa=a$. Thus, by Gale maximality of its sink set, $\hat{\mathcal{U}}$ contains $p$, which terminates at $b>a$, rather than $\tilde{p}$. 

Observe that $\hat{\mathcal{V}}_2(r)$ originates on strands $\hat{O}_2(r)$ and therefore, by greediness of $\hat{\mathcal{U}}$, $\epsilon(\hat{\mathcal{V}}_2(r))\leq \epsilon(\hat{\mathcal{U}})$. However, since $\hat{\mathcal{U}}$ contains $\hat{\mathcal{V}}_1(r+1)=\hat{\mathcal{V}}_2(r+1)$ and the path $p$, we can subtract $\epsilon(\hat{\mathcal{V}}_2(r+1))$ from both sides to obtain $\epsilon(\hat{p})\leq \epsilon(p)$. 

Finally, we observe that since $\mathcal{V}_1(r+1)$ is greedy but $\mathcal{V}_2(r+1)$ is not, we have $\epsilon(\mathcal{V}_2(r+1))< \epsilon(\mathcal{V}_1(r+1))$. Similarly, since  $\hat{\mathcal{V}}_1(r+1)$ is greedy, we have $\epsilon(\hat{\mathcal{V}}_2(r+1))\leq \epsilon(\hat{\mathcal{V}}_1(r+1))$. Putting it all together,

\begin{align*}
\epsilon(\mathcal{V}_2)+\epsilon(\hat{\mathcal{V}}_2) &=\epsilon(\mathcal{V}_2(r+1))+\epsilon(\hat{\mathcal{V}}_2(r)) \\ &  =\epsilon(\mathcal{V}_2(r+1))+\epsilon(\hat{\mathcal{V}}_2(r+1))+\epsilon(p')\\ & <\epsilon(\mathcal{V}_1(r+1))+\epsilon(\hat{\mathcal{V}}_1(r+1))+\epsilon(p) \\ & =\epsilon(\mathcal{V}_1(r))+\epsilon(\hat{\mathcal{V}}_1(r+1)) \\ & 
=\epsilon(\mathcal{V}_1)+\epsilon(\hat{\mathcal{V}}_1),
\end{align*}

proving that $\mathcal{V}^1$ and $\mathcal{V}^2$ cannot have the same weight. As a result, it is impossible to satisfy the \Pl relation, contradicting our assumption that $\kappa=a$.

 \end{proof}

\printbibliography
\end{document}